\documentclass[a4paper]{article}

\usepackage{geometry}
\usepackage{amsmath,amssymb,amsfonts,mathtools}
\usepackage{amsthm,mathrsfs}

\usepackage{indentfirst}
\usepackage{array}
\usepackage{color}
\usepackage{cite}
\usepackage{microtype}
\def\baselinestretch{1.2}
\usepackage[colorlinks=true]{hyperref}

\allowdisplaybreaks[1]
\numberwithin{equation}{section}

\theoremstyle{plain}
\newtheorem{theorem}{Theorem}[section]
\newtheorem{lemma}[theorem]{Lemma}
\newtheorem{proposition}[theorem]{Proposition}
\newtheorem{corollary}[theorem]{Corollary}

\theoremstyle{remark}

\newcommand{\HH}{\mathbb H}
\newcommand{\BB}{\mathbb B}
\newcommand{\RR}{\mathbb R}
\newcommand{\cD}{\mathcal D}
\newcommand{\cM}{\mathcal M}
\newcommand{\cR}{\mathcal R}
\newcommand{\cK}{\mathcal K}
\newcommand{\cX}{\mathcal X}
\newcommand{\dd}{\,\mathrm d}
\newcommand{\dv}{\dd v_g}
\newcommand{\ip}[2]{\left\langle #1,#2\right\rangle_\lambda}
\newcommand{\norm}[1]{\left\lVert #1\right\rVert_\lambda}
\newcommand{\dnorm}[1]{\left\lVert #1\right\rVert_{H^{-1}_\lambda}}
\newcommand{\dist}{\operatorname{dist}}
\newcommand{\Span}{\operatorname{span}}
\newcommand{\weakto}{\rightharpoonup}
\newcommand{\eps}{\varepsilon}
\newcommand{\lamc}{\lambda_c}
\newcommand{\lambdab}{\lambda_{\mathrm{bot}}}
\newcommand{\mix}{\mathrm{mix}}
\newcommand{\Euc}{\mathrm E}
\newcommand{\Hyp}{\mathrm H}

\begin{document}

\title{Sharp stability for mixed bubbles on hyperbolic space}

\author{Juncheng Wei, Zetian Yan, Zhipeng Yang\thanks{Corresponding author: yangzhipeng326@163.com.}}
\date{}
\maketitle
\AtEndDocument{%
\par
\medskip
\noindent\begin{minipage}{\textwidth}
\small
\renewcommand{\baselinestretch}{1}\selectfont
\noindent
\textbf{Juncheng Wei:}\\[0.2em]
\textsc{Department of Mathematics, Chinese University of Hong Kong, Shatin, NT, Hong Kong}\\[0.3em]
\textit{E-mail address}: \texttt{wei@math.cuhk.edu.hk}\\[0.7em]
\noindent
\textbf{Zetian Yan:}\\[0.2em]
\textsc{Department of Mathematics, Chinese University of Hong Kong, Shatin, NT, Hong Kong}\\[0.3em]
\textit{E-mail address}: \texttt{zetianyan@cuhk.edu.hk}\\[0.7em]
\noindent
\textbf{Zhipeng Yang:}\\[0.2em]
\textsc{Department of Mathematics, Yunnan Key Laboratory of Modern Analytical Mathematics and Applications, Yunnan Normal University, 650500 Kunming, China}\\
\textsc{Department of Mathematics: Analysis, Logic and Discrete Mathematics, Ghent University, 9000 Ghent, Belgium}\\[0.3em]
\textit{E-mail address}: \texttt{yangzhipeng326@163.com}\par
\end{minipage}%
}

\begin{abstract}
We prove sharp residual stability estimates for the critical Poincar\'e--Sobolev equation near sums of one positive hyperbolic solution and one projected Euclidean bubble in dimensions at least four. The estimates are uniform in the relative positions of the centers. The sharp rate is linear except in dimension six, where it contains a square root of a logarithm. For nonnegative functions, the same estimates hold under an energy constraint near the sum of the two profile energies for \(\lambda\) outside a finite set, and an additional mass constraint removes this restriction. For arbitrary finite numbers of profiles, we identify and realize all nonnegative Palais--Smale levels. We characterize the parameters at which different pairs of profile numbers have the same energy and show that the limiting energy and mass determine both profile numbers.
\end{abstract}

\par
\smallskip
\noindent \textbf{Keywords}: Critical Poincar\'e--Sobolev equation, hyperbolic space, quantitative stability.

\smallskip
\noindent \textbf{MSC2020}: 35B35, 35B33, 35J61.

\section{Introduction and main results}
\label{Sec1}

We consider the critical Poincar\'e--Sobolev equation
\[
-\Delta_{\HH^n}u-\lambda  u=|u|^{2^*-2}u \quad\text{in }\HH^n, \qquad2^*=\frac{2n}{n-2}.
\]
Here \(\HH^n\) is hyperbolic space with its canonical metric \(g\), distance \(d_g\), volume element \(\dv\), gradient \(\nabla_g\), and Laplace--Beltrami operator \(\Delta_{\HH^n}\).  Unless otherwise stated, \(n\ge4\) and \(\lambda\) are fixed, with
\begin{equation}\label{eq1.1}
\lamc=\frac{n(n-2)}4<\lambda< \lambdab=\frac{(n-1)^2}{4}.
\end{equation}
The number \(\lambdab\) is the bottom of the spectrum of \(-\Delta_{\HH^n}\).  Thus the quadratic form
\[
\norm{u}^2=\int_{\HH^n} \bigl(|\nabla_g u|_g^2-\lambda  u^2\bigr)\dv
\]
defines a norm equivalent to the usual \(H^1\)-norm.

The spectral gap controls the mass by the energy norm, as shown in Section~\ref{Sec2.1}. The associated inner product is
\[
\langle u,v\rangle_\lambda =\int_{\HH^n} \bigl(\langle\nabla_g u,\nabla_g v\rangle_g-\lambda uv\bigr)\dv.
\]
We set
\[
P^{\HH}_\lambda=-\Delta_{\HH^n}-\lambda,\qquad\cR_\lambda(u)=P^{\HH}_\lambda  u-|u|^{2^*-2}u, \qquad\Gamma_\lambda(u)=\|\cR_\lambda(u)\|_{H^{-1}_\lambda}.
\]
We study two related questions. The first asks how the equation residual controls the energy distance to a sum of one hyperbolic profile and one concentrating Euclidean profile. The second asks how energy and mass determine the types and numbers of profiles in a nonnegative sequence whose residual tends to zero. The local stability estimate answers the first question. The classification of finite profile configurations and their energy resonances answers the second, and supplies global criteria for applying the local estimate.

Mancini and Sandeep \cite{ManciniSandeep2008} proved existence and uniqueness, up to isometries, of the positive hyperbolic solution with finite energy in the range \eqref{eq1.1}.  If \(B^{\HH}_{0,\lambda}\) denotes its radial profile, centered at a fixed origin \(o\in\HH^n\), the full family is
\[
B^{\HH}_{z,\lambda}(x)=B^{\HH}_{0,\lambda}\bigl(d_g(x,z)\bigr),\qquad  z\in\HH^n.
\]
Its mass is \(M_\lambda=\|B^{\HH}_{0,\lambda}\|_{L^2(\HH^n)}^2>0\). Ganguly and Sandeep \cite{GangulySandeep2015} proved that the kernel of the linearized operator consists precisely of the infinitesimal isometry directions.  The Euclidean profiles are the Aubin--Talenti bubbles \cite{Aubin1976,Talenti1976},
\[
B^{\RR}_{y,\delta}(x)=\kappa_n \left(\frac{\delta}{\delta^2+|x-y|^2}\right)^{\frac{n-2}{2}},\qquad\kappa_n=[n(n-2)]^{(n-2)/4},
\]
where \(\delta>0\) and \(y\in\RR^n\).  We write \(B^{\RR}_{\delta}=B^{\RR}_{0,\delta}\).  By the classification of Caffarelli, Gidas, and Spruck, these are all positive solutions of \(-\Delta  v=v^{2^*-1}\) in \(\RR^n\) with finite energy \cite{CaffarelliGidasSpruck1989}.

The qualitative description of critical loss of compactness is based on the concentration-compactness principle of Lions \cite{Lions1985,Lions1985II} and Struwe's global compactness theorem \cite{Struwe1984}.  Profile decompositions of Solimini and G\'erard describe the translation and dilation parameters responsible for the failure of compactness \cite{Solimini1995,Gerard1998}, while the Brezis--Lieb lemma gives the corresponding nonlinear splitting \cite{BrezisLieb1983}.  On hyperbolic space, Bhakta and Sandeep \cite{BhaktaSandeep2012} proved that a critical Palais--Smale sequence can contain both hyperbolic translates and localized Euclidean bubbles. See also \cite{SandeepTintarev2019} for profile decompositions on noncompact manifolds.  These results identify the limiting profiles and their energies and establish convergence of the remainder in the energy norm.

Quantitative stability has two related forms.  For the Sobolev inequality, the remainder problem of Brezis and Lieb \cite{BrezisLieb1985} was resolved by Bianchi and Egnell \cite{BianchiEgnell1991}: the Sobolev deficit controls the squared distance to the family of optimizers. The stability constant and its dependence on dimension have been studied further in \cite{Konig2023,DolbeaultEstebanFigalliFrankLoss2025}. For approximate solutions of the Euler--Lagrange equation, one instead seeks an estimate in terms of its residual.  Ciraolo, Figalli, and Maggi \cite{CiraoloFigalliMaggi2018} proved such an estimate for nonnegative functions in an energy range containing one Euclidean bubble. Figalli and Glaudo \cite{FigalliGlaudo2020} obtained a linear residual estimate near sums of weakly interacting Euclidean bubbles in dimensions three to five, and showed that this rate fails in higher dimensions. Deng, Sun, and Wei \cite{DengSunWei2025} subsequently established the sharp logarithmic rate in dimension six and the sharp power rates in higher dimensions.

For the hyperbolic space, Bhakta, Ganguly, Karmakar, and Mazumdar established quantitative stability for a single profile \cite{BGKMOneBubble} and for any fixed finite number of weakly interacting hyperbolic profiles \cite{BGKMPartI}. Their multibubble theorem also includes subcritical powers. In the critical case, it gives a linear residual estimate in dimensions four and five, while such an estimate fails for sums of at least two hyperbolic bubbles in dimensions at least six. The interaction with concentrating Aubin--Talenti bubbles is left open there. Their global conclusion at pure hyperbolic energy levels excludes parameters for which the ratio of the hyperbolic and Euclidean bubble energies is rational. Compactness results for potential and forcing perturbations have also been obtained in \cite{BhaktaGangulyGuptaSahoo2024,BGGS2025}.

The interaction of a concentrating bubble with a nonzero background has also been studied in other geometries. The Brezis--Nirenberg problem and Rey's analysis of the Green function show how lower order terms and the boundary affect critical concentration \cite{BrezisNirenberg1983,Rey1990}. Chen and Kim \cite{ChenKim2024} established sharp residual stability estimates for the Yamabe problem on closed manifolds. Besides linear estimates in dimensions three to five, they studied a single bubble in higher dimensions, where the optimal rate depends on the dimension, the background solution, and the geometry. Chen, Kim, and Wei \cite{ChenKimWei2026} proved sharp quantitative stability estimates for the Brezis--Nirenberg problem, including dimension-dependent rates and boundary effects. For a single bubble whose center stays in a compact subset of the domain, their estimates include a linear rate in dimensions \(n\ge7\) and a logarithmic rate in dimension six.

We treat the interaction of a concentrating Euclidean profile with a hyperbolic profile moving over a noncompact isometry orbit. The hyperbolic profile has a kernel generated by its center parameters, and its value at the concentration center may tend to zero. We obtain sharp residual estimates in every dimension \(n\ge4\), uniformly in the relative positions of the centers, including coincidence and escape of the hyperbolic center. We also classify all finite nonnegative Palais--Smale levels, realize every pair of profile numbers, and describe all equal-energy configurations as the parameter varies. Energy and mass then select the profile numbers, including at resonant parameters.

A sum of one profile of each type is the simplest configuration in which concentration and escape can occur simultaneously.  To formulate the result, let \(B^{\HH}_{\mathrm{pr},\xi,\eps}\in  H^1_\lambda(\HH^n)\) be the Dirichlet resolvent projection of a Euclidean bubble of hyperbolic scale \(\eps>0\), centered at \(\xi\in\HH^n\), as constructed by the resolvent equation \eqref{eq2.4} and the isometry action in Section~\ref{Sec2.1}.  Its Euclidean limit after rescaling is an Aubin--Talenti bubble.  For small \(\eps_0>0\), define
\[
\cM_{\mix}(\eps_0)= \{B^{\HH}_{z,\lambda}+B^{\HH}_{\mathrm{pr},\xi,\eps}:z,\xi\in\HH^n,\ 0<\eps<\eps_0\}.
\]
The distance to a subset \(\mathcal S\subset  H^1_\lambda(\HH^n)\) is \(\dist_{H^1_\lambda}(u,\mathcal S)=\inf_{v\in\mathcal S}\norm{u-v}\). Besides this distance, our estimate controls the interaction and the signed dilation pairing of the projected bubble.  Put
\[
r=d_g(z,\xi),\qquad\cD=\eps\partial_\eps,
\]
and define
\[
I(\eps,r)=\int_{\HH^n}B^{\HH}_{z,\lambda}(B^{\HH}_{\mathrm{pr},\xi,\eps})^{2^*-1}\dv, \qquad  s(\eps)=-\bigl\langle\cR_\lambda(B^{\HH}_{\mathrm{pr},\xi,\eps}),\cD  B^{\HH}_{\mathrm{pr},\xi,\eps}\bigr\rangle.
\]
The quantity \(s(\eps)\) is the signed dilation pairing. Both quantities are invariant under a common isometry of the centers. In particular, \(s\) depends only on \(\eps\). Proposition~\ref{Prop2.10} proves that \(s(\eps)>0\) for all sufficiently small \(\eps\).

For \(0<t<1\), define
\[
\Psi_n(t)=\begin{cases}
t\sqrt{\log(2/t)},&n=6,\\
t,&n=4,5\text{ or }n\ge7,
\end{cases}
\qquad\Psi_n(0)=0.
\]

\begin{theorem}\label{Thm1.1}
Let \(n\ge4\) and assume \eqref{eq1.1}. There exist \(\eps_0,\eta_0,C>0\), depending only on \(n,\lambda\), such that every real-valued \(u\in H^1_\lambda(\HH^n)\) satisfying
\begin{equation}\label{eq1.2}
\norm{u-B^{\HH}_{z_*,\lambda}-B^{\HH}_{\mathrm{pr},\xi_*,\eps_*}}<\eta_0
\end{equation}
for some \(z_*,\xi_*\in\HH^n\) and \(0<\eps_*<\eps_0/2\) admits a decomposition
\[
u=\nu_\Hyp B^{\HH}_{z,\lambda}+\nu_\Euc B^{\HH}_{\mathrm{pr},\xi,\eps}+\rho, \qquad 0<\eps<\eps_0,\qquad \rho\perp_\lambda\cK_{(z,\eps,\xi)},
\]
where \(\cK_{(z,\eps,\xi)}\) is defined in \eqref{eq4.1}. Moreover, \(0<\Gamma_\lambda(u)<1\) and
\begin{equation}\label{eq1.3}
|\nu_\Hyp-1|+|\nu_\Euc-1|+s(\eps)+I(\eps,d_g(z,\xi)) \le C\Gamma_\lambda(u),
\end{equation}
\begin{equation}\label{eq1.4}
\norm{\rho}+\dist_{H^1_\lambda}\bigl(u,\cM_\mix(\eps_0)\bigr) \le C\Psi_n(\Gamma_\lambda(u)).
\end{equation}
\end{theorem}

Since \(s(\eps)>0\), estimate \eqref{eq1.3} excludes solutions of the equation in the neighborhood \eqref{eq1.2}, including sign-changing functions. Proposition~\ref{Prop4.8} proves optimality of the linear rate in \eqref{eq1.4} for \(n\ne6\). In dimension six, Proposition~\ref{Prop4.9} constructs positive functions for which the distance is comparable to \(\Psi_6(\Gamma_\lambda(u))\). In dimensions \(n\ge7\), the mixed configuration still has a linear rate, whereas the corresponding linear estimate fails for sums of at least two Euclidean bubbles or at least two hyperbolic bubbles \cite{DengSunWei2025,BGKMPartI}.

For the global results, choose \(\eps_0\) small enough for the preceding theorem and for
\[
\sup_{0<\eps<\eps_0}\|B^{\HH}_{\mathrm{pr},o,\eps}\|_2 <(\sqrt3-1)M_\lambda^{1/2},
\]
which is possible by Lemma~\ref{Lem2.5}.

Our second result removes the neighborhood assumption for nonnegative functions at the mixed energy.  Let \(S_{n,\lambda}\) and \(\mathcal S_{n,2}\) be the sharp constants in
\[
S_{n,\lambda}\|u\|_{2^*}^2\le\norm{u}^2, \qquad\mathcal S_{n,2}\|v\|_{2^*}^2\le\|\nabla  v\|_2^2,
\]
for \(u\in  H^1_\lambda(\HH^n)\) and \(v\in\dot  H^1(\RR^n)\), respectively.  Here \(\dot  H^1(\RR^n)\) is the completion of \(C_c^\infty(\RR^n)\) for the gradient norm.  Write
\[
J_\lambda(u)=\frac12\norm{u}^2 -\frac1{2^*}\int_{\HH^n}|u|^{2^*}\dv.
\]
A hyperbolic profile has energy \(S_{n,\lambda}^{n/2}/n\), and a Euclidean bubble has energy \(\mathcal S_{n,2}^{n/2}/n\).  The mixed energy is therefore
\[
E_{\mix}(\lambda)=\frac1n(S_{n,\lambda}^{n/2}+\mathcal S_{n,2}^{n/2}).
\]

The additional issue is whether this energy determines the types and numbers of profiles. A sequence \((u_\ell)\) is Palais--Smale at level \(c\) if \(J_\lambda(u_\ell)\to c\) and \(\Gamma_\lambda(u_\ell)\to0\). For nonnegative sequences at levels
\[
0<c<2\mathcal S_{n,2}^{n/2}/n,
\]
Theorem~\ref{Thm1.4} gives positivity, classification, and energy quantization, counting a nonzero weak limit as one hyperbolic profile. At these levels, there is at most one Euclidean profile. Since \(0<S_{n,\lambda}<\mathcal S_{n,2}\), a configuration at the mixed energy other than one profile of each type can only consist of \(m\ge3\) hyperbolic profiles, in which case
\[
(m-1)S_{n,\lambda}^{n/2}=\mathcal S_{n,2}^{n/2}.
\]
For each fixed dimension \(n\), define
\[
\Lambda_{\mix}=\left\{\lambda\in(\lamc,\lambdab): \left(\frac{S_{n,\lambda}}{\mathcal S_{n,2}}\right)^{n/2} =\frac1{m-1}\ \text{for some }m\in\mathbb N,\ m\ge3\right\}.
\]
Section~\ref{Sec5.2} shows that this set is finite.

\begin{theorem}\label{Thm1.2}
Let \(n\ge4\), assume \eqref{eq1.1}, and suppose that \(\lambda\notin\Lambda_\mix\). There exist \(\eta>0\) and \(\Gamma_0\in(0,1)\), depending only on \(n,\lambda\), such that every nonnegative \(u\in H^1_\lambda(\HH^n)\) satisfying
\[
|J_\lambda(u)-E_\mix(\lambda)|<\eta, \qquad \Gamma_\lambda(u)<\Gamma_0
\]
satisfies all conclusions of Theorem~\ref{Thm1.1}.
\end{theorem}

The global theorem at pure hyperbolic levels in \cite{BGKMPartI} assumes that \(S_{n,\lambda}^{n/2}/\mathcal S_{n,2}^{n/2}\) is irrational. The excluded parameters form a countable dense set, as shown in Theorem~\ref{Thm6.1}. At the mixed energy, only the reciprocal integers defining the finite set \(\Lambda_{\mix}\) permit a competing configuration. This exclusion is necessary for Theorem~\ref{Thm1.2}: at every \(\lambda\in\Lambda_{\mix}\), separated hyperbolic profiles have energy tending to \(E_{\mix}(\lambda)\) and residual tending to zero, while their distance from \(\cM_{\mix}(\eps_0)\) stays positive. See Lemma~\ref{Lem5.6}.

Our third result distinguishes these configurations by their mass. The mass is Lipschitz on bounded subsets of \(H^1_\lambda(\HH^n)\). In the profile decomposition, each hyperbolic profile contributes \(M_\lambda\), whereas each concentrating Euclidean profile has vanishing mass. Thus a mass constraint fixes the number of hyperbolic profiles, including at resonant parameters.

\begin{theorem}\label{Thm1.3}
Let \(n\ge4\) and assume \eqref{eq1.1}. Fix
\begin{equation}\label{eq1.5}
0<\eta_M<M_\lambda,\qquad 0<\eta_E<\frac{\mathcal S_{n,2}^{n/2}}n.
\end{equation}
There exists \(\Gamma_0\in(0,1)\), depending on \(n,\lambda,\eta_E,\eta_M\), such that every nonnegative \(u\in H^1_\lambda(\HH^n)\) satisfying
\begin{equation}\label{eq1.6}
\begin{gathered}
|J_\lambda(u)-E_{\mix}(\lambda)|<\eta_E,\qquad \left|\|u\|_{L^2(\HH^n)}^2-M_\lambda\right|<\eta_M,\\
\Gamma_\lambda(u)<\Gamma_0
\end{gathered}
\end{equation}
satisfies all conclusions of Theorem~\ref{Thm1.1}.
\end{theorem}

The tolerances in \eqref{eq1.5} are independent of the residual. In dimension six, the functions from Proposition~\ref{Prop4.9} eventually satisfy \eqref{eq1.6}, so the logarithmic rate remains optimal at resonant parameters as well.

The global classification applies to nonnegative Palais--Smale sequences at every energy level. For \(m,k\in\mathbb N_0\), define
\begin{equation}\label{eq1.7}
E_{m,k}(\lambda) =\frac{mS_{n,\lambda}^{n/2}+k\mathcal S_{n,2}^{n/2}}n.
\end{equation}
Thus \(E_{1,1}(\lambda)=E_{\mix}(\lambda)\). Our next result identifies all possible profile numbers and realizes every pair.

\begin{theorem}\label{Thm1.4}
Let \(n\ge4\) and assume \eqref{eq1.1}. If \(u_\ell\ge0\) is Palais--Smale at a level \(c\), then, after passage to a subsequence, there are \(m,k\in\mathbb N_0\) and parameters such that
\[
u_\ell=\sum_{i=1}^m B^{\HH}_{z_{i,\ell},\lambda} +\sum_{j=1}^k B^{\HH}_{\mathrm{pr},\xi_{j,\ell},\eps_{j,\ell}} +\rho_\ell, \qquad \norm{\rho_\ell}\longrightarrow0.
\]
The parameters satisfy
\[
d_g(z_{i,\ell},z_{j,\ell})\longrightarrow\infty\quad(i\ne j), \qquad \eps_{j,\ell}\longrightarrow0,
\]
\[
\frac{\eps_{i,\ell}}{\eps_{j,\ell}} +\frac{\eps_{j,\ell}}{\eps_{i,\ell}} +\frac{d_g(\xi_{i,\ell},\xi_{j,\ell})^2}{\eps_{i,\ell}\eps_{j,\ell}} \longrightarrow\infty\quad(i\ne j).
\]
Empty sums are zero, and a nonzero weak limit is counted among the hyperbolic profiles. Moreover,
\[
c=E_{m,k}(\lambda),\qquad \norm{u_\ell}^2\longrightarrow nc,\qquad \|u_\ell\|_{L^2(\HH^n)}^2\longrightarrow mM_\lambda.
\]
Conversely, every pair \((m,k)\in\mathbb N_0^2\) is realized by a nonnegative Palais--Smale sequence with weak limit zero, exactly these profile numbers, parameters satisfying the preceding separation conditions, and \(\rho_\ell=0\). Hence the nonnegative Palais--Smale levels are exactly the numbers \(E_{m,k}(\lambda)\).
\end{theorem}

Section~\ref{Sec5} proves Theorem~\ref{Thm1.4} by global compactness, positivity of the extracted profiles, and replacement of localized Euclidean profiles by projected bubbles. The decomposition allows arbitrary finite profile numbers; the sharp residual rates in Theorem~\ref{Thm1.1} concern one profile of each type.

Equality between energies associated with different profile numbers is governed by the ratio
\begin{equation}\label{eq1.8}
\mathfrak r_n(\lambda) =\left(\frac{S_{n,\lambda}}{\mathcal S_{n,2}}\right)^{n/2}.
\end{equation}
Theorem~\ref{Thm6.1} proves that the centered positive solution depends \(C^1\) on \(\lambda\) and that
\[
\frac{\mathrm d}{\mathrm d\lambda}\frac{S_{n,\lambda}^{n/2}}n =-\frac12M_\lambda, \qquad \mathfrak r_n'(\lambda) =-\frac{nM_\lambda}{2\mathcal S_{n,2}^{n/2}}<0.
\]
The ratio is a decreasing \(C^1\) diffeomorphism onto the interval in \eqref{eq6.2}.

Theorem~\ref{Thm6.1} and Proposition~\ref{Prop6.2} characterize the parameters at which two distinct profile counts have the same energy. These parameters form a countable dense subset of the admissible interval. If \(\mathfrak r_n(\lambda)=p/q\), with \(p,q\in\mathbb N\) relatively prime, two nonnegative integer pairs have the same energy precisely when their difference is an integer multiple of \((q,-p)\). The smallest energy with two representations is
\[
\frac{qS_{n,\lambda}^{n/2}}n =\frac{p\mathcal S_{n,2}^{n/2}}n,
\]
with the two pairs \((q,0)\) and \((0,p)\). Both are realized by Palais--Smale sequences. Under any fixed upper bound on this first equal-energy level, there are only finitely many such parameters in the entire interval \((\lamc,\lambdab)\). The set \(\Lambda_{\mix}\) corresponds to the representations of the particular level \(E_{1,1}(\lambda)\).

Theorem~\ref{Thm6.3} and Corollaries~\ref{Cor6.4}--\ref{Cor6.5} determine profile numbers from energy and mass. Corollary~\ref{Cor6.5} shows that fixed constraints around \(E_{m_0,k_0}(\lambda)\) and \(m_0M_\lambda\) determine any prescribed pair \((m_0,k_0)\in\mathbb N_0^2\), at every admissible parameter. Applied to \((m_0,k_0)=(1,1)\), the arbitrary-profile decomposition places the function in the neighborhood of Theorem~\ref{Thm1.1} and proves Theorem~\ref{Thm1.3} with the energy window in \eqref{eq1.5}.

The realization in Theorem~\ref{Thm1.4} and the energy classification answer the question raised in \cite{BGKMPartI} about the coexistence of the two profile types at a prescribed energy. At each fixed pure hyperbolic level \(NS_{n,\lambda}^{n/2}/n\), only finitely many parameters admit another representation. In dimensions four and five, Corollary~\ref{Cor6.6} combines this classification with the local theorem of \cite{BGKMPartI}: the corresponding global linear estimate requires only this finite exclusion under the norm constraint alone, and no exclusion after imposing a mass constraint. The union of these finite sets over all \(N\) is still countable and dense, as explained after Corollary~\ref{Cor6.6}.

For the local estimate, the Euclidean bubble must first be replaced by an element of the energy space. In the conformal ball model, it has nonzero boundary trace. The Dirichlet resolvent projection corrects this trace, and boundary Hardy estimates control the resulting residual and its parameter derivatives. The conformal quadratic form is related to the Hardy--Sobolev--Maz'ya inequalities \cite{BenguriaFrankLoss2008,TertikasTintarev2007}.

A second issue is uniformity as the hyperbolic center escapes. The interaction estimates retain the value of the hyperbolic profile at the concentration center. On the orthogonal complement of the two profiles and their parameter derivatives, energy decoupling combines the hyperbolic and Euclidean spectral gaps into a uniform coercivity estimate.

The main new estimate is a positive lower bound for the dilation pairing of the mixed configuration. It controls the projection term \(s(\eps)\) and the interaction \(I(\eps,r)\) simultaneously, although their relative size depends on the positions of the centers. With \(\mu=\lambda-\lamc\in(0,1/4)\), Propositions~\ref{Prop3.4} and~\ref{Prop2.10} give
\[
I(\eps,r)\sim_{n,\lambda}\eps^{\frac{n-2}{2}}B^{\HH}_{0,\lambda}(r), \qquad  s(\eps)\sim_{n,\lambda}
\begin{cases}
\eps^2\log(2/\eps),&n=4,\\
\eps^2,&n\ge5.
\end{cases}
\]
The hyperbolic decay rate is
\[
\gamma_\lambda=\frac{n-1+\sqrt{(n-1)^2-4\lambda}}2, \qquad  B^{\HH}_{0,\lambda}(r)\sim_{n,\lambda}e^{-\gamma_\lambda  r}.
\]
The precise expansion and its derivatives are proved in Proposition~\ref{Prop3.1}.  In dimensions four and five, either contribution can dominate, depending on the relative position of the centers. In dimension six they can have the same order, whereas for \(n\ge7\) the interaction is uniformly \(o(s(\eps))\).  We retain both terms and prove the signed estimate
\[
-\bigl\langle\cR_\lambda(B^{\HH}_{z,\lambda}+B^{\HH}_{\mathrm{pr},\xi,\eps}),\cD B^{\HH}_{\mathrm{pr},\xi,\eps} \bigr\rangle\ge  c\bigl(s(\eps)+I(\eps,r)\bigr),
\]
for all sufficiently small \(\eps\), uniformly for \(r\ge0\) (Proposition~\ref{Prop4.4}). Combining this estimate with orthogonal decomposition and coercivity gives Theorem~\ref{Thm1.1}.  Global compactness and energy quantization give Theorem~\ref{Thm1.2}. The additional mass splitting selects one profile of each type at every admissible parameter and gives Theorem~\ref{Thm1.3}.

The parameter and dimension restrictions enter at specific points of this argument. The condition \(\mu>0\) gives the positive leading projection term in the dilation pairing, while \(\mu<1/4\) makes the conformal Hardy form coercive. In dimension six, the residual of the mixed configuration has dual norm \(O(\eps^2\sqrt{\log(2/\eps)})\), while its pairings with the two profiles and the normalized scale derivative remain \(O(\eps^2)\). The proof first controls \(s+I\) and then the orthogonal remainder, which yields the logarithmic rate. For \(n\ge7\), the nonlinear interaction term is \(O(\eps^{(n+2)/4})=o(\eps^2)\) in the dual energy norm, and the residual of the projected bubble is \(O(\eps^2)\). The nonlinear remainder can be absorbed because \(2^*-1>1\).

Sections~\ref{Sec2} and~\ref{Sec3} establish the projection and interaction estimates. Section~\ref{Sec4} proves local stability and sharpness. Section~\ref{Sec5} proves the profile decomposition, realization, and global estimate away from \(\Lambda_\mix\). Section~\ref{Sec6} treats parameter dependence, energy resonance, and selection by energy and mass.

Throughout the paper, \(f\sim g\) means \(cg\le f\le Cg\) with positive constants independent of the variables being estimated. Constants may depend on the fixed pair \((n,\lambda)\); any additional dependence is indicated. A ratio tending to one is written \(f=g(1+o(1))\).

\section{Projected Euclidean bubble}
\label{Sec2}

Fix \(n\ge4\) and \(\lamc<\lambda<\lambdab\), with the notation of Section~\ref{Sec1}. The scale threshold \(\eps_0=\eps_0(n,\lambda)>0\) may be decreased finitely many times; set \(\delta_0=\eps_0/2<1/2\). All profile estimates are uniform in the center.

We denote by \(\delta\) the Euclidean coordinate scale in the Poincar\'e ball model and by \(\eps\) the corresponding hyperbolic scale. They are related by
\[
\delta=\frac{\eps}{2}.
\]
For scale differentiation and logarithmic terms, we use
\[
\cD=\eps\partial_\eps=\delta\partial_\delta, \qquad \ell_\delta=\log(1/\delta).
\]
For a scalar quantity \(F_\delta\), we write \(F_\delta=O_\cD(r_\delta)\) if
\[
|F_\delta|+|\cD F_\delta|\le Cr_\delta.
\]
The operator \(\cD\) acts with all other variables held fixed. Function-valued estimates specify their norms.

\subsection{Conformal covariance and resolvent projection}
\label{Sec2.1}

McKean's spectral-bottom inequality \cite{McKean1970} gives, for every \(u\in H^1_\lambda(\HH^n)\),
\[
\left(1-\frac{\lambda}{\lambdab}\right)\|\nabla_gu\|_2^2 \le\norm{u}^2\le\|\nabla_gu\|_2^2, \qquad \|u\|_2^2\le\frac{\norm{u}^2}{\lambdab-\lambda}.
\]
In particular, convergence in the energy norm implies convergence in \(L^2\).

Let \(\BB^n\) be the Euclidean unit ball, and write \(\BB_R\) for the Euclidean ball of radius \(R\) centered at the origin. In the Poincar\'e ball model, we fix
\[
\HH^n=(\BB^n,g),\qquad g=\varrho^2\,dx^2,\qquad \varrho(x)=\frac2{1-|x|^2}.
\]
Set
\[
\mu=\lambda-\lamc\in(0,\tfrac14),\qquad \mathcal{P}^{\RR}_\mu=-\Delta_{\RR^n}-\mu\varrho^2,\qquad \alpha_\lambda=\frac{1+\sqrt{1-4\mu}}2>\frac12.
\]
Here \(H_0^1(\BB^n)\) is the closure of \(C_c^\infty(\BB^n)\) in the Dirichlet norm \(\|\nabla v\|_2\), with zero boundary trace. For \(w,v\in H_0^1(\BB^n)\), define the symmetric bilinear form \(\mathcal B_\mu(w,v)=\int_{\BB^n}(\nabla w\cdot\nabla v-\mu\varrho^2wv)\,dx\).

Conformal covariance and integration by parts directly give
\begin{align}
P^{\HH}_\lambda(\varrho^{-\frac{n-2}{2}}v)&=\varrho^{-\frac{n+2}{2}}\mathcal{P}^{\RR}_\mu v,\label{eq2.1}\\
\norm{\varrho^{-\frac{n-2}{2}}v}^2&=\mathcal B_\mu(v,v) =\int_{\BB^n}(|\nabla v|^2-\mu\varrho^2v^2)\,dx, \qquad v\in H_0^1(\BB^n).\notag
\end{align}
The critical integrals also agree:
\begin{align*}
\int_{\HH^n}|\varrho^{-\frac{n-2}{2}}v|^{2^*}\dv=\int_{\BB^n}|v|^{2^*}\,dx.
\end{align*}

\begin{lemma}\label{Lem2.1}
For every \(v\in H_0^1(\BB^n)\),
\begin{equation}\label{eq2.2}
\frac14\int_{\BB^n}\varrho^2v^2\,dx\le\int_{\BB^n}|\nabla v|^2\,dx,
\end{equation}
and consequently
\begin{equation}\label{eq2.3}
\mathcal B_\mu(v,v)\ge(1-4\mu)\|\nabla v\|_2^2.
\end{equation}
\end{lemma}

\begin{proof}
For \(v\in C_c^\infty(\BB^n)\), put \(u=\varrho^{-(n-2)/2}v\). Conformal covariance gives
\[
0\le\int_{\HH^n}\bigl(|\nabla_gu|_g^2-\lambdab u^2\bigr)\dv =\int_{\BB^n}|\nabla v|^2\,dx-\frac14\int_{\BB^n}\varrho^2v^2\,dx.
\]
This is \eqref{eq2.2}; density and \(\mu>0\) then give \eqref{eq2.3}.
\end{proof}

The correspondence in \eqref{eq2.1} identifies the dual spaces as well as the energy spaces. Indeed, multiplication by \(\varrho^{-(n-2)/2}\) maps \(C_c^\infty(\BB^n)\) bijectively onto the compactly supported smooth hyperbolic functions. The quadratic identity, together with \eqref{eq2.3}, therefore extends this map to an onto isometry from \(H_0^1(\BB^n)\), equipped with \(\mathcal B_\mu(v,v)^{1/2}\), onto \(H^1_\lambda(\HH^n)\). Polarization gives the corresponding bilinear identity. If \(h\in H^{-1}(\BB^n)=(H_0^1(\BB^n))^*\), define its hyperbolic representative \(F\) by
\[
\bigl\langle F,\varrho^{-\frac{n-2}{2}}v\bigr\rangle =\langle h,v\rangle, \qquad v\in H_0^1(\BB^n).
\]
For an ordinary function \(h\), this definition means \(F=\varrho^{-(n+2)/2}h\), with duality on the hyperbolic side taken against \(\mathrm dv_g\). In particular, if the Euclidean dual norm is taken with respect to \(\|\nabla v\|_2\), then
\[
\|h\|_{H^{-1}(\BB^n)} \le \dnorm{F} =\sup_{v\ne0} \frac{|\langle h,v\rangle|}{\mathcal B_\mu(v,v)^{1/2}} \le (1-4\mu)^{-1/2}\|h\|_{H^{-1}(\BB^n)}.
\]
For \(q=(2^*)'=2n/(n+2)\), the same exponents give \(\|\varrho^{-(n+2)/2}h\|_{L^q(\HH^n)}=\|h\|_{L^q(\BB^n)}\).

The singular potential is also bounded as a map into the dual space. For \(w,v\in H_0^1(\BB^n)\), Hardy and Cauchy--Schwarz give
\[
\left|\int_{\BB^n}\varrho^2wv\,dx\right| \le \|\varrho w\|_2\|\varrho v\|_2 \le4\|\nabla w\|_2\|\nabla v\|_2.
\]
By the coercivity established above and the uniform \(L^{(2^*)'}\)-bound for \((B^{\RR}_{\delta})^{2^*-1}\), the Lax--Milgram theorem yields a unique solution \(B^{\RR}_{\mathrm{pr},\delta}\in H_0^1(\BB^n)\) of
\begin{equation}\label{eq2.4}
\mathcal{P}^{\RR}_\mu B^{\RR}_{\mathrm{pr},\delta} =(B^{\RR}_{\delta})^{2^*-1} \quad\text{in }\BB^n.
\end{equation}
To establish positivity, write \(f_-=\max\{-f,0\}\) and test \eqref{eq2.4} with \(-(B^{\RR}_{\mathrm{pr},\delta})_-\). Coercivity gives
\[
(1-4\mu)\int_{\BB^n} \left|\nabla(B^{\RR}_{\mathrm{pr},\delta})_-\right|^2 \le -\int_{\BB^n}(B^{\RR}_{\delta})^{2^*-1} (B^{\RR}_{\mathrm{pr},\delta})_- \le 0,
\]
and hence \(B^{\RR}_{\mathrm{pr},\delta}\ge0\). Interior elliptic regularity applies since the potential is smooth on compact subsets of \(\BB^n\). Moreover,
\[
-\Delta B^{\RR}_{\mathrm{pr},\delta} =(B^{\RR}_{\delta})^{2^*-1} +\mu\varrho^2 B^{\RR}_{\mathrm{pr},\delta}>0,
\]
so the strong maximum principle yields \(B^{\RR}_{\mathrm{pr},\delta}>0\) in \(\BB^n\).

Rotational invariance of the equation and uniqueness imply that \(B^{\RR}_{\mathrm{pr},\delta}\) is radial. Integrating its radial equation from \(0\) to \(|x|\), we obtain
\[
(B^{\RR}_{\mathrm{pr},\delta})'(|x|) =-|x|^{1-n}\int_0^{|x|} s^{n-1} \left( (B^{\RR}_{\delta}(s))^{2^*-1} +\mu\varrho(s)^2 B^{\RR}_{\mathrm{pr},\delta}(s) \right)\,ds<0, \qquad 0<|x|<1.
\]
We therefore define the corresponding hyperbolic profiles by
\[
B^{\HH}_{\mathrm{pr},0,\eps} =\varrho^{-\frac{n-2}{2}}B^{\RR}_{\mathrm{pr},\eps/2}, \qquad B^{\HH}_{\mathrm{pr},\xi,\eps} =B^{\HH}_{\mathrm{pr},0,\eps}\circ\tau_\xi^{-1},
\]
where \(\tau_\xi\) is a hyperbolic isometry sending the origin to \(\xi\). Radiality of \(B^{\HH}_{\mathrm{pr},0,\eps}\) ensures that the latter definition is independent of the choice of \(\tau_\xi\). Finally, conformal covariance gives
\[
P^{\HH}_\lambda B^{\HH}_{\mathrm{pr},0,\eps} =\varrho^{-\frac{n+2}{2}} (B^{\RR}_{\eps/2})^{2^*-1}.
\]

To compare the projected bubble with the Euclidean bubble, we separate the correction required by the Dirichlet boundary condition from that generated by the potential term. Set
\[
c_\delta=B^{\RR}_{\delta}(1),\qquad \mathring{B}^{\RR}_\delta=B^{\RR}_{\delta}-c_\delta,\qquad v_\delta=B^{\RR}_{\mathrm{pr},\delta} -\mathring{B}^{\RR}_\delta.
\]
Here \(c_\delta=\kappa_n\delta^{(n-2)/2}(1+O(\delta^2))>0\); thus \(B^{\RR}_\delta\) itself has nonzero boundary trace. Since \(B^{\RR}_{\delta}\) is constant on \(\partial\BB^n\), the function \(\mathring{B}^{\RR}_\delta\) is its Dirichlet projection for the Laplacian:
\[
-\Delta\mathring{B}^{\RR}_\delta =(B^{\RR}_{\delta})^{2^*-1} \quad\text{in }\BB^n,\qquad \mathring{B}^{\RR}_\delta=0 \quad\text{on }\partial\BB^n.
\]
In particular, both \(\mathring{B}^{\RR}_\delta\) and \(v_\delta\) belong to \(H_0^1(\BB^n)\). The source \(\varrho^2\mathring{B}^{\RR}_\delta\) behaves like a positive multiple of \((1-|x|)^{-1}\) near the boundary and is not integrable there. Nevertheless, Hardy and Cauchy--Schwarz give
\[
\int_{\BB^n}\varrho^2\mathring{B}^{\RR}_\delta|\phi|\,dx \le4\|\nabla\mathring{B}^{\RR}_\delta\|_2\|\nabla\phi\|_2, \qquad \phi\in H_0^1(\BB^n).
\]
Thus it defines an element of \(H^{-1}(\BB^n)\) by integration. Subtracting the defining equations and using positivity of the resolvent, we obtain
\begin{equation}\label{eq2.5}
\mathcal{P}^{\RR}_\mu v_\delta =\mu\varrho^2\mathring{B}^{\RR}_\delta,\qquad -\Delta v_\delta =\mu\varrho^2B^{\RR}_{\mathrm{pr},\delta},\qquad v_\delta\ge0.
\end{equation}
The full projection error therefore takes the form
\begin{equation*}
B^{\RR}_{\mathrm{pr},\delta}-B^{\RR}_{\delta} =v_\delta-c_\delta.
\end{equation*}

\subsection{Green kernel and the correction \texorpdfstring{$v_\delta$}{v-delta}}
\label{Sec2.2}

For a radial function \(u\), primes denote derivatives of its radial profile, and
\[
(\mathcal{P}^{\RR}_\mu u)(x) =-u''(|x|)-\frac{n-1}{|x|}u'(|x|) -\frac{4\mu}{(1-|x|^2)^2}u(|x|),\qquad 0<|x|<1.
\]
Radial kernels below are taken relative to the measure \(s^{n-1}\,ds\).

\begin{lemma}\label{Lem2.2}
The radial Friedrichs inverse of \(\mathcal{P}^{\RR}_\mu\) has a positive kernel \(G_\mu(|x|,s)\) with
\begin{align}
c\max\{|x|,s\}^{2-n}\le G_\mu(|x|,s)&\le C\max\{|x|,s\}^{2-n}, &&0<|x|,s\le\tfrac34,\label{eq2.6}\\
G_\mu(|x|,s)&\le C(1-s)^{\alpha_\lambda}, &&0\le|x|\le\tfrac12,\quad\tfrac34\le s<1.\label{eq2.7}
\end{align}
Let \(h\ge0\) be radial and locally integrable, and suppose that, for some \(A<\infty\),
\[
\int_{\BB^n}h|\phi|\,dx\le A\|\nabla\phi\|_2 \qquad\text{for every }\phi\in H_0^1(\BB^n).
\]
Define \(h\in H^{-1}(\BB^n)\) by \(\langle h,\phi\rangle=\int_{\BB^n}h\phi\,dx\). Then
\begin{equation}\label{eq2.8}
((\mathcal{P}^{\RR}_\mu)^{-1}h)(|x|)=\int_0^1G_\mu(|x|,s)h(s)s^{n-1}\,ds \quad\text{for almost every }x\in\BB^n.
\end{equation}
\end{lemma}

\begin{proof}
We construct the radial Green kernel $G_\mu(|x|,s)$ from two solutions of the homogeneous equation, chosen to satisfy the appropriate condition at each endpoint. Let \(u_0\) be the solution regular at the origin, normalized by
\[
\mathcal{P}^{\RR}_\mu u_0=0,\qquad u_0(0)=1,\qquad u_0'(0)=0.
\]
At \(|x|=1\), the indicial roots are \(\alpha_\lambda\) and \(1-\alpha_\lambda\). Since \(\alpha_\lambda>\tfrac12\), only the \(\alpha_\lambda\)-branch has square-integrable derivative near the boundary. We therefore choose the Friedrichs solution \(u_1\), normalized by
\[
u_1(|x|) =(1-|x|)^{\alpha_\lambda}(1+O(1-|x|)) \qquad\text{as }|x|\uparrow1.
\]

Both solutions are strictly positive on \((0,1)\). Indeed, each is positive near its normalization endpoint. If either solution had a first zero when continued into the interval, its restriction to the region between that zero and the corresponding endpoint would have zero boundary trace. Extending this restriction by zero would give a nonzero function in \(H_0^1(\BB^n)\) whose quadratic form vanishes, by testing the homogeneous equation with the restricted solution. This contradicts the coercivity estimate \eqref{eq2.3}.

The same argument shows that \(u_1\) cannot be regular at the origin, since it would then be a nonzero solution of the homogeneous equation in \(H_0^1(\BB^n)\). The local behavior of the radial equation at the origin thus gives
\[
u_1(|x|)=c_*|x|^{2-n}(1+o(1)) \qquad\text{as }|x|\to 0
\]
for some \(c_*>0\). In particular, \(u_0\) and \(u_1\) are linearly independent. Their weighted Wronskian
\[
\mathcal W_\mu =|x|^{n-1}\bigl(u_0(|x|)u_1'(|x|)-u_0'(|x|)u_1(|x|)\bigr)
\]
is constant by the radial equation, and the preceding asymptotics give \(\mathcal W_\mu=(2-n)c_*<0\). The Green kernel, with respect to the measure \(s^{n-1}\,ds\), is therefore
\[
G_\mu(|x|,s) =\frac{1}{|\mathcal W_\mu|}
\begin{cases}
u_0(|x|)u_1(s),& |x|\le s,\\
u_0(s)u_1(|x|),& s\le|x|.
\end{cases}
\]

The Wronskian also fixes the normalization of this kernel. Writing \(t=|x|\), its one-sided radial derivatives satisfy
\[
\partial_tG_\mu(s+,s)-\partial_tG_\mu(s-,s) =\frac{\mathcal W_\mu}{|\mathcal W_\mu|s^{n-1}} =-s^{1-n}.
\]
Hence integration of the radial equation across \(t=s\) gives a unit point source relative to \(s^{n-1}\,ds\).

The kernel formula gives \(G_\mu>0\). Moreover, \(u_0\) is bounded above and below by positive constants on every compact subinterval of \([0,1)\), while
\[
u_1(|x|)\sim |x|^{2-n}\quad\text{as }|x|\to0, \qquad u_1(|x|)\sim(1-|x|)^{\alpha_\lambda} \quad\text{as }|x|\to1.
\]
Substituting these bounds into the kernel formula proves \eqref{eq2.6}--\eqref{eq2.7}.

We next justify \eqref{eq2.8} in the statement. For smooth compactly supported radial \(h\), the representation follows from the construction of \(G_\mu\). Radial mollification then extends it to bounded compactly supported radial functions: the corresponding solutions converge in \(H_0^1(\BB^n)\), and the kernel integrals converge because, for each fixed \(|x|>0\), \(G_\mu(|x|,\cdot)\) is bounded on their common compact support.

For \(h\) as in the statement, define
\[
h_\ell =\min\{h,\ell\}\mathbf1_{\{|x|<1-1/\ell\}}, \qquad \ell\ge2.
\]
Then \(h_\ell\) converges to \(h\) from below almost everywhere. Moreover, for every \(\phi\in H_0^1(\BB^n)\),
\[
\left|\int_{\BB^n}h_\ell\phi\,dx\right| \le\int_{\BB^n}h|\phi|\,dx \le A\|\nabla\phi\|_2.
\]
Thus \(\|h_\ell\|_{H^{-1}}\le A\), and dominated convergence against each \(\phi\in H_0^1(\BB^n)\) gives \(h_\ell\rightharpoonup h\) in \(H^{-1}(\BB^n)\).

Set
\[
u=(\mathcal{P}^{\RR}_\mu)^{-1}h, \qquad u_\ell=(\mathcal{P}^{\RR}_\mu)^{-1}h_\ell.
\]
Positivity of the resolvent gives \(0\le u_\ell\le u_{\ell+1}\le u\). Testing the difference equation with \(u-u_\ell\), we obtain
\[
\begin{aligned}
0\le\mathcal B_\mu(u-u_\ell,u-u_\ell) =\int_{\BB^n}(h-h_\ell)(u-u_\ell)\,dx \le\int_{\BB^n}(h-h_\ell)u\,dx\longrightarrow0,
\end{aligned}
\]
where dominated convergence applies since \(\int hu<\infty\). Coercivity therefore yields \(u_\ell\to u\) strongly in \(H_0^1(\BB^n)\). Passing to an almost everywhere convergent subsequence and applying monotone convergence to the kernel representation for \(u_\ell\), we conclude that
\[
u(|x|) =\lim_{\ell\to\infty} \int_0^1G_\mu(|x|,s)h_\ell(s)s^{n-1}\,ds =\int_0^1G_\mu(|x|,s)h(s)s^{n-1}\,ds
\]
for almost every \(x\in\BB^n\), proving \eqref{eq2.8}.
\end{proof}

\begin{lemma}\label{Lem2.3}
For \(0<\delta<\delta_0\), the correction \(v_\delta\) satisfies
\begin{align}
0\le v_\delta(|x|) &\le C
\begin{cases}
\delta(1+\ell_\delta),& n=4,\\
\delta^{\frac{6-n}{2}},& n\ge5,
\end{cases}
&&0\le|x|\le\delta, \label{eq2.9}\\
0\le v_\delta(|x|) &\le C
\begin{cases}
\delta(1+\log(1/|x|)),& n=4,\\
\delta^{\frac{n-2}{2}}|x|^{4-n},& n\ge5,
\end{cases}
&&\delta\le|x|\le\tfrac12. \label{eq2.10}
\end{align}
The projected bubble is comparable to the Euclidean bubble in the interior:
\[
cB^{\RR}_{\delta}(|x|) \le B^{\RR}_{\mathrm{pr},\delta}(|x|) \le CB^{\RR}_{\delta}(|x|), \qquad 0\le|x|\le\tfrac12,
\]
and satisfies the boundary decay estimate
\begin{equation}\label{eq2.11}
B^{\RR}_{\mathrm{pr},\delta}(|x|) \le C\delta^{\frac{n-2}{2}}(1-|x|)^{\alpha_\lambda}, \qquad \tfrac12\le|x|<1.
\end{equation}
\end{lemma}

\begin{proof}
We begin with an energy estimate for \(v_\delta\), obtained by controlling \(\varrho^2\mathring{B}^{\RR}_\delta\) in the first equation of \eqref{eq2.5}. The explicit formula for \(B^{\RR}_\delta\) yields
\[
0\le\mathring{B}^{\RR}_\delta(|x|) \le C
\begin{cases}
\delta^{-\frac{n-2}{2}}, &0\le|x|\le\delta,\\
\delta^{\frac{n-2}{2}}|x|^{2-n}, &\delta\le|x|\le\tfrac34,\\
\delta^{\frac{n-2}{2}}(1-|x|), &\tfrac12\le|x|<1.
\end{cases}
\]
We estimate \(\varrho^2\mathring{B}^{\RR}_\delta\) in \(H^{-1}(\BB^n)\), taking the dual norm with respect to \(\|\nabla\phi\|_2\) on \(H_0^1(\BB^n)\).

On \(|x|<\tfrac34\), the conformal factor is bounded, so Sobolev duality reduces the estimate to an \(L^{(2^*)'}\)-bound. For \(n=4,5\), direct integration of the preceding pointwise bounds gives
\[
\|\varrho^2\mathring{B}^{\RR}_\delta\|_ {L^{(2^*)'}(\{|x|<3/4\})} \le C\delta^{\frac{n-2}{2}}.
\]
For \(n\ge7\), the integrability \(B^{\RR}_1\in L^{(2^*)'}(\RR^n)\) and scaling instead give
\[
\|\varrho^2\mathring{B}^{\RR}_\delta\|_ {L^{(2^*)'}(\{|x|<3/4\})} \le C\delta^2.
\]

Near the boundary, the last pointwise bound gives
\[
\varrho^2\mathring{B}^{\RR}_\delta \le C\delta^{\frac{n-2}{2}}(1-|x|)^{-1}.
\]
Hence, by Cauchy--Schwarz and the Hardy inequality \eqref{eq2.2}, every \(\phi\in H_0^1(\BB^n)\) satisfies
\[
\begin{aligned}
\int_{\{x\in\BB^n:\,|x|>3/4\}} \varrho^2\mathring{B}^{\RR}_\delta|\phi|\,dx\le C\delta^{\frac{n-2}{2}} \left( \int_{\BB^n}\frac{|\phi|^2}{(1-|x|)^2}\,dx \right)^{1/2}\le C\delta^{\frac{n-2}{2}}\|\nabla\phi\|_2.
\end{aligned}
\]

It remains to treat the interior contribution when \(n=6\). Direct \(L^{3/2}\)-estimation gives a factor \(\ell_\delta^{2/3}\). To obtain the sharper factor \(\ell_\delta^{1/2}\), we use the Hardy inequality at the origin,
\[
\int_{\BB^6}\frac{|\phi|^2}{|x|^2}\,dx \le\frac14\int_{\BB^6}|\nabla\phi|^2\,dx.
\]
Since in the interior \(\varrho^2\mathring{B}^{\RR}_\delta\) is bounded by \(C\delta^2(\delta+|x|)^{-4}\), weighted Cauchy--Schwarz gives
\begin{equation}\label{eq2.12}
\begin{aligned}
\int_{\{|x|<3/4\}} \frac{\delta^2|\phi|}{(\delta+|x|)^4}\,dx\le C\delta^2 \left( \int_0^{3/4}\frac{t^7}{(\delta+t)^8}\,dt \right)^{1/2}\|\nabla\phi\|_2\le C\delta^2\sqrt{\ell_\delta}\, \|\nabla\phi\|_2.
\end{aligned}
\end{equation}

Combining the interior and boundary estimates and testing the first equation in \eqref{eq2.5} with \(v_\delta\), we obtain
\[
(1-4\mu)\|\nabla v_\delta\|_2^2 \le \mathcal B_\mu(v_\delta,v_\delta) =\mu\langle\varrho^2\mathring{B}^{\RR}_\delta,v_\delta\rangle \le \mu\|\varrho^2\mathring{B}^{\RR}_\delta\|_{H^{-1}} \|\nabla v_\delta\|_2.
\]
Consequently,
\begin{equation}\label{eq2.13}
\|\nabla v_\delta\|_2 \le C
\begin{cases}
\delta^{\frac{n-2}{2}},&n=4,5,\\
\delta^2\sqrt{\ell_\delta},&n=6,\\
\delta^2,&n\ge7.
\end{cases}
\end{equation}

The preceding absolute-integral bounds for \(\mu\varrho^2\mathring{B}^{\RR}_\delta\) allow us to apply Lemma~\ref{Lem2.2} to the first equation in \eqref{eq2.5}, giving
\[
v_\delta(|x|) =\mu\int_0^1G_\mu(|x|,s)\varrho(s)^2 \mathring{B}^{\RR}_\delta(s)s^{n-1}\,ds.
\]
By interior regularity and the kernel bounds, this representation holds pointwise for \(0\le|x|\le\tfrac12\). On this region, the contribution from \(\tfrac34<s<1\) is bounded by
\[
C\delta^{\frac{n-2}{2}} \int_{3/4}^1(1-s)^{\alpha_\lambda-1}\,ds \le C\delta^{\frac{n-2}{2}},
\]
using \eqref{eq2.7} and the boundary bound for \(\varrho^2\mathring{B}^{\RR}_\delta\).

For the remaining integral, we use \eqref{eq2.6} together with the pointwise bounds for \(\mathring{B}^{\RR}_\delta\). If \(0\le|x|\le\delta\), splitting at \(s=|x|\) and \(s=\delta\) yields
\[
v_\delta(|x|) \le C\left( \delta^{\frac{6-n}{2}} +\delta^{\frac{n-2}{2}} \int_\delta^{3/4}s^{3-n}\,ds +\delta^{\frac{n-2}{2}} \right).
\]
If \(\delta\le|x|\le\tfrac12\), splitting instead at \(s=\delta\) and \(s=|x|\) gives
\[
v_\delta(|x|) \le C\left( \delta^{\frac{n+2}{2}}|x|^{2-n} +\delta^{\frac{n-2}{2}}|x|^{4-n} +\delta^{\frac{n-2}{2}} \int_{|x|}^{3/4}s^{3-n}\,ds +\delta^{\frac{n-2}{2}} \right).
\]
Evaluating these integrals proves \eqref{eq2.9}--\eqref{eq2.10}. The logarithmic terms arise precisely when \(n=4\).

We next establish the comparison with the Euclidean bubble on \(0\le|x|\le\tfrac12\). The lower bound follows from
\[
B^{\RR}_{\mathrm{pr},\delta} =\mathring{B}^{\RR}_\delta+v_\delta \ge\mathring{B}^{\RR}_\delta \ge cB^{\RR}_\delta.
\]
For the upper bound, the estimates just proved imply
\[
\frac{v_\delta(|x|)}{B^{\RR}_\delta(|x|)} \le C
\begin{cases}
(\delta+|x|)^2 \bigl(1+|\log(\delta+|x|)|\bigr),&n=4,\\
(\delta+|x|)^2,&n\ge5.
\end{cases}
\]
The right-hand side is uniformly bounded on this interval, so \(B^{\RR}_{\mathrm{pr},\delta}\le B^{\RR}_\delta+v_\delta\le CB^{\RR}_\delta\).

To obtain the boundary decay, we use the interior comparison at \(|x|=\tfrac12\), which gives
\[
B^{\RR}_{\mathrm{pr},\delta}(\tfrac12) \le C\delta^{\frac{n-2}{2}}.
\]
Direct differentiation shows that
\[
\mathcal{P}^{\RR}_\mu(1-|x|^2)^{\alpha_\lambda} =2\alpha_\lambda(n-2+2\alpha_\lambda) (1-|x|^2)^{\alpha_\lambda-1}>0.
\]
Since \((B^{\RR}_\delta)^{2^*-1} \le C\delta^{\frac{n+2}{2}}\) on \(\tfrac12<|x|<1\), a sufficiently large constant \(C\) ensures that
\[
\mathcal{P}^{\RR}_\mu \left(C\delta^{\frac{n-2}{2}} (1-|x|^2)^{\alpha_\lambda}\right) \ge (B^{\RR}_\delta)^{2^*-1}
\]
throughout this annulus, and that the same barrier dominates \(B^{\RR}_{\mathrm{pr},\delta}\) at its inner boundary. Both functions have zero outer trace, and the barrier has finite Dirichlet energy because \(\alpha_\lambda>\tfrac12\). Testing the difference inequality with the positive part of the solution minus the barrier and using \eqref{eq2.3} proves \eqref{eq2.11}.

\end{proof}

\begin{lemma}\label{Lem2.4}
Set
\[
e_\delta =B^{\RR}_{\mathrm{pr},\delta}-B^{\RR}_\delta =v_\delta-c_\delta, \qquad f_\delta=v_\delta+c_\delta.
\]
For \(0<\delta<\delta_0\) and \(j=1,2\),
\[
|\cD^jB^{\RR}_\delta|\le CB^{\RR}_\delta,\qquad |\cD^jc_\delta|\le Cc_\delta,\qquad |\cD^j\mathring{B}^{\RR}_\delta| \le C\mathring{B}^{\RR}_\delta,
\]
and
\begin{equation}\label{eq2.14}
\begin{aligned}
|\cD^jv_\delta|\le Cv_\delta,\qquad |\cD^jB^{\RR}_{\mathrm{pr},\delta}| \le CB^{\RR}_{\mathrm{pr},\delta},\qquad |e_\delta|+|\cD e_\delta|\le Cf_\delta.
\end{aligned}
\end{equation}
The scale derivatives of the correction also satisfy
\begin{equation}\label{eq2.15}
\sum_{j=0}^2\|\nabla\cD^jv_\delta\|_2 \le C
\begin{cases}
\delta^{\frac{n-2}{2}},&n=4,5,\\
\delta^2\sqrt{\ell_\delta},&n=6,\\
\delta^2,&n\ge7.
\end{cases}
\end{equation}
Finally, the nonnegative function \(f_\delta\) obeys
\begin{equation}\label{eq2.16}
f_\delta(|x|)\le C
\begin{cases}
\delta(1+\ell_\delta), &n=4,\quad 0\le|x|\le\delta,\\
\delta(1+\log(1/|x|)), &n=4,\quad \delta\le|x|\le\tfrac12,\\
\delta^{\frac{n-2}{2}}(\delta+|x|)^{4-n}, &n\ge5,\quad 0\le|x|\le\tfrac12,\\
\delta^{\frac{n-2}{2}}, &n\ge4,\quad \tfrac12\le|x|<1.
\end{cases}
\end{equation}
\end{lemma}

\begin{proof}
The derivative bounds for \(B^{\RR}_\delta\) and \(c_\delta\) follow directly from their explicit formulas. For \(\mathring{B}^{\RR}_\delta\), we use the integral identity
\[
\mathring{B}^{\RR}_\delta(|x|) =(n-2)\kappa_n\delta^{\frac{n-2}{2}} \int_{|x|}^1 \frac{s}{(\delta^2+s^2)^{n/2}}\,ds.
\]
Applying \(\cD\) once or twice to the integrand, including the prefactor, changes its absolute value by at most a constant factor. Since the integrand is nonnegative,
\[
|\cD^j\mathring{B}^{\RR}_\delta| \le C_j\mathring{B}^{\RR}_\delta, \qquad j=1,2.
\]

We next justify differentiation of the equation for \(v_\delta\). For \(\delta>0\), the map \(\delta\mapsto\mathring{B}^{\RR}_\delta\) is smooth with values in \(H_0^1(\BB^n)\). By \eqref{eq2.2}, multiplication by \(\varrho^2\) is bounded from \(H_0^1(\BB^n)\) to \(H^{-1}(\BB^n)\), while coercivity gives a bounded inverse
\[
(\mathcal P^{\RR}_\mu)^{-1}: H^{-1}(\BB^n)\longrightarrow H_0^1(\BB^n).
\]
It follows that \(v_\delta\) depends smoothly on \(\delta\) in \(H_0^1(\BB^n)\), and
\[
\mathcal P^{\RR}_\mu\cD^jv_\delta =\mu\varrho^2\cD^j\mathring{B}^{\RR}_\delta, \qquad j=1,2.
\]
Since \(C_j\mathring{B}^{\RR}_\delta \pm\cD^j\mathring{B}^{\RR}_\delta\ge0\), positivity of the resolvent yields
\[
C_jv_\delta\pm\cD^jv_\delta\ge0,
\]
and hence \(|\cD^jv_\delta|\le C_jv_\delta\). Combining these bounds with
\[
B^{\RR}_{\mathrm{pr},\delta} =\mathring{B}^{\RR}_\delta+v_\delta, \qquad e_\delta=v_\delta-c_\delta
\]
proves \eqref{eq2.14}.

For the energy estimates, coercivity and the preceding pointwise bounds give, for \(j=0,1,2\),
\[
\|\nabla\cD^jv_\delta\|_2 \le C \|\varrho^2\cD^j\mathring{B}^{\RR}_\delta\|_{H^{-1}} \le C \|\varrho^2\mathring{B}^{\RR}_\delta\|_{H^{-1}}.
\]
Inserting the bounds for \(\varrho^2\mathring{B}^{\RR}_\delta\) established in the proof of \eqref{eq2.13} yields \eqref{eq2.15}.

It remains to estimate \(f_\delta=v_\delta+c_\delta\). On \(0\le|x|\le\tfrac12\), the asserted bounds follow from \eqref{eq2.9}--\eqref{eq2.10} and \(c_\delta\le C\delta^{\frac{n-2}{2}}\). On \(\tfrac12\le|x|<1\), we use \(0\le v_\delta\le B^{\RR}_{\mathrm{pr},\delta}\) and \eqref{eq2.11} to obtain
\[
f_\delta(|x|) \le C\delta^{\frac{n-2}{2}} (1-|x|)^{\alpha_\lambda} +C\delta^{\frac{n-2}{2}} \le C\delta^{\frac{n-2}{2}}.
\]
This proves \eqref{eq2.16}.
\end{proof}
\subsection{Scale and center derivatives}
\label{Sec2.3}

Choose a hyperbolic orthonormal basis \(e_1^\xi,\ldots,e_n^\xi\) of the tangent space \(T_\xi\HH^n\). Let \(\exp_\xi:T_\xi\HH^n\to\HH^n\) be the hyperbolic exponential map. We differentiate with respect to the center along geodesics:
\[
\partial_{\xi,j}B^{\HH}_{\mathrm{pr},\xi,\eps} = \left.\frac{d}{dt}\right|_{t=0} B^{\HH}_{\mathrm{pr},\exp_\xi(te_j^\xi),\eps}.
\]
The normalized scale and center derivatives are denoted by
\[
Z_0=\cD B^{\HH}_{\mathrm{pr},\xi,\eps}, \qquad Z_j=\eps\partial_{\xi,j}B^{\HH}_{\mathrm{pr},\xi,\eps}, \quad 1\le j\le n.
\]
The factor \(\eps\) normalizes center displacements relative to the concentration scale. Since this factor also depends on \(\eps\), the product rule gives
\begin{equation}\label{eq2.17}
\cD Z_j =\cD\bigl(\eps\partial_{\xi,j} B^{\HH}_{\mathrm{pr},\xi,\eps}\bigr) =\eps\partial_{\xi,j} \bigl(\cD B^{\HH}_{\mathrm{pr},\xi,\eps}\bigr)+Z_j.
\end{equation}
For second center derivatives, \(\nabla_{\xi,i}\partial_{\xi,j} B^{\HH}_{\mathrm{pr},\xi,\eps}\) denotes the covariant Hessian with respect to \(\xi\), evaluated on \((e_i^\xi,e_j^\xi)\).

To describe the concentration limit, we first recenter the profile at the origin and pass to its Euclidean representative. For \(v\in H_0^1(\BB^n)\), we write
\[
\widehat v(y) =
\begin{cases}
\delta^{\frac{n-2}{2}}v(\delta y), &|y|<\delta^{-1},\\
0,&|y|\ge\delta^{-1},
\end{cases}
\qquad \delta=\frac{\eps}{2}.
\]
Throughout this section, hats denote this rescaling followed by zero extension. The zero-trace condition ensures that \(\widehat v\in\dot H^1(\RR^n)\).

Finally, we denote the scale derivative of the unit Euclidean bubble by
\[
Z_0^{\RR}(y) = \left.\cD B^{\RR}_\delta(y)\right|_{\delta=1} =-\frac{n-2}{2}B^{\RR}_1(y) -y\cdot\nabla B^{\RR}_1(y).
\]
Scale derivatives are taken before rescaling.

\begin{lemma}\label{Lem2.5}
Fix \(0<r_0<R_0<1\) and \(\chi\in C_c^\infty(\BB^n)\) with \(\chi=1\) on \(\BB_{r_0}\) and \(\operatorname{supp}\chi\subset\BB_{R_0}\). There is a constant \(C\), depending only on \(n,\lambda,r_0,R_0\) and \(\|\chi\|_{C^1}\), such that, for \(0<\delta<\delta_0\),
\begin{equation}\label{eq2.18}
\begin{aligned}
&\|\nabla(B^{\RR}_{\mathrm{pr},\delta} -\chi B^{\RR}_\delta)\|_{L^2(\BB^n)} +\|\nabla(\cD B^{\RR}_{\mathrm{pr},\delta} -\chi\cD B^{\RR}_\delta)\|_{L^2(\BB^n)}\le C
\begin{cases}
\delta^{\frac{n-2}{2}},&n=4,5,\\
\delta^2\sqrt{\ell_\delta},&n=6,\\
\delta^2,&n\ge7.
\end{cases}
\end{aligned}
\end{equation}
Consequently, as \(\delta\to 0\),
\begin{equation}\label{eq2.19}
\widehat{B^{\RR}_{\mathrm{pr},\delta}} \longrightarrow B^{\RR}_1, \qquad \widehat{\cD B^{\RR}_{\mathrm{pr},\delta}} \longrightarrow Z_0^{\RR},
\end{equation}
strongly in both \(\dot H^1(\RR^n)\) and \(L^{2^*}(\RR^n)\). Moreover, the hyperbolic profiles satisfy
\[
\sup_{\xi\in\HH^n} \|B^{\HH}_{\mathrm{pr},\xi,\eps}\|_{L^2(\HH^n)} \longrightarrow0 \qquad\text{as }\eps\to 0.
\]
\end{lemma}

\begin{proof}
For \(j=0,1\), the decomposition \(B^{\RR}_{\mathrm{pr},\delta} =B^{\RR}_\delta+v_\delta-c_\delta\) gives
\[
\cD^jB^{\RR}_{\mathrm{pr},\delta} -\chi\cD^jB^{\RR}_\delta =\cD^jv_\delta-\cD^jc_\delta +(1-\chi)\cD^jB^{\RR}_\delta.
\]
The term \(\cD^jc_\delta\) is constant in the spatial variable and therefore has zero gradient. Extending \(\chi\) by zero to \(\RR^n\), the explicit bubble formula and the fact that \(1-\chi\) vanishes near the origin yield
\[
\sum_{j=0}^1 \|\nabla((1-\chi)\cD^jB^{\RR}_\delta)\|_{L^2(\RR^n)} \le C\delta^{\frac{n-2}{2}}.
\]
Combining these estimates with \eqref{eq2.15} proves \eqref{eq2.18}.

Both differences in \eqref{eq2.18} belong to \(H_0^1(\BB^n)\), so their zero extensions belong to \(\dot H^1(\RR^n)\). After rescaling, the cutoff profiles satisfy
\[
\widehat{\chi B^{\RR}_\delta}(y) =\chi(\delta y)B^{\RR}_1(y), \qquad \widehat{\chi\cD B^{\RR}_\delta}(y) =\chi(\delta y)Z_0^{\RR}(y).
\]
The preceding bound for \(\nabla((1-\chi)\cD^jB^{\RR}_\delta)\), \(j=0,1\), and scale invariance of the Dirichlet norm show that these functions converge strongly in \(\dot H^1(\RR^n)\) to \(B^{\RR}_1\) and \(Z_0^{\RR}\), respectively. Applying the same rescaling to \eqref{eq2.18} proves the two limits in \eqref{eq2.19}. Their strong convergence in \(L^{2^*}(\RR^n)\) follows from the Sobolev inequality.

For the hyperbolic \(L^2\)-limit, recall that \(\delta=\eps/2\) and
\[
\|B^{\HH}_{\mathrm{pr},0,\eps}\|_{L^2(\HH^n)}^2 =\int_{\BB^n} \varrho^2(B^{\RR}_{\mathrm{pr},\delta})^2\,dx.
\]
The interior comparison in Lemma~\ref{Lem2.3} and the explicit Euclidean bubble formula give
\[
\int_{\{|x|<1/2\}} \varrho^2(B^{\RR}_{\mathrm{pr},\delta})^2\,dx \le C
\begin{cases}
\delta^2(1+\ell_\delta),&n=4,\\
\delta^2,&n\ge5.
\end{cases}
\]
Near the boundary, \eqref{eq2.11} yields
\[
\int_{\{x\in\BB^n:\,|x|>1/2\}} \varrho^2(B^{\RR}_{\mathrm{pr},\delta})^2\,dx \le C\delta^{n-2} \int_0^{1/2}t^{2\alpha_\lambda-2}\,dt \le C\delta^{n-2},
\]
where the integral is finite because \(\alpha_\lambda>\tfrac12\). Both contributions tend to zero. Since hyperbolic isometries preserve the \(L^2\)-norm, the convergence is uniform in \(\xi\).
\end{proof}

For \(h\in\BB^n\), use the M\"obius translation
\[
\tau_h(x)= \frac{(1+2h\cdot x+|x|^2)h+(1-|h|^2)x} {1+2h\cdot x+|h|^2|x|^2}, \qquad \tau_h(0)=h,\quad \tau_h^{-1}=\tau_{-h}.
\]
Define the weighted action
\[
(\mathsf C_\beta(h)f)(x) = \left( \frac{1-|h|^2}{1-2h\cdot x+|h|^2|x|^2} \right)^\beta f(\tau_{-h}(x)), \qquad \beta\in\RR.
\]
Here \(h\) is the Euclidean ball coordinate of the center.

Let \(e_1,\ldots,e_n\) denote the standard Euclidean basis. For \(1\le i,j\le n\), set
\[
\begin{aligned}
X_j(x) &=(1+|x|^2)e_j-2x_jx,\\
X_{ij}(x) &=\bigl(8x_ix_j-2(1+|x|^2)\delta_{ij}\bigr)x -2|x|^2(x_ie_j+x_je_i).
\end{aligned}
\]
Define the differential operators
\[
\mathcal L_j^{(\beta)}f =X_j\cdot\nabla f-2\beta x_jf
\]
and
\begin{align}
\mathcal K_{ij}^{(\beta)}f ={}&\nabla^2f[X_i,X_j]+X_{ij}\cdot\nabla f -2\beta x_iX_j\cdot\nabla f -2\beta x_jX_i\cdot\nabla f\notag\\
&+\bigl( 4\beta(\beta+1)x_ix_j -2\beta(1+|x|^2)\delta_{ij} \bigr)f. \label{eq2.20}
\end{align}
Here \(\nabla^2 f[X_i,X_j]\) is the Euclidean Hessian of \(f\) evaluated on \((X_i,X_j)\).

\begin{lemma}\label{Lem2.6}
For every \(h\in\BB^n\) and \(f\in C_c^\infty(\BB^n)\),
\[
\mathcal P^{\RR}_\mu \mathsf C_{\frac{n-2}{2}}(h)f = \mathsf C_{\frac{n+2}{2}}(h) \mathcal P^{\RR}_\mu f.
\]
Moreover, for every \(\beta\in\RR\) and \(1\le i,j\le n\),
\[
\left. \partial_{h_j}\bigl(\mathsf C_\beta(h)f\bigr) \right|_{h=0} =-\mathcal L_j^{(\beta)}f, \qquad \left. \partial_{h_i}\partial_{h_j} \bigl(\mathsf C_\beta(h)f\bigr) \right|_{h=0} =\mathcal K_{ij}^{(\beta)}f.
\]
\end{lemma}

\begin{proof}
The definition of \(\mathsf C_{\frac{n-2}{2}}(h)\) gives
\[
\varrho^{-\frac{n-2}{2}} \mathsf C_{\frac{n-2}{2}}(h)f = \bigl(\varrho^{-\frac{n-2}{2}}f\bigr) \circ\tau_{-h}.
\]
Since \(\tau_{-h}\) is a hyperbolic isometry, the first identity follows from the conformal covariance formula \eqref{eq2.1}.

To compute the parameter derivatives, we expand the translation and its conformal factor at \(h=0\):
\begin{align*}
\tau_{-h}(x) =x+2(h\cdot x)x-(1+|x|^2)h+\bigl( 4(h\cdot x)^2-(1+|x|^2)|h|^2 \bigr)x -2|x|^2(h\cdot x)h +O(|h|^3),
\end{align*}
and
\begin{align*}
\left( \frac{1-|h|^2}{1-2h\cdot x+|h|^2|x|^2} \right)^\beta =1+2\beta h\cdot x +2\beta(\beta+1)(h\cdot x)^2-\beta(1+|x|^2)|h|^2+O(|h|^3).
\end{align*}
Combining these expansions with the second-order Taylor formula for \(f\) yields
\[
\mathsf C_\beta(h)f = f-\sum_{j=1}^n h_j\mathcal L_j^{(\beta)}f +\frac12\sum_{i,j=1}^n h_ih_j\mathcal K_{ij}^{(\beta)}f +O(|h|^3).
\]
For fixed \(f\) and \(\beta\), the remainder is uniform in \(x\in\BB^n\). The linear and quadratic terms give the asserted derivative formulas.
\end{proof}

Let \(\widetilde{B^{\RR}_{\mathrm{pr},\delta}}\) denote zero extension to \(\RR^n\). Equation \eqref{eq2.18} and the Euclidean bubble tail give
\begin{equation}\label{eq2.21}
\left\|\nabla\left(\widetilde{B^{\RR}_{\mathrm{pr},\delta}} -B^{\RR}_\delta\right)\right\|_{L^2(\RR^n)}\longrightarrow0.
\end{equation}
The M\"obius action extends to \(\dot H^1(\RR^n)\) and preserves its norm:
\[
\|\nabla\mathsf C_{\frac{n-2}{2}}(h)f\|_2 =\|\nabla f\|_2,\qquad h\in\BB^n.
\]
Indeed, this holds for translations, orthogonal transformations, dilations, and the Kelvin transform, hence for their compositions and, by density, for every \(f\in\dot H^1\). Applying it to the difference in \eqref{eq2.21} makes that approximation uniform in \(h\), including as \(|h|\to1\). Since the transformation preserves the ball and its complement, its weighted action commutes with zero extension.

At the origin, choose \(e_j^0=\tfrac12e_j\). For \(1\le j\le n\), let \(Z_{j,\delta}\) denote the Euclidean representative \(\varrho^{(n-2)/2}Z_j\) at \(\xi=0\) and \(\eps=2\delta\). Its rescaling is denoted by \(\widehat Z_{j,\delta}\).

\begin{lemma}\label{Lem2.7}
Set
\[
V_\delta(y) =\frac{4\mu\delta^2}{(1-\delta^2|y|^2)^2}, \qquad y\in\BB_{1/\delta}.
\]
Suppose \(k_\delta,\phi_\delta\in H_0^1(\BB_{1/\delta})\) and, after zero extension to \(\RR^n\),
\[
k_\delta\longrightarrow k \quad\text{strongly in }\dot H^1(\RR^n), \qquad \sup_\delta\|\nabla\phi_\delta\|_2<\infty.
\]
Then
\begin{equation}\label{eq2.22}
\int_{\BB_{1/\delta}}V_\delta k_\delta^2\,dy \longrightarrow0, \qquad \int_{\BB_{1/\delta}}V_\delta k_\delta\phi_\delta\,dy \longrightarrow0.
\end{equation}
\end{lemma}

\begin{proof}
Choose a fixed \(\psi\in C_c^\infty(\RR^n)\) approximating \(k\) in \(\dot H^1(\RR^n)\). For sufficiently small \(\delta\), its support lies in \(\BB_{1/\delta}\), and \(k_\delta-\psi\in H_0^1(\BB_{1/\delta})\). The scaled Hardy inequality gives
\[
\left( \int_{\BB_{1/\delta}}V_\delta k_\delta^2\,dy \right)^{1/2} \le 2\sqrt\mu\,\|\nabla(k_\delta-\psi)\|_2 + \left( \int_{\BB_{1/\delta}}V_\delta\psi^2\,dy \right)^{1/2}.
\]
Because \(V_\delta\to0\) uniformly on the fixed support of \(\psi\), strong convergence of \(k_\delta\) implies
\[
\limsup_{\delta\downarrow0} \left( \int_{\BB_{1/\delta}}V_\delta k_\delta^2\,dy \right)^{1/2} \le 2\sqrt\mu\,\|\nabla(k-\psi)\|_2.
\]
Density proves the first limit. For the second, Cauchy--Schwarz and scaled Hardy yield
\[
\left| \int_{\BB_{1/\delta}} V_\delta k_\delta\phi_\delta\,dy \right| \le 2\sqrt\mu\,\|\nabla\phi_\delta\|_2 \left( \int_{\BB_{1/\delta}}V_\delta k_\delta^2\,dy \right)^{1/2} \longrightarrow0.
\]
\end{proof}

\begin{lemma}\label{Lem2.8}
The map
\[
(\eps,\xi)\longmapsto B^{\HH}_{\mathrm{pr},\xi,\eps}
\]
is strongly \(C^2\) from \((0,\infty)\times\HH^n\) into \(H^1_\lambda(\HH^n)\). For \(0<\eps<\eps_0\) and \(\xi\in\HH^n\),
\begin{equation}\label{eq2.23}
\sum_{j=0}^n\norm{Z_j}\le C
\end{equation}
and
\begin{equation}\label{eq2.24}
\begin{aligned}
&\norm{\cD^2B^{\HH}_{\mathrm{pr},\xi,\eps}} +\sum_{j=1}^n\norm{\cD Z_j} +\sum_{i,j=1}^n \norm{ \eps^2\nabla_{\xi,i}\partial_{\xi,j} B^{\HH}_{\mathrm{pr},\xi,\eps} } \le C.
\end{aligned}
\end{equation}
The rescaled center derivatives satisfy
\begin{equation}
\widehat Z_{j,\delta} =-\widehat{\, \delta\mathcal L_j^{(\frac{n-2}{2})} B^{\RR}_{\mathrm{pr},\delta}\,} \longrightarrow-\partial_jB^{\RR}_1, \qquad 1\le j\le n, \label{eq2.25}
\end{equation}
strongly in \(\dot H^1(\RR^n)\) as \(\delta=\eps/2\to 0\).
\end{lemma}

\begin{proof}
Set
\[
F_\delta=(B^{\RR}_\delta)^{2^*-1}, \qquad F_1=(B^{\RR}_1)^{2^*-1}.
\]
We first establish the strong differentiability needed for the parameter calculations. For each fixed \(\delta>0\), the map
\[
(t,h)\longmapsto \mathsf C_{\frac{n+2}{2}}(h)F_t
\]
is \(C^2\) with values in \(L^{(2^*)'}(\BB^n)\) near \((\delta,0)\). Indeed, when \(t\) stays in a compact subinterval of \((0,\infty)\) and \(h\) is sufficiently close to zero, the derivatives through order two are continuous and uniformly bounded on the closed ball. Applying the bounded resolvent
\[
(\mathcal P^{\RR}_\mu)^{-1}: H^{-1}(\BB^n)\longrightarrow H_0^1(\BB^n)
\]
therefore gives a \(C^2\) family in \(H_0^1(\BB^n)\).

The covariance identity in Lemma~\ref{Lem2.6} extends to \(H_0^1(\BB^n)\) by density. Together with uniqueness, it identifies this family as
\[
(\mathcal P^{\RR}_\mu)^{-1} \bigl(\mathsf C_{\frac{n+2}{2}}(h)F_t\bigr) = \mathsf C_{\frac{n-2}{2}}(h) B^{\RR}_{\mathrm{pr},t}.
\]
Thus the parameter derivatives below exist strongly in \(H_0^1(\BB^n)\). Differentiating this identity and using Lemma~\ref{Lem2.6}, we obtain
\begin{align*}
\mathcal P^{\RR}_\mu (\cD B^{\RR}_{\mathrm{pr},\delta}) &=\cD F_\delta,\\
\mathcal P^{\RR}_\mu (\cD^2B^{\RR}_{\mathrm{pr},\delta}) &=\cD^2F_\delta, \\
\mathcal P^{\RR}_\mu \bigl( \delta\mathcal L_j^{(\frac{n-2}{2})} B^{\RR}_{\mathrm{pr},\delta} \bigr) &= \delta\mathcal L_j^{(\frac{n+2}{2})}F_\delta, \\
\mathcal P^{\RR}_\mu\cD \bigl( \delta\mathcal L_j^{(\frac{n-2}{2})} B^{\RR}_{\mathrm{pr},\delta} \bigr) &= \delta\mathcal L_j^{(\frac{n+2}{2})} (\cD+1)F_\delta, \\
\mathcal P^{\RR}_\mu \bigl( \delta^2\mathcal K_{ij}^{(\frac{n-2}{2})} B^{\RR}_{\mathrm{pr},\delta} \bigr) &= \delta^2\mathcal K_{ij}^{(\frac{n+2}{2})}F_\delta.
\end{align*}
The mixed derivative includes the contribution from differentiating the prefactor \(\delta\):
\[
\cD\bigl( \delta\mathcal L_j^{(\frac{n-2}{2})} B^{\RR}_{\mathrm{pr},\delta} \bigr) = \delta\mathcal L_j^{(\frac{n-2}{2})} (\cD+1)B^{\RR}_{\mathrm{pr},\delta},
\]
in agreement with \eqref{eq2.17}.

We next estimate the expressions involving \(F_\delta\). The scaling formula is
\[
F_\delta(x) =\delta^{-\frac{n+2}{2}}F_1(x/\delta),
\]
and the explicit bubble formula gives
\[
|\nabla^jF_1(y)| \le C_j(1+|y|)^{-n-2-j}, \qquad j=0,1,2.
\]
For a fixed weight \(\beta\), direct substitution into \(\mathcal L_j^{(\beta)}\) yields
\begin{align}
\delta^{\frac{n+2}{2}} \bigl(\delta\mathcal L_j^{(\beta)}F_\delta\bigr) (\delta y)= \partial_jF_1(y) +\delta^2\bigl[ (|y|^2e_j-2y_jy)\cdot\nabla F_1(y) -2\beta y_jF_1(y) \bigr]. \label{eq2.26}
\end{align}
The bracketed expression belongs to \(L^{(2^*)'}(\RR^n)\). Similarly, \eqref{eq2.20} gives, for \(|y|<1/\delta\),
\begin{align*}
&\left| \delta^{\frac{n+2}{2}} \bigl(\delta^2\mathcal K_{ij}^{(\beta)}F_\delta\bigr) (\delta y) -\partial_{ij}F_1(y) \right|\\
&\qquad\le C\delta^2\bigl[ (1+|y|^2)|\nabla^2F_1(y)| +(1+|y|)|\nabla F_1(y)| +|F_1(y)| \bigr].
\end{align*}
The function in brackets also belongs to \(L^{(2^*)'}(\RR^n)\).

For the mixed derivative, we use
\[
(\cD+1)F_\delta(x) = \delta^{-\frac{n+2}{2}} \left[ -\frac n2F_1(y)-y\cdot\nabla F_1(y) \right]_{y=x/\delta}.
\]
The function \(-\frac n2F_1-y\cdot\nabla F_1\) satisfies the same decay bounds as \(F_1\) through one derivative. The expressions \(\cD F_\delta\) and \(\cD^2F_\delta\) likewise have uniformly bounded \(L^{(2^*)'}\)-norms. Taking \(\beta=\frac{n+2}{2}\), these estimates give
\begin{align}
&\|\cD F_\delta\|_{(2^*)'} +\|\cD^2F_\delta\|_{(2^*)'} +\sum_{j=1}^n \|\delta\mathcal L_j^{(\frac{n+2}{2})} F_\delta\|_{(2^*)'}\notag\\
&\quad +\sum_{j=1}^n \|\delta\mathcal L_j^{(\frac{n+2}{2})} (\cD+1)F_\delta\|_{(2^*)'} +\sum_{i,j=1}^n \|\delta^2\mathcal K_{ij}^{(\frac{n+2}{2})} F_\delta\|_{(2^*)'} \le C, \label{eq2.27}
\end{align}
where all norms in \eqref{eq2.27} are taken over \(\BB^n\).

To convert these estimates into bounds for the geodesic center derivatives, recall that the metric at the origin is \(4\,dx^2\) and its Christoffel symbols vanish there. Choosing \(e_j^0=\tfrac12e_j\), and using \(\eps=2\delta\), we obtain
\[
\begin{aligned}
\left.\varrho^{\frac{n-2}{2}}Z_j\right|_{\xi=0} &= -\delta\mathcal L_j^{(\frac{n-2}{2})} B^{\RR}_{\mathrm{pr},\delta},\\
\left. \varrho^{\frac{n-2}{2}} \eps^2\nabla_{\xi,i}\partial_{\xi,j} B^{\HH}_{\mathrm{pr},\xi,\eps} \right|_{\xi=0} &= \delta^2\mathcal K_{ij}^{(\frac{n-2}{2})} B^{\RR}_{\mathrm{pr},\delta}.
\end{aligned}
\]
Coercivity, Sobolev duality, and \eqref{eq2.27} now prove \eqref{eq2.23}--\eqref{eq2.24} at \(\xi=0\). Hyperbolic isometry invariance gives the same bounds at every center. The strong \(C^2\) dependence proved above is transported in the same way.

To prove \eqref{eq2.25}, choose \(\chi\in C_c^\infty(\BB_{1/2})\) equal to one near zero. Both \(\widehat Z_{j,\delta}\) and \(\chi(\delta\,\cdot)\partial_jB^{\RR}_1\) belong to \(H_0^1(\BB_{1/\delta})\). Equation \eqref{eq2.26}, together with \(-\Delta(\partial_jB^{\RR}_1)=\partial_jF_1\), gives
\[
\begin{aligned}
&\left\| (-\Delta-V_\delta) \bigl( \widehat Z_{j,\delta} +\chi(\delta\,\cdot)\partial_jB^{\RR}_1 \bigr) \right\|_{H^{-1}(\BB_{1/\delta})}\\
&\qquad\le C\delta^2 +C\left\| \nabla\bigl( (1-\chi(\delta\,\cdot))\partial_jB^{\RR}_1 \bigr) \right\|_{L^2(\RR^n)} +C\left\| V_\delta\chi(\delta\,\cdot)\partial_jB^{\RR}_1 \right\|_{L^{(2^*)'}(\BB_{1/\delta})}.
\end{aligned}
\]
Here the \(H^{-1}\)-norm is dual to the gradient norm on \(H_0^1(\BB_{1/\delta})\). The explicit derivative of \(B^{\RR}_1\) gives
\[
\left\| \nabla\bigl( (1-\chi(\delta\,\cdot))\partial_jB^{\RR}_1 \bigr) \right\|_{L^2(\RR^n)} \le C\delta^{n/2}.
\]
On the support of \(\chi(\delta\,\cdot)\), \(V_\delta\le C\delta^2\). Since \(\partial_jB^{\RR}_1(y)=O(|y|^{1-n})\) as \(|y|\to\infty\), direct integration yields
\[
\left\| V_\delta\chi(\delta\,\cdot)\partial_jB^{\RR}_1 \right\|_{L^{(2^*)'}(\BB_{1/\delta})} \le C
\begin{cases}
\delta^2(1+\ell_\delta)^{3/4},&n=4,\\
\delta^2,&n\ge5.
\end{cases}
\]
In dimension four the truncated \(L^{4/3}\)-norm of \(\partial_jB^{\RR}_1\) grows at most like \((1+\ell_\delta)^{3/4}\). For \(n\ge5\), \(\partial_jB^{\RR}_1\in L^{(2^*)'}(\RR^n)\).

Rescaling \eqref{eq2.3} gives
\begin{equation*}
\int_{\BB_{1/\delta}} \bigl(|\nabla\phi|^2-V_\delta\phi^2\bigr)\,dy \ge(1-4\mu)\|\nabla\phi\|_2^2, \qquad \phi\in H_0^1(\BB_{1/\delta}).
\end{equation*}
Applying this estimate to the preceding difference equation shows that
\[
\left\| \nabla\bigl( \widehat Z_{j,\delta} +\chi(\delta\,\cdot)\partial_jB^{\RR}_1 \bigr) \right\|_2 \longrightarrow0.
\]
The cutoff estimate then proves \eqref{eq2.25}.

\end{proof}

\subsection{Mass, energy, and residual estimates}
\label{Sec2.4}

For \(n\ge5\), define
\[
\mathfrak m_n =\int_{\RR^n}(B^{\RR}_1)^2\,dx =\kappa_n^2\pi^{n/2} \frac{\Gamma((n-4)/2)}{\Gamma(n-2)}.
\]
In particular,
\[
\mathfrak m_5=\frac{\kappa_5^2\pi^3}{2}, \qquad \mathfrak m_6=96\pi^3.
\]
For a domain \(\mathcal O\subset\RR^n\), we also write
\[
E_0(v;\mathcal O) =\frac12\int_{\mathcal O}|\nabla v|^2\,dx -\frac1{2^*}\int_{\mathcal O}|v|^{2^*}\,dx, \qquad E_0(v)=E_0(v;\BB^n).
\]

\begin{lemma}\label{Lem2.9}
As \(\eps\to 0\), the mass and Euclidean energy of the projected bubble satisfy
\begin{equation}\label{eq2.28}
\begin{array}{c|c|c}
n &\|B^{\HH}_{\mathrm{pr},\xi,\eps}\|_{L^2(\HH^n)}^2 &E_0(B^{\RR}_{\mathrm{pr},\eps/2}) -\dfrac{\mathcal S_{n,2}^{n/2}}{n} \\ \hline 4 &\kappa_4^2|\mathbb S^3|\eps^2\log(2/\eps) +O_\cD(\eps^2) &O_\cD(\eps^2) \\[3pt]
5 &\mathfrak m_5\eps^2+O_\cD(\eps^3) &O_\cD(\eps^3) \\[3pt]
6 &\mathfrak m_6\eps^2 +O_\cD(\eps^4\log(2/\eps)) &O_\cD(\eps^4\log(2/\eps)) \\[3pt]
n\ge7 &\mathfrak m_n\eps^2+O_\cD(\eps^4) &O_\cD(\eps^4).
\end{array}
\end{equation}
\end{lemma}

\begin{proof}
By hyperbolic isometry invariance, it suffices to work at \(\xi=0\). Set \(\delta=\eps/2\). The mass is
\[
\|B^{\HH}_{\mathrm{pr},0,\eps}\|_{L^2(\HH^n)}^2 = \int_{\BB^n} \varrho^2(B^{\RR}_{\mathrm{pr},\delta})^2\,dx.
\]
Using \(B^{\RR}_{\mathrm{pr},\delta}=B^{\RR}_\delta+e_\delta\) and the bounds in Lemma~\ref{Lem2.4}, we obtain
\begin{equation*}
\begin{aligned}
&\left| (B^{\RR}_{\mathrm{pr},\delta})^2 -(B^{\RR}_\delta)^2 \right|+ \left| \cD\bigl( (B^{\RR}_{\mathrm{pr},\delta})^2 -(B^{\RR}_\delta)^2 \bigr) \right| \le C\bigl(B^{\RR}_\delta f_\delta+f_\delta^2\bigr).
\end{aligned}
\end{equation*}
Integrating the bounds in \eqref{eq2.16} over the interior gives
\[
\int_{\{|x|<1/2\}} \bigl(B^{\RR}_\delta f_\delta+f_\delta^2\bigr)\,dx \le C
\begin{cases}
\delta^2,&n=4,\\
\delta^3,&n=5,\\
\delta^4\ell_\delta,&n=6,\\
\delta^4,&n\ge7.
\end{cases}
\]
To verify these orders for \(n\ge5\), note that \(f_\delta^2\le CB^{\RR}_\delta f_\delta\) on \(|x|\le\tfrac12\), and the integral of \(B^{\RR}_\delta f_\delta\) is bounded by
\[
C\delta^{n-2} \int_0^{1/2} \frac{t^{n-1}}{(\delta+t)^{2n-6}}\,dt.
\]

The radial integral is \(O(\delta^3)\) for \(n=5\), \(O(\delta^4\ell_\delta)\) for \(n=6\), and \(O(\delta^4)\) for \(n\ge7\). For \(n=4\), the contribution from \(\delta\le|x|\le\tfrac12\) is bounded by
\[
C\delta^2\int_\delta^{1/2} \left[ t(1+\log(1/t)) +t^3(1+\log(1/t))^2 \right]\,dt \le C\delta^2.
\]
The contribution from \(0\le|x|\le\delta\) is
\[
O\bigl( \delta^4(1+\ell_\delta) +\delta^6(1+\ell_\delta)^2 \bigr),
\]
and is therefore of smaller order.

Since \(\varrho^2=4+O(|x|^2)\) on \(|x|<\tfrac12\), replacing the conformal weight by its value at the origin introduces an error bounded by
\[
C\int_{\{|x|<1/2\}} |x|^2(B^{\RR}_\delta)^2\,dx.
\]
Direct integration gives the same orders in the four dimension ranges as above. These estimates also hold after applying \(\cD\), since
\[
|\cD B^{\RR}_\delta|\le CB^{\RR}_\delta, \qquad |\cD f_\delta|\le Cf_\delta.
\]
Near the boundary, \eqref{eq2.11} and \eqref{eq2.14} give
\[
\int_{\{x\in\BB^n:\,|x|>1/2\}} \varrho^2(B^{\RR}_{\mathrm{pr},\delta})^2\,dx =O_\cD(\delta^{n-2}),
\]
because
\[
\int_0^{1/2}t^{2\alpha_\lambda-2}\,dt<\infty.
\]

For \(n\ge5\), scaling gives
\[
\int_{\RR^n}(B^{\RR}_\delta)^2\,dx =\mathfrak m_n\delta^2,
\]
and restricting this integral to \(|x|<\tfrac12\) changes it by \(O_\cD(\delta^{n-2})\). In dimension four, the leading contribution is instead
\[
\begin{aligned}
&4\kappa_4^2|\mathbb S^3|\delta^2 \int_0^{1/2} \frac{t^3}{(\delta^2+t^2)^2}\,dt= 4\kappa_4^2|\mathbb S^3|\delta^2\ell_\delta +O_\cD(\delta^2).
\end{aligned}
\]
Combining these formulas with the preceding error bounds and substituting \(\delta=\eps/2\) proves the mass column of \eqref{eq2.28}.

We now turn to the Euclidean energy. Define the nonlinear Taylor remainder
\[
\Phi_\delta = (B^{\RR}_{\mathrm{pr},\delta})^{2^*} -(B^{\RR}_\delta)^{2^*} -2^*(B^{\RR}_\delta)^{2^*-1}e_\delta.
\]
The nonlinear remainder satisfies
\begin{equation}\label{eq2.29}
|\Phi_\delta|+|\cD\Phi_\delta| \le C\left( (B^{\RR}_\delta)^{2^*-2}f_\delta^2 +f_\delta^{2^*} \right).
\end{equation}
To prove this, first consider the region \(f_\delta<B^{\RR}_\delta/2\). There \(B^{\RR}_{\mathrm{pr},\delta}\) and \(B^{\RR}_\delta\) are comparable, and Taylor's formula, together with
\[
|\cD B^{\RR}_\delta|\le CB^{\RR}_\delta, \qquad |e_\delta|+|\cD e_\delta|\le Cf_\delta,
\]
bounds both \(\Phi_\delta\) and its scale derivative by \(C(B^{\RR}_\delta)^{2^*-2}f_\delta^2\). On the complementary region,
\[
B^{\RR}_{\mathrm{pr},\delta}+B^{\RR}_\delta \le Cf_\delta,
\]
so direct differentiation bounds every term by \(Cf_\delta^{2^*}\).

The corresponding integral estimate is
\begin{equation}\label{eq2.30}
\int_{\BB^n} \left( (B^{\RR}_\delta)^{2^*-2}f_\delta^2 +f_\delta^{2^*} \right)\,dx \le C
\begin{cases}
\delta^4(1+\ell_\delta)^3,&n=4,\\
\delta^4,&n\ge5.
\end{cases}
\end{equation}
In dimension four, the largest contribution on \(\delta\le|x|\le\tfrac12\) is bounded by
\[
C\delta^4\int_\delta^{1/2} t^{-1}(1+\log(1/t))^2\,dt \le C\delta^4(1+\ell_\delta)^3.
\]
The remaining interior terms satisfy the same bound. In dimensions five and six,
\[
(B^{\RR}_\delta)^{2^*-2}f_\delta^2 \le
\begin{cases}
C\delta^5(\delta+|x|)^{-6},&n=5,\\
C\delta^6(\delta+|x|)^{-8},&n=6,
\end{cases}
\qquad 0\le|x|\le\tfrac12.
\]
Their radial integrals are \(O(\delta^4)\), while the integrals of \(f_\delta^{2^*}\) over the same region are \(O(\delta^5)\) and \(O(\delta^6\ell_\delta)\), respectively.

For \(n\ge7\), we apply Sobolev to \(v_\delta\) and estimate the constant term separately:
\[
\|f_\delta\|_{L^{2^*}(\BB^n)} \le \|v_\delta\|_{L^{2^*}(\BB^n)} +c_\delta|\BB^n|^{1/2^*} \le C\bigl(\|\nabla v_\delta\|_2+c_\delta\bigr) \le C\delta^2.
\]
H\"older's inequality and the uniform \(L^{2^*}\)-bound for \(B^{\RR}_\delta\) then give
\[
\int_{\BB^n} (B^{\RR}_\delta)^{2^*-2}f_\delta^2\,dx \le C\delta^4, \qquad \int_{\BB^n}f_\delta^{2^*}\,dx \le C\delta^{2\cdot2^*}.
\]
Finally, on \(\tfrac12\le|x|<1\), both \(B^{\RR}_\delta\) and \(f_\delta\) are bounded by \(C\delta^{\frac{n-2}{2}}\), so this region contributes \(O(\delta^n)\) in every dimension. These estimates prove \eqref{eq2.30}.

Recall that \(e_\delta\) has trace \(-c_\delta\) and \(\nabla e_\delta=\nabla v_\delta\). Let \(\nu\) be the outward Euclidean unit normal to \(\partial\BB^n\), let \(\partial_\nu\) denote the corresponding normal derivative, and let \(d\sigma\) denote the Euclidean surface measure. Using the Euclidean bubble equation and integrating by parts, we obtain the exact identity
\begin{align}
E_0(B^{\RR}_{\mathrm{pr},\delta}) -\frac{\mathcal S_{n,2}^{n/2}}{n} ={}& -E_0(B^{\RR}_\delta;\RR^n\setminus\BB^n) +\int_{\partial\BB^n} \partial_\nu B^{\RR}_\delta\,e_\delta\,d\sigma +\frac12\int_{\BB^n}|\nabla v_\delta|^2\,dx -\frac1{2^*}\int_{\BB^n}\Phi_\delta\,dx. \label{eq2.31}
\end{align}
The exterior energy is \(O_\cD(\delta^{n-2})\). Since \(e_\delta=-c_\delta\) on \(\partial\BB^n\), the boundary term is explicitly
\[
\int_{\partial\BB^n} \partial_\nu B^{\RR}_\delta\,e_\delta\,d\sigma = (n-2)|\mathbb S^{n-1}|\kappa_n^2 \delta^{n-2}(1+\delta^2)^{-(n-1)} =O_\cD(\delta^{n-2}).
\]
The energy estimate \eqref{eq2.15} controls the squared gradient correction and its scale derivative, since
\[
\cD\left( \frac12\int_{\BB^n}|\nabla v_\delta|^2\,dx \right) = \int_{\BB^n} \nabla v_\delta\cdot\nabla\cD v_\delta\,dx.
\]
Finally, \eqref{eq2.29} and \eqref{eq2.30} control both the integral of \(\Phi_\delta\) and its scale derivative.

Combining these estimates in \eqref{eq2.31} gives the respective remainders \(O_\cD(\delta^2)\) for \(n=4\), \(O_\cD(\delta^3)\) for \(n=5\), \(O_\cD(\delta^4\ell_\delta)\) for \(n=6\), and \(O_\cD(\delta^4)\) for \(n\ge7\). Substituting \(\delta=\eps/2\) proves the energy column of \eqref{eq2.28}.
\end{proof}

\begin{proposition}\label{Prop2.10}
The hyperbolic and Euclidean energies of the projected bubble are related by
\begin{equation}\label{eq2.32}
\begin{aligned}
J_\lambda(B^{\HH}_{\mathrm{pr},\xi,\eps}) &= E_0(B^{\RR}_{\mathrm{pr},\eps/2}) -\frac\mu2 \|B^{\HH}_{\mathrm{pr},\xi,\eps}\|_{L^2(\HH^n)}^2,\\
\cD J_\lambda(B^{\HH}_{\mathrm{pr},\xi,\eps}) &= \cD E_0(B^{\RR}_{\mathrm{pr},\eps/2}) -\frac\mu2\cD \|B^{\HH}_{\mathrm{pr},\xi,\eps}\|_{L^2(\HH^n)}^2.
\end{aligned}
\end{equation}
The signed dilation pairing
\[
s(\eps) =-\langle \cR_\lambda(B^{\HH}_{\mathrm{pr},\xi,\eps}),Z_0 \rangle
\]
is independent of \(\xi\) and satisfies
\begin{equation}\label{eq2.33}
s(\eps) =
\begin{cases}
\mu\kappa_4^2|\mathbb S^3|\eps^2\log(2/\eps) +O(\eps^2),&n=4,\\
\mu\mathfrak m_5\eps^2+O(\eps^3),&n=5,\\
\mu\mathfrak m_6\eps^2 +O(\eps^4\log(2/\eps)),&n=6,\\
\mu\mathfrak m_n\eps^2+O(\eps^4),&n\ge7.
\end{cases}
\end{equation}
In particular, \(s(\eps)>0\) for \(0<\eps<\eps_0\).
\end{proposition}

\begin{proof}
The first identity in \eqref{eq2.32} follows from conformal covariance and the corresponding identity for the quadratic form. Its scale derivative is justified by the strong parameter regularity established in Lemma~\ref{Lem2.8}. The variational identity gives
\[
\cD J_\lambda(B^{\HH}_{\mathrm{pr},\xi,\eps}) = \langle \cR_\lambda(B^{\HH}_{\mathrm{pr},\xi,\eps}),Z_0 \rangle,
\]
and hence
\[
s(\eps) = -\cD E_0(B^{\RR}_{\mathrm{pr},\eps/2}) +\frac\mu2\cD \|B^{\HH}_{\mathrm{pr},\xi,\eps}\|_{L^2(\HH^n)}^2.
\]
Insert the differentiated expansions from Lemma~\ref{Lem2.9}. In dimension four,
\[
\cD\bigl(\eps^2\log(2/\eps)\bigr) =2\eps^2\log(2/\eps)-\eps^2,
\]
so differentiating the logarithm contributes to the \(O(\eps^2)\) remainder. For \(n\ge5\), differentiating \(\mathfrak m_n\eps^2\) gives \(2\mathfrak m_n\eps^2\), which combines with the coefficient \(\mu/2\). The \(O_\cD\)-bounds in Lemma~\ref{Lem2.9} give precisely the ordinary \(O\)-remainders in \eqref{eq2.33}. Since \(\mu>0\), the leading coefficient is positive in every dimension.
\end{proof}

Set
\[
\omega_n(\eps) =
\begin{cases}
\eps^2\log(2/\eps),&n=4,\\
\eps^2,&n\ge5.
\end{cases}
\]

\begin{proposition}\label{Prop2.11}
For \(0<\eps<\eps_0\),
\begin{align}
\dnorm{ \cR_\lambda(B^{\HH}_{\mathrm{pr},\xi,\eps}) } &\le C\omega_n(\eps), \label{eq2.34}\\
\dnorm{ \cD\cR_\lambda(B^{\HH}_{\mathrm{pr},\xi,\eps}) } &\le C\omega_n(\eps), \label{eq2.35}\\
c\omega_n(\eps) \le s(\eps) &\le C\dnorm{ \cR_\lambda(B^{\HH}_{\mathrm{pr},\xi,\eps}) } \le C\omega_n(\eps). \label{eq2.36}
\end{align}
Moreover,
\begin{equation}\label{eq2.37}
\begin{aligned}
\left| \norm{B^{\HH}_{\mathrm{pr},\xi,\eps}}^2 -\mathcal S_{n,2}^{n/2} \right|+ \left| \int_{\HH^n} (B^{\HH}_{\mathrm{pr},\xi,\eps})^{2^*}\dv -\mathcal S_{n,2}^{n/2} \right|+ \left| J_\lambda(B^{\HH}_{\mathrm{pr},\xi,\eps}) -\frac{\mathcal S_{n,2}^{n/2}}n \right| \le C\omega_n(\eps).
\end{aligned}
\end{equation}
\end{proposition}

\begin{proof}
By isometry invariance, it suffices to prove the estimates at \(\xi=0\), with \(\delta=\eps/2\). Conformal covariance gives
\[
\cR_\lambda(B^{\HH}_{\mathrm{pr},0,\eps}) = \varrho^{-\frac{n+2}{2}} \left( (B^{\RR}_\delta)^{2^*-1} -(B^{\RR}_{\mathrm{pr},\delta})^{2^*-1} \right).
\]
We first establish the pointwise bound
\begin{equation}\label{eq2.38}
\begin{aligned}
\left| (B^{\RR}_\delta)^{2^*-1} -(B^{\RR}_{\mathrm{pr},\delta})^{2^*-1} \right|+ \left| \cD\bigl( (B^{\RR}_\delta)^{2^*-1} -(B^{\RR}_{\mathrm{pr},\delta})^{2^*-1} \bigr) \right|\le C\left( (B^{\RR}_\delta)^{2^*-2}f_\delta +f_\delta^{2^*-1} \right).
\end{aligned}
\end{equation}
The first bound follows from the mean value theorem and \(|e_\delta|\le f_\delta\). For the scale derivative, we write
\[
\begin{aligned}
\cD\bigl( (B^{\RR}_{\mathrm{pr},\delta})^{2^*-1} -(B^{\RR}_\delta)^{2^*-1} \bigr)= (2^*-1) \left( (B^{\RR}_{\mathrm{pr},\delta})^{2^*-2} -(B^{\RR}_\delta)^{2^*-2} \right)\cD B^{\RR}_\delta+ (2^*-1)(B^{\RR}_{\mathrm{pr},\delta})^{2^*-2} \cD e_\delta.
\end{aligned}
\]
On the region \(f_\delta<B^{\RR}_\delta/2\), the two bubbles are comparable, so the mean value theorem and \eqref{eq2.14} bound this expression by \(C(B^{\RR}_\delta)^{2^*-2}f_\delta\). On the complementary region,
\[
B^{\RR}_{\mathrm{pr},\delta}+B^{\RR}_\delta \le Cf_\delta,
\]
and every term is bounded by \(Cf_\delta^{2^*-1}\). This proves \eqref{eq2.38}.

We now estimate the two terms on the right-hand side of \eqref{eq2.38}. For \(n=4\), the bounds in \eqref{eq2.16} give
\[
\begin{aligned}
\|(B^{\RR}_\delta)^{2^*-2}f_\delta \|_{L^{4/3}(\BB_{1/2})} &\le C\delta^2(1+\ell_\delta),\\
\|f_\delta^{2^*-1}\|_{L^{4/3}(\BB_{1/2})} &\le C\delta^3.
\end{aligned}
\]
For the first estimate, the region \(|x|\le\delta\) contributes \(O(\delta^2(1+\ell_\delta))\) to the norm. On \(\delta\le|x|\le\tfrac12\), its \(4/3\)-power is bounded by
\[
C\delta^4\int_\delta^{1/2} t^{-7/3}(1+\log(1/t))^{4/3}\,dt \le C\delta^{8/3}(1+\ell_\delta)^{4/3}.
\]
Taking the power \(3/4\) gives the asserted bound.

For \(n=5\), the interior pointwise bounds are
\[
(B^{\RR}_\delta)^{2^*-2}f_\delta \le\frac{C\delta^{7/2}}{(\delta+|x|)^5}, \qquad f_\delta^{2^*-1} \le\frac{C\delta^{7/2}}{(\delta+|x|)^{7/3}}, \qquad 0\le|x|\le\tfrac12.
\]
The corresponding \(L^{10/7}(\BB_{1/2})\)-norms are \(O(\delta^2)\) and \(O(\delta^{7/2})\), respectively.

For \(n=6\), we use the pointwise estimates
\[
B^{\RR}_\delta f_\delta \le\frac{C\delta^4}{(\delta+|x|)^6}, \qquad f_\delta^2 \le\frac{C\delta^4}{(\delta+|x|)^4}, \qquad 0\le|x|\le\tfrac12.
\]
Their \(L^{3/2}(\BB_{1/2})\)-norms are \(O(\delta^2)\) and \(O(\delta^4\ell_\delta^{2/3})\).

For \(n\ge7\), Sobolev embedding and \eqref{eq2.15} give
\[
\|f_\delta\|_{L^{2^*}(\BB^n)} \le C\bigl(\|\nabla v_\delta\|_2+c_\delta\bigr) \le C\delta^2.
\]
H\"older's inequality therefore yields
\[
\begin{aligned}
\|(B^{\RR}_\delta)^{2^*-2}f_\delta \|_{L^{(2^*)'}(\BB^n)} +\|f_\delta^{2^*-1}\|_{L^{(2^*)'}(\BB^n)}\le C\bigl(\delta^2+\delta^{2(2^*-1)}\bigr) \le C\delta^2.
\end{aligned}
\]
For the boundary region in every dimension, \eqref{eq2.16} and the explicit bubble formula give
\[
B^{\RR}_\delta+f_\delta \le C\delta^{\frac{n-2}{2}}, \qquad \tfrac12\le|x|<1.
\]
Its contribution to the \(L^{(2^*)'}\)-norm in \eqref{eq2.38} is consequently \(O(\delta^{\frac{n+2}{2}})\). Combining the interior and boundary estimates, and using conformal covariance, Sobolev duality, and \eqref{eq2.3}, proves \eqref{eq2.34} and \eqref{eq2.35} after substituting \(\delta=\eps/2\).

For the first inequality in \eqref{eq2.36}, Proposition~\ref{Prop2.10} gives
\[
s(\eps)\sim_{n,\lambda}\omega_n(\eps).
\]
Since \(s(\eps)>0\) at sufficiently small scales, the uniform bound for \(Z_0\) in \eqref{eq2.23} implies
\[
\begin{aligned}
s(\eps)= \left| \langle \cR_\lambda(B^{\HH}_{\mathrm{pr},\xi,\eps}), Z_0 \rangle \right|\le \dnorm{ \cR_\lambda(B^{\HH}_{\mathrm{pr},\xi,\eps}) }\norm{Z_0}\le C\dnorm{ \cR_\lambda(B^{\HH}_{\mathrm{pr},\xi,\eps}) }.
\end{aligned}
\]
Together with \eqref{eq2.34}, this proves \eqref{eq2.36}.

Finally, coercivity applied to \eqref{eq2.4} gives
\[
\norm{B^{\HH}_{\mathrm{pr},\xi,\eps}}\le C.
\]
Testing the residual with the projected bubble and using \eqref{eq2.34}, we obtain
\[
\begin{aligned}
&\left| \norm{B^{\HH}_{\mathrm{pr},\xi,\eps}}^2 -\int_{\HH^n} (B^{\HH}_{\mathrm{pr},\xi,\eps})^{2^*}\dv \right|= \left| \langle \cR_\lambda(B^{\HH}_{\mathrm{pr},\xi,\eps}), B^{\HH}_{\mathrm{pr},\xi,\eps} \rangle \right| \le C\omega_n(\eps).
\end{aligned}
\]
Lemma~\ref{Lem2.9} and \eqref{eq2.32} also give
\[
\left| J_\lambda(B^{\HH}_{\mathrm{pr},\xi,\eps}) -\frac{\mathcal S_{n,2}^{n/2}}n \right| \le C\omega_n(\eps).
\]
The exact identity
\[
\begin{aligned}
J_\lambda(B^{\HH}_{\mathrm{pr},\xi,\eps})= \frac1n\norm{B^{\HH}_{\mathrm{pr},\xi,\eps}}^2+ \frac1{2^*} \langle \cR_\lambda(B^{\HH}_{\mathrm{pr},\xi,\eps}), B^{\HH}_{\mathrm{pr},\xi,\eps} \rangle
\end{aligned}
\]
therefore yields
\[
\left| \norm{B^{\HH}_{\mathrm{pr},\xi,\eps}}^2 -\mathcal S_{n,2}^{n/2} \right| \le C\omega_n(\eps).
\]
The identity obtained by testing the residual with the projected bubble then gives the estimate for the critical integral, completing the proof of \eqref{eq2.37}.
\end{proof}
\section{Hyperbolic decay and mixed interactions}
\label{Sec3}

We prove decay and parameter bounds for the bilinear interaction \(Q\), compare it with the nonlinear interaction \(I\), and estimate the mixed residual in the dual energy norm, uniformly in the centers.

We retain the scale notation from Section~\ref{Sec2}:
\[
\delta=\eps/2,\qquad \cD=\eps\partial_\eps=\delta\partial_\delta,\qquad \ell_\delta=\log(1/\delta).
\]
The operator \(\cD\) acts with the centers, and hence their distance \(r\) in \eqref{eq3.8}, held fixed. The notation \(O_\cD\) controls a remainder and its first normalized scale derivative, as specified in Section~\ref{Sec2}.

\subsection{Hyperbolic decay and bilinear interactions}
\label{Sec3.1}

The parameters satisfy
\begin{equation*}
\gamma_\lambda-\frac{n-2}{2} =\frac{1+\sqrt{1-4\mu}}2=\alpha_\lambda.
\end{equation*}

\begin{proposition}\label{Prop3.1}
There exists \(\mathsf A_\lambda\in(0,\infty)\) such that, for every integer \(j\geq0\), there are \(r_0>0\) and \(C_{j,\lambda}>0\) for which
\begin{equation}\label{eq3.1}
\left| \partial_r^j\left( B^{\HH}_{0,\lambda}(r)-\mathsf A_\lambda e^{-\gamma_\lambda r} \right)\right| \leq C_{j,\lambda}e^{-(\gamma_\lambda+2)r}, \qquad r\geq r_0.
\end{equation}
In particular,
\begin{equation}\label{eq3.2}
B^{\HH}_{0,\lambda}(r)=\mathsf A_\lambda e^{-\gamma_\lambda r} \bigl(1+O(e^{-2r})\bigr),\qquad\frac{(B^{\HH}_{0,\lambda})'(r)}{B^{\HH}_{0,\lambda}(r)} =-\gamma_\lambda+O(e^{-2r}).
\end{equation}
Globally, for \(r\geq0\),
\begin{align}
c e^{-\gamma_\lambda r} &\leq B^{\HH}_{0,\lambda}(r)\leq C e^{-\gamma_\lambda r}, \label{eq3.3}\\
c\min\{r,1\}B^{\HH}_{0,\lambda}(r) &\leq-(B^{\HH}_{0,\lambda})'(r) \leq C\min\{r,1\}B^{\HH}_{0,\lambda}(r), \label{eq3.4}\\
|(B^{\HH}_{0,\lambda})^{(k)}(r)| &\leq C_kB^{\HH}_{0,\lambda}(r),\qquad k\geq0. \label{eq3.5}
\end{align}
\end{proposition}

\begin{proof}
Write \(B=B^{\HH}_{0,\lambda}\), \(m=(n-1)/2\), \(\nu=\sqrt{\lambdab-\lambda}>0\), and \(v=(\sinh r)^mB\), so \(\gamma_\lambda=m+\nu\). By \cite[Lemma~3.4]{ManciniSandeep2008}, \(r^{-1}\log B\to-\gamma_\lambda\) and \(B'/B\to-\gamma_\lambda\). Choose \(0<\vartheta<\min\{\nu,2,\gamma_\lambda-(n-2)/2\}\). Then \(v+|v'|=O(e^{-(\nu-\vartheta)r})\), while
\[
v''-\nu^2v=Vv,\qquad V(r)=\frac{(n-1)(n-3)}4\operatorname{csch}^2r-B^{2^*-2}, \qquad V^{(j)}=O_j(e^{-2r})\quad(j\ge0).
\]
The last bounds follow inductively from \(w'=-w^2-(n-1)\coth r\,w-\lambda-B^{2^*-2}\), where \(w=B'/B\) is bounded, and \((2^*-2)(\gamma_\lambda-\vartheta)>2\). Putting \(y=v'+\nu v\) and integrating \(y'-\nu y=Vv\) from infinity gives
\[
y(r)=-e^{\nu r}\int_r^\infty e^{-\nu s}V(s)v(s)\,ds =O(e^{-(\nu+2-\vartheta)r}).
\]
Thus \((e^{\nu r}v)'=O(e^{-(2-\vartheta)r})\), and \(v=ce^{-\nu r}+O(e^{-(\nu+2-\vartheta)r})\) for some \(c\ge0\). If \(c=0\), this contradicts \(r^{-1}\log v\to-\nu\); hence \(c>0\). Reinserting \(v=O(e^{-\nu r})\) into the integral sharpens the bounds to \(y=O(e^{-(\nu+2)r})\) and \(v-ce^{-\nu r}=O(e^{-(\nu+2)r})\). The formula for \(y\) gives the same bound for the first derivative of this remainder; the equation \(v''=\nu^2v+Vv\) and the bounds on \(V^{(j)}\) give every higher derivative inductively. Since \((\sinh r)^{-m}=2^me^{-mr}(1+O(e^{-2r}))\), also after differentiation, \eqref{eq3.1}--\eqref{eq3.2} follow with \(\mathsf A_\lambda=2^mc>0\).

Finally, the radial equation gives \(B'(r)<0\) for \(r>0\), \(B'(0)=0\), and \(B''(0)=-(\lambda B(0)+B(0)^{2^*-1})/n<0\). These facts, smoothness at the origin, and the preceding asymptotics give \eqref{eq3.3}, \eqref{eq3.4}, and \eqref{eq3.5} by compactness on intermediate intervals.
\end{proof}

Set
\[
\mathcal C=\{x\in\BB^n:1/2<|x|<1\}.
\]
For \(\zeta\in\mathbb S^{n-1}\), let \(z_{r,\zeta}=\tanh(r/2)\zeta\) be the point at distance \(r\) from \(0\) on the unit-speed geodesic ray with ideal endpoint \(\zeta\), and define
\[
h_{r,\zeta}(x)=\varrho(x)^{\frac{n-2}{2}} B^{\HH}_{0,\lambda}\bigl(d_g(x,z_{r,\zeta})\bigr).
\]

\begin{lemma}\label{Lem3.2}
There is a constant \(C=C(n,\lambda)\) such that
\begin{equation}\label{eq3.6}
\sup_{\zeta\in\mathbb S^{n-1}}\int_{\mathcal C} \frac{(1-|x|^2)^{\alpha_\lambda}}{|x-\zeta|^{2\gamma_\lambda}}\,dx\le C.
\end{equation}
Moreover, for every \(r\ge0\), \(x\in\BB^n\), \(\zeta\in\mathbb S^{n-1}\), and \(j=0,1,2\),
\begin{equation}\label{eq3.7}
|\partial_r^j h_{r,\zeta}(x)| \le C B^{\HH}_{0,\lambda}(r) \frac{(1-|x|^2)^{\alpha_\lambda}} {(e^{-r}+|x-\zeta|)^{2\gamma_\lambda}}.
\end{equation}
\end{lemma}

\begin{proof}
Since \((1-|x|)\le|x-\zeta|\) and \(1-|x|^2\le2(1-|x|)\), polar coordinates centered at \(\zeta\) give
\[
\int_{\mathcal C} \frac{(1-|x|^2)^{\alpha_\lambda}}{|x-\zeta|^{2\gamma_\lambda}}\,dx \le C\int_0^2 t^{n-1+\alpha_\lambda-2\gamma_\lambda}\,dt =C\int_0^2 t^{1-\alpha_\lambda}\,dt<\infty.
\]
Here \(n+\alpha_\lambda-2\gamma_\lambda=2-\alpha_\lambda>0\), so the bound is uniform in \(\zeta\). Multiplication by any uniformly bounded weight on \(\mathcal C\) preserves this estimate.

For the pointwise estimate, the ball distance formula implies
\begin{align*}
e^{-d_g(x,z_{r,\zeta})}\le \frac{(1-|x|^2)(1-\left(\tanh(r/2)\right)^2)} {1-2\tanh(r/2)x\cdot\zeta+(\tanh(r/2))^2|x|^2}\le C\frac{e^{-r}(1-|x|^2)} {(e^{-r}+|x-\zeta|)^2}.
\end{align*}
Indeed, \(1-\tanh(r/2)\sim e^{-r}\), \(1-(\tanh(r/2))^2\sim e^{-r}\), and
\[
\begin{aligned}
1&-2\tanh(r/2)x\cdot\zeta+(\tanh(r/2))^2|x|^2\\
&=(1-\tanh(r/2))^2 +2\tanh(r/2)(1-\tanh(r/2))(1-x\cdot\zeta)+(\tanh(r/2))^2|x-\zeta|^2.
\end{aligned}
\]
This is bounded below by a constant times \((e^{-r}+|x-\zeta|)^2\) when \(\tanh(r/2)\ge1/2\). When \(\tanh(r/2)\le1/2\), the same conclusion follows from \(1-2\tanh(r/2)x\cdot\zeta+(\tanh(r/2))^2|x|^2\ge(1-\tanh(r/2)|x|)^2\ge1/4\). Combining this with \eqref{eq3.3} and \(\gamma_\lambda-\frac{n-2}{2}=\alpha_\lambda\) proves \eqref{eq3.7} for \(j=0\).

For the derivatives, the center curve is a unit-speed geodesic. At distance \(d\) from \(x\), the gradient norm of the radial profile is \(|(B^{\HH}_{0,\lambda})'(d)|\), and its Hessian has eigenvalues
\[
(B^{\HH}_{0,\lambda})''(d),\qquad (B^{\HH}_{0,\lambda})'(d)\coth d.
\]

At \(d=0\), smooth radiality gives, with \(\beta_2=(B^{\HH}_{0,\lambda})''(0)\),
\[
(B^{\HH}_{0,\lambda})'(d)=\beta_2d+O(d^3),\qquad (B^{\HH}_{0,\lambda})''(d)=\beta_2+O(d^2),\qquad (B^{\HH}_{0,\lambda})'(d)\coth d=\beta_2+O(d^2).
\]
The two Hessian eigenvalues therefore extend continuously with the same value at the center. They are bounded by \(CB^{\HH}_{0,\lambda}(d)\) near zero by smoothness and away from zero by \eqref{eq3.5}. The first and second derivatives along the center geodesic are the gradient and Hessian evaluated on its unit tangent. These bounds apply also when \(x=z_{r,\zeta}\). The distance estimate consequently proves \eqref{eq3.7} for \(j=1,2\).
\end{proof}

After a common isometry, we may assume that \(\xi=0\) and \(z_r=\tanh(r/2)e_1\), where \(e_1\) is the first Euclidean coordinate vector, and write
\[
h_r=\varrho^{\frac{n-2}{2}}B^{\HH}_{z_r,\lambda}=h_{r,e_1}.
\]
For \(z,\xi\in\HH^n\), let
\begin{align}\label{eq3.8}
r=d_g(z,\xi),\qquad Q(\eps,r) =\langle B^{\HH}_{z,\lambda},B^{\HH}_{\mathrm{pr},\xi,\eps}\rangle_\lambda.
\end{align}
Self-adjointness and conformal covariance give the exact representation
\[
Q(\eps,r)=\int_{\BB^n}h_r(B^{\RR}_{\delta})^{2^*-1}\,dx.
\]

The weak equations for the two conformal representatives give
\[
\mathcal B_\mu(h_r,B^{\RR}_{\mathrm{pr},\delta}) =\int_{\BB^n}h_r(B^{\RR}_\delta)^{2^*-1}\,dx =\int_{\BB^n}h_r^{2^*-1}B^{\RR}_{\mathrm{pr},\delta}\,dx.
\]
Both conformal representatives belong to \(H_0^1(\BB^n)\), and \((B^{\RR}_\delta)^{2^*-1}\in L^{(2^*)'}(\BB^n)\). Sobolev duality therefore justifies the pairings. The explicit Euclidean bubble enters through its source. In the scalar interaction \(I\) from \eqref{eq3.13}, the corresponding source is \((B^{\RR}_{\mathrm{pr},\delta})^{2^*-1}\), so the comparison of \(I\) and \(Q\) reduces to the difference of these two powers.

Define the interaction error by
\[
\mathcal E_Q(\eps,r) =Q(\eps,r)-(n-2)|\mathbb S^{n-1}|\kappa_n\eps^{\frac{n-2}{2}}B^{\HH}_{0,\lambda}(r).
\]

\begin{theorem}\label{Thm3.3}
As \(\eps\to0\), uniformly for every \(r\ge0\),
\[
Q(\eps,r)=(n-2)|\mathbb S^{n-1}|\kappa_n \eps^{\frac{n-2}{2}}B^{\HH}_{0,\lambda}(r) \left(1+O_\cD\!\left(\eps^2\log\frac2\eps\right)\right).
\]
More precisely,
\begin{equation}\label{eq3.9}
|\mathcal E_Q(\eps,r)|+|\cD\mathcal E_Q(\eps,r)| \le C\eps^{\frac{n+2}{2}}\log\frac2\eps B^{\HH}_{0,\lambda}(r).
\end{equation}
The parameter derivatives also satisfy
\begin{equation}\label{eq3.10}
|\partial_r^j\cD^k Q(\eps,r)| \le C\eps^{\frac{n-2}{2}}B^{\HH}_{0,\lambda}(r), \qquad 0\le j\le2,\quad 0\le k\le1,
\end{equation}
and \(\partial_r\cD^kQ(\eps,0)=0\) for \(k=0,1\).
\end{theorem}

\begin{proof}
Write \(S_\delta=(B^{\RR}_\delta)^{2^*-1}\). The explicit bubble gives
\[
\int_{\RR^n}S_\delta\,dx =(n-2)|\mathbb S^{n-1}|\kappa_n\delta^{\frac{n-2}{2}}, \qquad |\cD^kS_\delta|\le CS_\delta,\quad k=0,1.
\]
The mass identity follows by integrating \(\frac{d}{dt}[t^n(1+t^2)^{-n/2}] =nt^{n-1}(1+t^2)^{-(n+2)/2}\) and using \(\kappa_n^{4/(n-2)}=n(n-2)\). The tail and second moment satisfy
\begin{align}
\left|\int_{\BB^n}\cD^kS_\delta\,dx -\cD^k\bigl((n-2)|\mathbb S^{n-1}|\kappa_n \delta^{\frac{n-2}{2}}\bigr)\right| &\le C\delta^{\frac{n+2}{2}},\label{eq3.11}\\
\int_{|x|<1/2}|x|^2|\cD^kS_\delta|\,dx &\le C\delta^{\frac{n+2}{2}}(1+\ell_\delta), \qquad k=0,1.\label{eq3.12}
\end{align}
Indeed, after \(|x|=\delta t\), the radial integrands for these two estimates decay respectively as \(t^{-3}\) and \(t^{-1}\).

We need the interior Hessian bound
\[
\sup_{|x|\le1/2}|\nabla_x^2h_r(x)| \le CB^{\HH}_{0,\lambda}(r).
\]
For \(0\le r\le2\), this follows from smoothness of the translated radial solution and positivity of \(B^{\HH}_{0,\lambda}\). For \(r\ge2\), put \(d=d_g(x,z_r)\). Then \(|d-r|\le\log3\), and
\[
\cosh d=\frac{1+|x|^2}{1-|x|^2}\cosh r -\frac{2x_1}{1-|x|^2}\sinh r.
\]
On \(|x|\le1/2\), this formula and \(d\ge2-\log3>0\) give uniform bounds for \(\nabla_xd\) and \(\nabla_x^2d\). The chain rule and Proposition~\ref{Prop3.1}, including \(B^{\HH}_{0,\lambda}(d)\le CB^{\HH}_{0,\lambda}(r)\), prove the Hessian bound.

Since \(\cD^kS_\delta\) is radial, its first moment on \(\BB^n\) vanishes. Taylor's formula therefore estimates the interior part of
\[
\int_{\BB^n} \bigl(h_r(x)-h_r(0)-\nabla h_r(0)\cdot x\bigr) \cD^kS_\delta(x)\,dx
\]
by \(CB^{\HH}_{0,\lambda}(r)\delta^{(n+2)/2}(1+\ell_\delta)\). On \(\mathcal C\), \(|\cD^kS_\delta|\le C\delta^{(n+2)/2}\), while Lemma~\ref{Lem3.2} controls the integral of \(h_r\) by \(CB^{\HH}_{0,\lambda}(r)\). The same bound holds for \(|h_r(0)|+|\nabla h_r(0)|\), by Proposition~\ref{Prop3.1}. Thus the annular contribution is at most \(CB^{\HH}_{0,\lambda}(r)\delta^{(n+2)/2}\). Combining \eqref{eq3.11}--\eqref{eq3.12} with \(h_r(0)=2^{(n-2)/2}B^{\HH}_{0,\lambda}(r)\) proves \eqref{eq3.9}.

For \eqref{eq3.10}, differentiate the exact representation:
\[
\partial_r^j\cD^kQ =\int_{\BB^n}\partial_r^jh_r\,\cD^kS_\delta\,dx.
\]
Lemma~\ref{Lem3.2} bounds \(|\partial_r^jh_r|\) by \(CB^{\HH}_{0,\lambda}(r)\) on \(|x|\le1/2\), and its annular integral by the same quantity, for \(j\le2\). The source mass and annular bounds above give \eqref{eq3.10}. These bounds also supply integrable majorants locally in the parameters, justifying each differentiation. Finally, reflection in \(x_1=0\) shows that \(Q(\eps,r)\) extends smoothly and evenly to negative \(r\), so \(\partial_r\cD^kQ(\eps,0)=0\).
\end{proof}

\subsection{Nonlinear interactions and residual estimates}
\label{Sec3.2}

Define the nonlinear scalar interaction by
\begin{align}\label{eq3.13}
I(\eps,r)=\int_{\HH^n}B^{\HH}_{z,\lambda} (B^{\HH}_{\mathrm{pr},\xi,\eps})^{2^*-1}\dv.
\end{align}
Self-adjointness gives
\[
Q(\eps,r)=\int_{\HH^n}(B^{\HH}_{z,\lambda})^{2^*-1} B^{\HH}_{\mathrm{pr},\xi,\eps}\dv,
\]
whereas the mixed energy expansion also involves \(I(\eps,r)\). In the common conformal coordinates,
\[
I-Q=\int_{\BB^n}h_r \left[(B^{\RR}_{\mathrm{pr},\delta})^{2^*-1} -(B^{\RR}_{\delta})^{2^*-1}\right]\,dx.
\]
The bound \eqref{eq2.38} for the projection error \(e_\delta=B^{\RR}_{\mathrm{pr},\delta}-B^{\RR}_\delta\) gives the following comparison without a finer expansion of that error.

\begin{proposition}\label{Prop3.4}
For every fixed \(n\ge4\) and \(\lamc<\lambda<\lambdab\), set \(a_4=2\) and \(a_n=1\) for \(n\ge5\). Uniformly for \(r\ge0\), as \(\eps\to0\),
\begin{equation}\label{eq3.14}
|I-Q|+|\cD(I-Q)| \le CB^{\HH}_{0,\lambda}(r)\eps^{\frac{n+2}{2}} \log^{a_n}\frac2\eps.
\end{equation}
Consequently,
\begin{align}
\frac IQ&=1+O_\cD\!\left(\eps^2\log^{a_n}\frac2\eps\right), \label{eq3.15}\\
\frac{\cD I}{I}&=\frac{n-2}{2} +O\!\left(\eps^2\log^{a_n}\frac2\eps\right). \label{eq3.16}
\end{align}
In particular, for all sufficiently small \(\eps\), uniformly in \(r\),
\begin{equation}\label{eq3.17}
I(\eps,r)\sim_{n,\lambda}Q(\eps,r) \sim_{n,\lambda}\eps^{\frac{n-2}{2}}B^{\HH}_{0,\lambda}(r), \qquad \cD I(\eps,r)\ge c_{n,\lambda}I(\eps,r).
\end{equation}
When \(n=6\), the mass \(\mathfrak m_6\) from Section~\ref{Sec2.4} gives
\begin{equation}\label{eq3.18}
I(\eps,r)=\mathfrak m_6B^{\HH}_{0,\lambda}(r)\eps^2 +O_\cD\!\left(B^{\HH}_{0,\lambda}(r)\eps^4\log\frac2\eps\right).
\end{equation}
\end{proposition}

\begin{proof}
By \eqref{eq2.38}, the difference of the two powers in the formula for \(I-Q\), and its normalized scale derivative, are bounded by
\[
C\left((B^{\RR}_\delta)^{2^*-2}f_\delta +f_\delta^{2^*-1}\right).
\]
On \(|x|\le1/2\), Lemma~\ref{Lem3.2} gives \(h_r\le CB^{\HH}_{0,\lambda}(r)\). In dimension four, \eqref{eq2.16} and the explicit bubble imply
\[
(B^{\RR}_\delta)^2f_\delta \le C\frac{\delta^3}{(\delta+|x|)^4} \left(1+\log\frac1{\delta+|x|}\right), \qquad f_\delta^3\le C\delta^3 \left(1+\log\frac1{\delta+|x|}\right)^3.
\]
Splitting the first radial integral at \(|x|=\delta\) gives
\[
\int_{|x|<1/2}(B^{\RR}_\delta)^2f_\delta\,dx \le C\delta^3\left(1+\ell_\delta +\int_\delta^{1/2}\frac{1+\log(1/t)}t\,dt\right) \le C\delta^3\ell_\delta^2.
\]
The integrability of \(t^3(1+\log(1/t))^3\) at zero gives \(\int_{|x|<1/2}f_\delta^3\,dx\le C\delta^3\).

For \(n\ge5\), the same bounds give
\[
(B^{\RR}_\delta)^{2^*-2}f_\delta \le C\frac{\delta^{\frac{n+2}{2}}}{(\delta+|x|)^n}, \qquad f_\delta^{2^*-1} \le C\frac{\delta^{\frac{n+2}{2}}} {(\delta+|x|)^{(n-4)(2^*-1)}}.
\]
Their interior integrals are bounded respectively by \(C\delta^{(n+2)/2}\ell_\delta\) and \(C\delta^{(n+2)/2}\), since \(n-(n-4)(2^*-1)=8/(n-2)>0\).

On \(\mathcal C\), both \(B^{\RR}_\delta\) and \(f_\delta\) are at most \(C\delta^{(n-2)/2}\). Thus \eqref{eq3.6}--\eqref{eq3.7} yield, in every dimension,
\[
\int_{\mathcal C}h_r \left((B^{\RR}_\delta)^{2^*-2}f_\delta +f_\delta^{2^*-1}\right)\,dx \le CB^{\HH}_{0,\lambda}(r)\delta^{\frac{n+2}{2}}.
\]
The same integrable majorants hold up to fixed constants in a dyadic neighborhood of each positive \(\delta\). They justify passing \(\cD\) under the integral. The regional estimates therefore prove \eqref{eq3.14}, including its scale derivative, uniformly in \(r\).

Theorem~\ref{Thm3.3} gives \(Q\sim\eps^{(n-2)/2} B^{\HH}_{0,\lambda}(r)\) and \(\cD Q/Q=(n-2)/2+O(\eps^2\log(2/\eps))\). Dividing \eqref{eq3.14} by \(Q\) and differentiating the quotient proves \eqref{eq3.15}. The identity
\[
\frac{\cD I}{I}=\frac{\cD Q}{Q}+\cD\log\frac IQ
\]
then gives \eqref{eq3.16} and \eqref{eq3.17}. For \(n=6\), combining \eqref{eq3.14} with \eqref{eq3.9} and \(4|\mathbb S^5|\kappa_6=\mathfrak m_6\) proves \eqref{eq3.18}.
\end{proof}

The scalar estimates must be supplemented by an \(H^{-1}_\lambda\) estimate to treat the residual equation for a perturbed configuration. Define
\begin{equation}\label{eq3.19}
\mathcal M_{z,\xi,\eps}=(B^{\HH}_{z,\lambda}+B^{\HH}_{\mathrm{pr},\xi,\eps})^{2^*-1}-(B^{\HH}_{z,\lambda})^{2^*-1}-(B^{\HH}_{\mathrm{pr},\xi,\eps})^{2^*-1}.
\end{equation}
Since \(B^{\HH}_{z,\lambda}\) solves the hyperbolic equation,
\[
\cR_\lambda(B^{\HH}_{z,\lambda}+B^{\HH}_{\mathrm{pr},\xi,\eps}) =\cR_\lambda(B^{\HH}_{\mathrm{pr},\xi,\eps})-\mathcal M_{z,\xi,\eps}.
\]
The next lemma gives the dimension-dependent dual bounds for this mixed term, its normalized scale derivative, and the two interaction terms arising from variations of the coefficients. Here \(\cD\) acts only on the projected Euclidean profile.

\begin{lemma}\label{Lem3.5}
For every fixed \(n\ge4\) and \(\lamc<\lambda<\lambdab\), there are \(\eps_0>0\) and \(C>0\) such that, for \(0<\eps<\eps_0\), uniformly in \(z,\xi\),
\begin{equation*}
\dnorm{\mathcal M_{z,\xi,\eps}} +\dnorm{\cD\mathcal M_{z,\xi,\eps}}
\le C\begin{cases}
I(\eps,r),&n=4,5,\\[1mm]
\eps^2\sqrt{\log(2/\eps)},&n=6,\\[1mm]
\eps^{\frac{n+2}{4}},&n\ge7.
\end{cases}
\end{equation*}
Each of
\begin{equation}\label{eq3.20}
\begin{aligned}
&\left[(B^{\HH}_{z,\lambda}+B^{\HH}_{\mathrm{pr},\xi,\eps})^{2^*-2} -(B^{\HH}_{z,\lambda})^{2^*-2}\right]B^{\HH}_{z,\lambda},\\
&\left[(B^{\HH}_{z,\lambda}+B^{\HH}_{\mathrm{pr},\xi,\eps})^{2^*-2} -(B^{\HH}_{\mathrm{pr},\xi,\eps})^{2^*-2}\right] B^{\HH}_{\mathrm{pr},\xi,\eps}
\end{aligned}
\end{equation}
has \(H^{-1}_\lambda\) norm bounded by the same right side. For \(n\ge7\), the bound \(C\eps^{(n+2)/4}\) also holds in \(L^{(2^*)'}(\HH^n)\) for \(\mathcal M_{z,\xi,\eps}\), \(\cD\mathcal M_{z,\xi,\eps}\), and the two expressions in \eqref{eq3.20}.
\end{lemma}

\begin{proof}
Normalize the centers by a common isometry as in \eqref{eq3.8}. In conformal variables,
\[
B^{\HH}_{z,\lambda}=\varrho^{-\frac{n-2}{2}}h_r,\qquad B^{\HH}_{\mathrm{pr},\xi,\eps}=\varrho^{-\frac{n-2}{2}}B^{\RR}_{\mathrm{pr},\delta}.
\]
Since \(\frac{n-2}{2}\,(2^*-1)=\frac{n+2}{2}\) and \(\dv=\varrho^n\,dx\), all conformal weights cancel in duality with \(\varrho^{-\frac{n-2}{2}}\phi\), \(\phi\in H_0^1(\BB^n)\). Thus the representative of \eqref{eq3.19} in the Euclidean metric is
\[
\mathcal M^{\RR}_{\delta,r} =(h_r+B^{\RR}_{\mathrm{pr},\delta})^{2^*-1}-h_r^{2^*-1}-(B^{\RR}_{\mathrm{pr},\delta})^{2^*-1} .
\]
The two expressions in \eqref{eq3.20} have the same conformal weight. Lemma~\ref{Lem2.4} gives, in every dimension,
\begin{equation}\label{eq3.21}
|\cD B^{\RR}_{\mathrm{pr},\delta}|\le C B^{\RR}_{\mathrm{pr},\delta}.
\end{equation}
Moreover,
\begin{equation}\label{eq3.22}
\cD\mathcal M^{\RR}_{\delta,r} =(2^*-1)\left[(h_r+B^{\RR}_{\mathrm{pr},\delta})^{2^*-2} -(B^{\RR}_{\mathrm{pr},\delta})^{2^*-2}\right] \cD B^{\RR}_{\mathrm{pr},\delta}.
\end{equation}

Suppose first that \(n=4,5\). For positive \(s,t\),
\begin{equation}\label{eq3.23}
(s+t)^{2^*-1}-s^{2^*-1}-t^{2^*-1} \leq C_{2^*-1}\bigl(s^{2^*-2}t+st^{2^*-2}\bigr).
\end{equation}
On a fixed interior ball, \(h_r\leq CB^{\HH}_{0,\lambda}(r)\) uniformly for \(r\ge0\), by Lemma~\ref{Lem3.2}. Sobolev duality and the scaling of the Aubin--Talenti bubble give
\[
\|B^{\RR}_{\mathrm{pr},\delta}\|_{L^{4/3}}+ \|(B^{\RR}_{\mathrm{pr},\delta})^{2}\|_{L^{4/3}}\leq C\delta\quad(n=4),
\]
and
\[
\|B^{\RR}_{\mathrm{pr},\delta}\|_{L^{10/7}}+ \|(B^{\RR}_{\mathrm{pr},\delta})^{4/3}\|_{L^{10/7}} \leq C\delta^{3/2} \quad(n=5),
\]
where all norms in these two displays are taken on the fixed interior ball. Hence \eqref{eq3.23} contributes at most \(C\delta^{\frac{n-2}{2}}B^{\HH}_{0,\lambda}(r)\) to \(H^{-1}\) there.

It remains to treat the annulus \(\{x\in\BB^n:1/2<|x|<1\}\). The Euclidean boundary Hardy inequality gives
\begin{equation}\label{eq3.24}
\|f\|_{H^{-1}(\BB^n)}\le C\|(1-|x|) f\|_{L^2(\BB^n)}.
\end{equation}
Indeed, for \(\phi\in H_0^1(\BB^n)\),
\[
\left|\int_{\BB^n}f\phi\,dx\right| \le\|(1-|x|) f\|_2\|\phi/(1-|x|)\|_2 \le C\|(1-|x|) f\|_2\|\nabla\phi\|_2.
\]
By \eqref{eq2.3}, this also controls the corresponding \(H^{-1}_\lambda\) norm.

Set \(\tau=e^{-r}\). Under the normalization \(z_r=\tanh(r/2)e_1\), Lemma~\ref{Lem3.2} and \eqref{eq2.11} give directly, on \(\mathcal C\),
\begin{equation}\label{eq3.25}
B^{\RR}_{\mathrm{pr},\delta}(x)\le C\delta^{\frac{n-2}{2}}(1-|x|)^{\alpha_\lambda}, \qquad h_r(x)\le C B^{\HH}_{0,\lambda}(r) \frac{(1-|x|)^{\alpha_\lambda}} {(\tau+|x-e_1|)^{2\gamma_\lambda}}.
\end{equation}
For \(0\le r\le1\), \(\tau\) and \(B^{\HH}_{0,\lambda}(r)\) are bounded below, so \eqref{eq3.23}--\eqref{eq3.25} yield the required bound immediately. Assume henceforth that \(r\ge1\). In polar coordinates centered at \(e_1\), use \(1-|x|\le|x-e_1|\), with \(t\) as the radial integration variable. All norms below are taken on \(\mathcal C\).

In dimension $4$, the dominant term is \(h_r^2 B^{\RR}_{\mathrm{pr},\delta}\). Since \(\gamma_\lambda=1+\alpha_\lambda\), \eqref{eq3.25} gives
\begin{align}
\|(1-|x|) h_r^2B^{\RR}_{\mathrm{pr},\delta}\|_2^2 &\le C\delta^2 B^{\HH}_{0,\lambda}(r)^4 \int_0^2\frac{t^{6\gamma_\lambda-1}} {(\tau+t)^{8\gamma_\lambda}}\,dt\notag\\
&\le C\delta^2 B^{\HH}_{0,\lambda}(r)^4\tau^{-2\gamma_\lambda} \le C\delta^2 B^{\HH}_{0,\lambda}(r)^2.\label{eq3.26}
\end{align}
For the other term,
\begin{align*}
\|(1-|x|) h_r(B^{\RR}_{\mathrm{pr},\delta})^2\|_2^2 &\le C\delta^4 B^{\HH}_{0,\lambda}(r)^2 \int_0^2\frac{t^{6\gamma_\lambda-1}} {(\tau+t)^{4\gamma_\lambda}}\,dt\\
&\le C\delta^4 B^{\HH}_{0,\lambda}(r)^2,
\end{align*}
because the integral is at most \(\int_0^2t^{2\gamma_\lambda-1}\,dt<\infty\).

In dimension $5$, the dominant term is \(h_r^{4/3}B^{\RR}_{\mathrm{pr},\delta}\). Now \(\gamma_\lambda=3/2+\alpha_\lambda\), so
\begin{align}
\|(1-|x|) h_r^{4/3}B^{\RR}_{\mathrm{pr},\delta}\|_2^2 &\le C\delta^3 B^{\HH}_{0,\lambda}(r)^{8/3} \int_0^2\frac{t^{14\gamma_\lambda/3-1}} {(\tau+t)^{16\gamma_\lambda/3}}\,dt\notag\\
&\le C\delta^3 B^{\HH}_{0,\lambda}(r)^{8/3}\tau^{-2\gamma_\lambda/3} \le C\delta^3 B^{\HH}_{0,\lambda}(r)^2.\label{eq3.27}
\end{align}
For the other term,
\begin{align*}
\|(1-|x|) h_r(B^{\RR}_{\mathrm{pr},\delta})^{4/3}\|_2^2 &\le C\delta^4 B^{\HH}_{0,\lambda}(r)^2 \int_0^2\frac{t^{14\gamma_\lambda/3-1}} {(\tau+t)^{4\gamma_\lambda}}\,dt\\
&\le C\delta^4 B^{\HH}_{0,\lambda}(r)^2,
\end{align*}
since this integral is bounded by \(\int_0^2t^{2\gamma_\lambda/3-1}\,dt\). Thus the smaller term contributes at most \(C\delta^2 B^{\HH}_{0,\lambda}(r)\) in either dimension, which is bounded by \(C\delta^{\frac{n-2}{2}}B^{\HH}_{0,\lambda}(r)\). The last inequalities in \eqref{eq3.26} and \eqref{eq3.27} use \(B^{\HH}_{0,\lambda}(r)\sim\tau^{\gamma_\lambda}\).

For the radial integrals in \eqref{eq3.26} and \eqref{eq3.27}, the substitution \(t=\tau u\) gives, whenever \(b>a>0\),
\[
\int_0^2\frac{t^{a-1}}{(\tau+t)^b}\,dt =\tau^{a-b}\int_0^{2/\tau}\frac{u^{a-1}}{(1+u)^b}\,du \le C_{a,b}\tau^{a-b}.
\]
Combining \eqref{eq3.24} with \eqref{eq3.26}--\eqref{eq3.27} proves
\[
\|\mathcal M_{z,\xi,\eps}\|_{H^{-1}_\lambda} \leq C\delta^{\frac{n-2}{2}}B^{\HH}_{0,\lambda}(r).
\]
Proposition~\ref{Prop3.4} shows that the right side is comparable to \(I(\eps,r)\) with fixed constants.

For the interaction terms arising from variations of the coefficients, the mean value theorem gives, for \(s,t>0\),
\begin{align*}
&\big[(s+t)^{2^*-2}-s^{2^*-2}\big]s +\big[(s+t)^{2^*-2}-t^{2^*-2}\big]t\le C\bigl(s^{2^*-2}t+st^{2^*-2}\bigr).
\end{align*}
Thus the same two monomials control their representatives in the Euclidean metric. Equations \eqref{eq3.21} and \eqref{eq3.22} give the identical bound for \(|\cD\mathcal M^{\RR}_{\delta,r}|\). The estimates on \(\{x\in\BB^n:|x|\le1/2\}\) and \(\{x\in\BB^n:1/2<|x|<1\}\) just proved therefore establish all the asserted dual bounds when \(n=4,5\).

Now let \(n=6\), so \(2^*=3\). The exact identity is
\[
\mathcal M_{z,\xi,\eps} =2B^{\HH}_{z,\lambda}B^{\HH}_{\mathrm{pr},\xi,\eps}, \qquad \mathcal M^{\RR}_{\delta,r}=2h_rB^{\RR}_{\mathrm{pr},\delta}.
\]
Both quantities in \eqref{eq3.20} equal \(B^{\HH}_{z,\lambda}B^{\HH}_{\mathrm{pr},\xi,\eps}\). On \(|x|\le1/2\), Lemmas~\ref{Lem2.3}--\ref{Lem2.4} and the uniform bound for \(h_r\) imply
\[
|2h_rB^{\RR}_{\mathrm{pr},\delta}| +|\cD(2h_rB^{\RR}_{\mathrm{pr},\delta})| \le C\frac{\delta^2}{(\delta+|x|)^4}.
\]

The weighted Cauchy--Schwarz estimate \eqref{eq2.12} therefore bounds the interior dual norm by \(C\delta^2\sqrt{\ell_\delta}\). The common pointwise bound gives the same estimate for the scale derivative and both coefficient terms.

For the exterior, let \(\Omega_{\mathrm{out}}\) denote the hyperbolic region represented by \(\mathcal C\) in the ball centered at \(\xi\). Critical conformal invariance and \eqref{eq2.11} give
\begin{equation}\label{eq3.28}
\|B^{\HH}_{\mathrm{pr},\xi,\eps}\|_{L^{2^*}(\Omega_{\mathrm{out}})} =\|B^{\RR}_{\mathrm{pr},\delta}\|_{L^{2^*}(\mathcal C)} \le C\delta^{\frac{n-2}{2}}.
\end{equation}
In dimension $6$, H\"older's inequality therefore yields
\[
\|B^{\HH}_{z,\lambda}B^{\HH}_{\mathrm{pr},\xi,\eps}\|_{L^{3/2}(\Omega_{\mathrm{out}})} \le\|B^{\HH}_{z,\lambda}\|_{L^3(\HH^6)} \|B^{\HH}_{\mathrm{pr},\xi,\eps}\|_{L^3(\Omega_{\mathrm{out}})} \le C\delta^2.
\]
The hyperbolic norm is independent of \(z\), and \eqref{eq3.21} gives the same estimate with the scale derivative of the projected bubble. Sobolev duality, together with the interior estimate and \(\eps=2\delta\), proves all the six-dimensional assertions uniformly in the centers.

Finally, let \(n\ge7\), so \(2<2^*<3\). For positive \(h,b\),
\begin{equation}\label{eq3.29}
0\le(h+b)^{2^*-1}-h^{2^*-1}-b^{2^*-1} \le C\min\{hb^{2^*-2},h^{2^*-2}b\}.
\end{equation}
Indeed, if \(b\ge h\), the mean value theorem after subtracting \(b^{2^*-1}\) gives \(Chb^{2^*-2}\), which is the smaller monomial. The other case is symmetric. The same two cases give
\begin{align*}
&\big[(h+b)^{2^*-2}-h^{2^*-2}\big]h +\big[(h+b)^{2^*-2}-b^{2^*-2}\big]b\le C\min\{hb^{2^*-2},h^{2^*-2}b\}.
\end{align*}
When \(b\ge h\), the first term is at most \(Chb^{2^*-2}\), and the second is at most \(Chb^{2^*-3}b\). Exchange \(h\) and \(b\) for the other case. Combining these inequalities with \eqref{eq3.21} and \eqref{eq3.22}, the Euclidean representatives of these four expressions are bounded by
\[
C\min\left\{h_r(B^{\RR}_{\mathrm{pr},\delta})^{2^*-2}, h_r^{2^*-2}B^{\RR}_{\mathrm{pr},\delta}\right\}.
\]

For every \(F\in L^{(2^*)'}(\BB^n)\), conformal change of variables gives
\[
\|\varrho^{-\frac{n+2}{2}}F\|_{L^{(2^*)'}(\HH^n)} =\|F\|_{L^{(2^*)'}(\BB^n)}, \qquad (2^*)'=\frac{2n}{n+2}.
\]
On \(|x|\le1/2\), \(h_r\le C\) uniformly in the center. By Lemma~\ref{Lem2.3},
\[
B^{\RR}_{\mathrm{pr},\delta}(|x|) \sim B^{\RR}_{\delta}(|x|) \sim\frac{\delta^{\frac{n-2}{2}}}{(\delta+|x|)^{n-2}}.
\]
It remains to estimate \(\min\{(B^{\RR}_{\mathrm{pr},\delta})^{2^*-2},B^{\RR}_{\mathrm{pr},\delta}\}\). We bound it by the first monomial on \(|x|\le\sqrt\delta\) and by the second on \(|x|>\sqrt\delta\). For \(\delta<1/4\), the three radial integrals satisfy
\begin{align*}
\int_0^\delta(\delta^{-2})^{\frac{2n}{n+2}}t^{n-1}\,dt &\le C\delta^{\frac{n(n-2)}{n+2}},\\
\int_\delta^{\sqrt\delta}(\delta^2t^{-4})^{\frac{2n}{n+2}} t^{n-1}\,dt &\le C\delta^{n/2},\\
\int_{\sqrt\delta}^{1/2} (\delta^{\frac{n-2}{2}}t^{2-n})^{\frac{2n}{n+2}} t^{n-1}\,dt &\le C\delta^{n/2}.
\end{align*}
The annular exponents have the required signs because
\[
n-\frac{8n}{n+2}=\frac{n(n-6)}{n+2}>0, \qquad n-\frac{2n(n-2)}{n+2}=-\frac{n(n-6)}{n+2}<0.
\]

Since \(n(n-2)/(n+2)>n/2\), taking the \((n+2)/(2n)\)-power gives the interior bound \(C\delta^{(n+2)/4}\).

On \(\Omega_{\mathrm{out}}\), retain the second monomial in \eqref{eq3.29}. H\"older's inequality and \eqref{eq3.28} give
\begin{align*}
&\|(B^{\HH}_{z,\lambda})^{2^*-2} B^{\HH}_{\mathrm{pr},\xi,\eps}\|_{L^{(2^*)'}(\Omega_{\mathrm{out}})}\le \|(B^{\HH}_{z,\lambda})^{2^*-2}\|_{L^{n/2}(\HH^n)} \|B^{\HH}_{\mathrm{pr},\xi,\eps}\|_{L^{2^*}(\Omega_{\mathrm{out}})} \le C\delta^{\frac{n-2}{2}}.
\end{align*}
Here \(\frac1{(2^*)'}=\frac2n+\frac1{2^*}\) and \((2^*-2)n/2=2^*\), so the first factor is independent of the hyperbolic center. Since \((n-2)/2>(n+2)/4\) for \(n>6\), the exterior satisfies the required bound as well. This proves the stronger Lebesgue estimates for these four expressions. Sobolev duality gives their \(H^{-1}_\lambda\) estimates and completes the proof.
\end{proof}

\section{Local stability and sharpness}
\label{Sec4}

We prove Theorem~\ref{Thm1.1} by choosing orthogonal parameters, establishing uniform coercivity, and testing the residual against the profiles and the scale derivative. The coefficient errors and \(s+I\) admit linear bounds in every dimension; the remainder has the logarithmic loss when \(n=6\). We then prove parameter localization and construct sharpness examples.

\subsection{Orthogonal decomposition and uniform coercivity}
\label{Sec4.1}

For \(\theta=(z,\eps,\xi)\), put \(r=d_g(z,\xi)\) and write
\[
H=B^{\HH}_{z,\lambda},\qquad B=B^{\HH}_{\mathrm{pr},\xi,\eps},\qquad \Sigma_\theta=H+B.
\]
Choose hyperbolic orthonormal frames at \(z\) and \(\xi\), parallel along the geodesic joining the centers when \(z\ne\xi\). Let \(Y_j=\partial_{z,j}H\), \(1\le j\le n\), denote the center derivatives of the hyperbolic profile. With \(\cD=\eps\partial_\eps\), retain the normalized derivatives \(Z_0=\cD B\) and \(Z_j=\eps\partial_{\xi,j}B\), \(1\le j\le n\), from Section~\ref{Sec2.3}, and set
\begin{align}
\cK_\theta&=\Span\{H,Y_1,\ldots,Y_n,B,Z_0,Z_1,\ldots,Z_n\}, \label{eq4.1}\\
\cX_\theta&=\cK_\theta^{\perp_\lambda}. \notag
\end{align}
The space \(\cK_\theta\) is the tangent space, at coefficients \((1,1)\), to the family obtained by varying the two coefficients, the centers, and the Euclidean scale. The coefficient derivatives \(H\) and \(B\) account for the negative directions of the limiting single-profile quadratic forms. Orthogonality to the full space \(\cK_\theta\) gives positive coercivity on \(\cX_\theta\). Both coefficients remain fixed at \(1\) in the comparison family \(\cM_\mix(\eps_0)\).

We next compute the cross terms in the Gram matrix. For \(r>0\), let \(\gamma:[0,r]\to\HH^n\) be the unit-speed geodesic from \(z\) to \(\xi\). Choose
\[
e_1^z=-\dot\gamma(0),\qquad e_1^\xi=-\dot\gamma(r),
\]
so the two vectors are parallel, the first points away from \(\xi\), and the second points toward \(z\). Complete them by parallel transverse frames \(e_i^z,e_i^\xi\), \(2\leq i\leq n\). Thus
\[
d_zr(e_1^z)=1,\qquad d_\xi r(e_1^\xi)=-1,
\]
and the Jacobi field formula in curvature \(-1\) gives
\[
\nabla^2_{z,\xi}r(e_i^z,e_j^\xi) =-\frac{\delta_{ij}}{\sinh r}\quad(i,j\geq2).
\]
Every cross entry is a parameter derivative of \(Q(\eps,r)=\ip{H}{B}\), defined in \eqref{eq3.8}. With primes denoting \(r\)-derivatives, the full block is
\begin{equation}\label{eq4.2}
\begin{array}{c|cccc}
&B^{\HH}_{\mathrm{pr},\xi,\eps}&Z_0&Z_1&Z_j,\ j\ge2\\
\hline B^{\HH}_{z,\lambda}&Q&\cD Q&-\eps Q'&0\\
Y_1&Q'&\cD Q'&-\eps Q''&0\\
Y_i,\ i\ge2&0&0&0& -\eps\dfrac{Q'}{\sinh r}\,\delta_{ij}
\end{array}
\end{equation}
where the last quotient is interpreted at \(r=0\) by its limit \(Q''(0)\). Indeed, smooth radiality gives \(Q'(\eps,0)=0\) and \(Q'(\eps,r)/\sinh r\to Q''(\eps,0)\). When the centers coincide, the resulting bilinear form is independent of the chosen common frame. Theorem~\ref{Thm3.3} therefore controls the block for all relative positions.

\begin{lemma}\label{Lem4.1}
For sufficiently small \(\eps\), the Gram matrix of the ordered family in \eqref{eq4.1}, computed in the energy inner product, is uniformly positive definite, independently of \(z,\xi\), and \(r=d_g(z,\xi)\).
\end{lemma}

\begin{proof}
The diagonal hyperbolic block is fixed by isometry. The strong limits \eqref{eq2.19} and \eqref{eq2.25}, together with the vanishing potential pairings \eqref{eq2.22} in Lemma~\ref{Lem2.7}, show that the diagonal Euclidean block converges to the Euclidean Gram matrix of
\[
B^{\RR}_{1},\quad Z_0^{\RR},\quad-\partial_1B^{\RR}_{1},\ldots,-\partial_nB^{\RR}_{1},
\]
which is positive definite. Every cross entry is listed in \eqref{eq4.2}. The estimates of Theorem~\ref{Thm3.3} imply that the operator norm of the entire cross block tends to zero uniformly in \(r\). The result follows by the Schur complement.
\end{proof}

Uniformity and local uniqueness below are understood in the scaled variables
\[
|dz|_g+\frac{|d\xi|_g}{\eps}+\frac{|d\eps|}{\eps}.
\]
A chart centered at \(\theta=(z,\eps,\xi)\) uses the increment
\begin{equation*}
\left( \exp_z^{-1}(z'),\ \log\frac{\eps'}{\eps},\ \frac1\eps\exp_\xi^{-1}(\xi') \right)
\end{equation*}
on a ball of fixed radius. Parameter selection uses local smooth orthonormal frames at the two centers separately, including when they coincide. The orthogonality conditions are invariant under changes of these frames.

For \(\zeta\in T_z\HH^n\), \(\eta\in T_\xi\HH^n\), and \(\sigma\in\RR\), take components in the chosen frames. The differential of the profile map at the base point is
\[
\left.\frac{\mathrm d}{\mathrm dt}\right|_{t=0} \Sigma_{(\exp_z(t\zeta),\,\eps e^{t\sigma},\, \exp_\xi(t\eps\eta))} =\sum_{j=1}^n\zeta_jY_j+\sigma Z_0 +\sum_{j=1}^n\eta_jZ_j.
\]
These scaled first variations are uniformly bounded by Lemma~\ref{Lem2.8}. Orthogonality depends only on their span, so changes of frames or coordinates do not change the equations. On a fixed small chart, the scale ratios, coefficients near \(1\), and differentials of the exponential maps are uniformly bounded and invertible, including at coincident centers.

\begin{proposition}\label{Prop4.2}
After decreasing \(\eps_0\) if necessary, there is \(\eta_0>0\) such that every \(u\) satisfying \eqref{eq1.2} has a locally unique representation in the scaled coordinates above, uniformly in the reference parameters:
\begin{equation}\label{eq4.3}
u=\nu_\Hyp B^{\HH}_{z,\lambda}+\nu_\Euc B^{\HH}_{\mathrm{pr},\xi,\eps}+\rho,\qquad\rho\in\cX_\theta.
\end{equation}
The parameters and remainder satisfy
\[
\begin{aligned}
&|\nu_\Hyp-1|+|\nu_\Euc-1|+d_g(z,z_*) +\left|\log\frac{\eps}{\eps_*}\right| +\frac{d_g(\xi,\xi_*)}{\eps_*}+\norm{\rho}\\
&\qquad\le C\norm{u-\Sigma_{\theta_*}}, \qquad\theta_*=(z_*,\eps_*,\xi_*).
\end{aligned}
\]
\end{proposition}

\begin{proof}
Fix \(\theta_*=(z_*,\eps_*,\xi_*)\) and use coordinates
\[
\mathbf p=(\nu_\Hyp,\nu_\Euc,\zeta,\sigma,\eta),\qquad z=\exp_{z_*}\zeta,\quad\eps=\eps_*e^\sigma,\quad\xi=\exp_{\xi_*}(\eps_*\eta),
\]
with base point \(\mathbf p_0=(1,1,0,0,0)\). Set
\[
f(\mathbf p)=\nu_\Hyp B^{\HH}_{z,\lambda}+\nu_\Euc B^{\HH}_{\mathrm{pr},\xi,\eps},\qquad f_i=\partial_{p_i}f,\qquad F_i(\mathbf p)=\ip{u-f(\mathbf p)}{f_i(\mathbf p)}.
\]
Write \(F=(F_i)_i\) and \(f_{ij}=\partial_{p_j}\partial_{p_i}f\). At \(\mathbf p_0\), the vectors \(f_i(\mathbf p_0)\) are precisely the functions listed in \eqref{eq4.1}, up to ordering. On a small fixed coordinate ball, \(\nu_\Hyp,\nu_\Euc\) stay away from zero, and the differentials of the exponential charts are invertible. The center derivatives of \(f\) are the corresponding linear combinations of \(\nu_\Hyp Y_j\) and \(\nu_\Euc(\eps_*/\eps)Z_j\), while its scale derivative is \(\nu_\Euc Z_0\). Thus the \(f_i(\mathbf p)\) span \(\cK_\theta\), and \(F(\mathbf p)=0\) is exactly the required orthogonality condition.

Let \(G_0=(\ip{f_j(\mathbf p_0)}{f_i(\mathbf p_0)})_{i,j}\). Lemma~\ref{Lem4.1} bounds \(G_0^{-1}\) uniformly. Lemma~\ref{Lem2.8}, Proposition~\ref{Prop3.1}, and the uniform geometry of the charts give
\[
\max_i\norm{f_i(\mathbf p)}+\max_{i,j}\norm{f_{ij}(\mathbf p)}\le C
\]
on a fixed ball about \(\mathbf p_0\). Writing \(d=\norm{u-f(\mathbf p_0)}\), we have
\[
\partial_{p_j}F_i(\mathbf p) =-\ip{f_j(\mathbf p)}{f_i(\mathbf p)}+\ip{u-f(\mathbf p)}{f_{ij}(\mathbf p)},
\]
and the preceding bounds imply, with a prime denoting the full derivative with respect to \(\mathbf p\),
\[
\|F'(\mathbf p)+G_0\|\le C\bigl(|\mathbf p-\mathbf p_0|+d\bigr).
\]
Indeed, the bounded second derivatives make \(f_i\) Lipschitz, and \(\norm{u-f(\mathbf p)}\le d+C|\mathbf p-\mathbf p_0|\).

Consider \(\mathcal T(\mathbf p)=\mathbf p+G_0^{-1}F(\mathbf p)\). Choose a fixed radius \(0<\eta_{\mathrm{chart}}<\log2\) sufficiently small and then choose \(\eta_0\) so that, whenever \(d<\eta_0\),
\[
\sup_{|\mathbf p-\mathbf p_0|\le \eta_{\mathrm{chart}}}\|\mathcal T'(\mathbf p)\|\le\tfrac12, \qquad|\mathcal T(\mathbf p_0)-\mathbf p_0|\le C_0d\le\tfrac{\eta_{\mathrm{chart}}}{2}.
\]
Thus \(\mathcal T\) maps the closed ball of radius \(\eta_{\mathrm{chart}}\) into itself and is a contraction there. Its unique fixed point satisfies \(F(\mathbf p)=0\) and \(|\mathbf p-\mathbf p_0|\le2C_0d\). In particular, the coefficient errors, scaled parameter increments, and \(\norm{\rho}\) are bounded by \(Cd\), proving the stated uniformity and local uniqueness. Changes of orthonormal frames preserve the equations, and common hyperbolic isometries preserve the construction. Finally, \(\eps_*<\eps_0/2\) and \(|\sigma|\le \eta_{\mathrm{chart}}<\log2\) ensure \(\eps<\eps_0\).
\end{proof}

Define the self-adjoint linearization at the reference sum by
\begin{equation*}
L_\theta=P^{\HH}_\lambda-(2^*-1)\Sigma_\theta^{2^*-2}.
\end{equation*}
For \(v\in H^1_\lambda(\HH^n)\), define its rescaling at \((\xi,\eps)\) by
\[
(\mathscr S_{\xi,\eps}v)(y) =\delta^{\frac{n-2}{2}}\varrho(\delta y)^{\frac{n-2}{2}} v\bigl(\tau_\xi(\delta y)\bigr), \qquad \delta=\eps/2,\quad |y|<\delta^{-1},
\]
and set it equal to zero outside this ball. Here \(\varrho(x)=2/(1-|x|^2)\), and \(\tau_\xi\) is a hyperbolic isometry sending the origin to \(\xi\), as in Section~\ref{Sec2.1}. The Euclidean representative has zero trace, and \eqref{eq2.3} shows that these maps are uniformly bounded from \(H^1_\lambda(\HH^n)\) to \(\dot H^1(\RR^n)\).

We next prove a uniform spectral gap on \(\cX_\theta\).

\begin{proposition}\label{Prop4.3}
For every fixed admissible \((n,\lambda)\), there are \(c,C,\eps_0>0\) such that
\begin{equation}\label{eq4.4}
\langle L_\theta\rho,\rho\rangle\ge c\norm{\rho}^2, \qquad \rho\in\cX_\theta,\quad0<\eps<\eps_0,
\end{equation}
uniformly in both centers. Define
\[
L_\theta^\perp:\cX_\theta\longrightarrow\cX_\theta^*, \qquad(L_\theta^\perp\rho)[\phi]=\langle L_\theta\rho,\phi\rangle.
\]
With the energy norm on \(\cX_\theta\) and its dual norm on \(\cX_\theta^*\),
\begin{equation}\label{eq4.5}
\bigl\|(L_\theta^\perp)^{-1}\bigr\|_{\cX_\theta^*\to\cX_\theta}\le C.
\end{equation}
\end{proposition}

\begin{proof}
We first recall the two single-profile spectral gaps. A weight in \(L^{n/2}\) defines a compact quadratic form on the energy space, by approximation with bounded compactly supported weights and local Rellich compactness; this also applies to weights converging strongly in \(L^{n/2}\).

The hyperbolic spectral gap is
\begin{equation}\label{eq4.6}
(2^*-1)\int_{\HH^n}(B^{\HH}_{0,\lambda})^{2^*-2}v^2\dv \le(1-c_\Hyp)\norm{v}^2, \quad v\perp_\lambda\{B^{\HH}_{0,\lambda},Y_1|_{z=0},\ldots,Y_n|_{z=0}\},
\end{equation}
for some \(c_\Hyp>0\). Indeed, \cite[Proposition 3.1 and Lemma 3.2]{BGKMOneBubble}, using the nondegeneracy theorem \cite[Theorem 1.1]{GangulySandeep2015}, identifies the first two weighted eigenspaces with the profile and its center derivatives, with eigenvalues \(1\) and \(2^*-1\). The next eigenvalue is strictly larger than \(2^*-1\). Energy orthogonality to these eigenfunctions is equivalent to weighted \(L^2\)-orthogonality, so their removal gives \eqref{eq4.6}.

The Euclidean gap is
\begin{equation}\label{eq4.7}
(2^*-1)\int_{\RR^n}(B^{\RR}_1)^{2^*-2}v^2\,dy \le(1-c_\Euc)\|\nabla v\|_2^2, \quad v\perp_{\dot H^1}\{B^{\RR}_1,Z_0^{\RR},\partial_1B^{\RR}_1,\ldots,\partial_nB^{\RR}_1\},
\end{equation}
with \(c_\Euc=4/(n+4)\). Under stereographic projection, the weighted eigenvalues are \(1+4j(j+n-1)/(n(n-2))\), \(j\ge0\). The constraints remove the spherical harmonics of degrees zero and one; the next eigenvalue is \((2^*-1)(n+4)/n\), which gives the stated constant. The first two weighted eigenspaces are identified in \cite{BianchiEgnell1991}.

Let \(\theta_\ell=(z_\ell,\eps_\ell,\xi_\ell)\), with \(\eps_\ell\to0\), \(\rho_\ell\in\cX_{\theta_\ell}\), and \(\norm{\rho_\ell}=1\). For this proof write
\[
H_\ell=B^{\HH}_{z_\ell,\lambda},\qquad B_\ell=B^{\HH}_{\mathrm{pr},\xi_\ell,\eps_\ell},\qquad \delta_\ell=\eps_\ell/2.
\]
After recentering at \(z_\ell\), take a weak limit \(\rho_\Hyp\) in \(H^1_\lambda\). Set \(\widehat\rho_\ell=\mathscr S_{\xi_\ell,\eps_\ell}\rho_\ell\), and pass to a weak limit \(\rho_\Euc\) in \(\dot H^1(\RR^n)\). Hyperbolic orthogonality passes directly to \(\rho_\Hyp\). The strong limits \eqref{eq2.19} and \eqref{eq2.25}, and the vanishing potential pairings \eqref{eq2.22}, pass Euclidean orthogonality to \(\rho_\Euc\). The two limits satisfy the constraints in \eqref{eq4.6} and \eqref{eq4.7}, respectively.

We claim that their energies satisfy
\begin{equation}\label{eq4.8}
\norm{\rho_\Hyp}^2+\|\nabla\rho_\Euc\|_2^2\le1.
\end{equation}
Choose fixed smooth compactly supported approximations \(f\) and \(v\) to the two limits. Let \(F_\ell\) be the hyperbolic translate of \(f\) to \(z_\ell\), and let \(V_\ell\) have Euclidean representative \(\delta_\ell^{-(n-2)/2}v(x/\delta_\ell)\) in the frame at \(\xi_\ell\). Then
\[
\norm{F_\ell}^2=\norm f^2,\qquad \norm{V_\ell}^2\longrightarrow\|\nabla v\|_2^2.
\]
The second assertion follows because the rescaled potential tends uniformly to zero on the fixed support of \(v\). Moreover, \(\|P^{\HH}_\lambda F_\ell\|_\infty\le C_f\) and \(\|V_\ell\|_{L^1(\HH^n)}=O(\delta_\ell^{(n+2)/2})\). Hence
\[
|\ip{F_\ell}{V_\ell}| =|\langle P^{\HH}_\lambda F_\ell,V_\ell\rangle| \le C_{f,v}\delta_\ell^{\frac{n+2}{2}}\longrightarrow0,
\]
uniformly in the relative positions. Weak convergence and the local vanishing of the potential also give
\[
\ip{\rho_\ell}{F_\ell}\longrightarrow\ip{\rho_\Hyp}{f}, \qquad \ip{\rho_\ell}{V_\ell}\longrightarrow\int\nabla\rho_\Euc\cdot\nabla v\,dy.
\]
For each fixed pair \((f,v)\), expanding \(\norm{\rho_\ell-F_\ell-V_\ell}^2\) and taking the preceding limits gives
\[
\begin{aligned}
0\le\lim_{\ell\to\infty} \norm{\rho_\ell-F_\ell-V_\ell}^2 ={}&1+\norm f^2+\|\nabla v\|_2^2 -2\ip{\rho_\Hyp}{f}\\
&-2\int_{\RR^n}\nabla\rho_\Euc\cdot\nabla v\,dy.
\end{aligned}
\]
After taking \(\ell\to\infty\) with these supports fixed, let \(f\to\rho_\Hyp\) in \(H^1_\lambda\) and \(v\to\rho_\Euc\) in \(\dot H^1\). This proves~\eqref{eq4.8}.

The compactness of the weighted forms yields
\begin{align*}
\int_{\HH^n}H_\ell^{2^*-2}\rho_\ell^2\dv &\longrightarrow\int_{\HH^n}(B^{\HH}_{0,\lambda})^{2^*-2}\rho_\Hyp^2\dv,\\
\int_{\HH^n}B_\ell^{2^*-2}\rho_\ell^2\dv &\longrightarrow\int_{\RR^n}(B^{\RR}_1)^{2^*-2}\rho_\Euc^2\,dy.
\end{align*}
For the second limit, use conformal invariance of this integral and the strong \(L^{2^*}\) convergence of the rescaled projected bubble. Its power \(2^*-2\) converges strongly in \(L^{n/2}\), since \((2^*-2)n/2=2^*\).

Only the comparison of the mixed potential now depends on the dimension. If \(n\ge6\), then \(0<2^*-2\le1\), so
\[
(H_\ell+B_\ell)^{2^*-2} \le H_\ell^{2^*-2}+B_\ell^{2^*-2},
\]
with equality for \(n=6\). If \(n=4,5\), then \(1<2^*-2\le2\), and
\[
0\le(H_\ell+B_\ell)^{2^*-2}-H_\ell^{2^*-2}-B_\ell^{2^*-2} \le C H_\ell^{2^*-3}B_\ell.
\]
The profiles \(H_\ell\) are uniformly bounded. Lemma~\ref{Lem2.5} gives \(\|B_\ell\|_2\to0\). For \(n=5\), interpolation with the uniformly bounded \(L^{10/3}\) norm gives \(\|B_\ell\|_{5/2}\to0\). For \(n=4\), the required norm is already \(L^2\). Thus in both dimensions the difference of potentials tends to zero in \(L^{n/2}\), uniformly in the centers.

Set \(c_* =\min\{c_\Hyp,c_\Euc\}>0\). In every dimension, the two weighted limits, their spectral gaps, and \eqref{eq4.8} imply
\begin{align*}
\limsup_{\ell\to\infty}(2^*-1)\int_{\HH^n}(H_\ell+B_\ell)^{2^*-2}\rho_\ell^2\dv &\le(1-c_*)\bigl(\norm{\rho_\Hyp}^2+\|\nabla\rho_\Euc\|_2^2\bigr)\\
&\le1-c_*.
\end{align*}
Therefore \(\liminf_\ell\langle L_{\theta_\ell}\rho_\ell,\rho_\ell\rangle\ge c_*\). Failure of a uniform gap \(c_*/2\) at arbitrarily small scales would contradict this conclusion. This proves \eqref{eq4.4}. Boundedness of the form and the Lax--Milgram theorem give \eqref{eq4.5}.
\end{proof}

\subsection{Quantitative estimates and sharpness}
\label{Sec4.2}

We write \(s=s(\eps)\) and \(I=I(\eps,r)\). Propositions~\ref{Prop2.10} and~\ref{Prop2.11} give
\begin{equation}\label{eq4.9}
\begin{gathered}
s>0,\qquad \dnorm{\cR_\lambda(B)}+\dnorm{\cD\cR_\lambda(B)}\le Cs,\\
s\sim\begin{cases}
\eps^2\log(2/\eps),&n=4,\\
\eps^2,&n\ge5.
\end{cases}
\end{gathered}
\end{equation}
Thus \(s+I\to0\) uniformly in the centers. For \(n\ge6\), Proposition~\ref{Prop3.4} also gives \(s+I\sim\eps^2\).

The residual norm and the three scalar pairings satisfy the following bounds.
\begin{proposition}\label{Prop4.4}
For sufficiently small \(\eps>0\), uniformly for \(r\ge0\),
\begin{equation}\label{eq4.10}
-\langle\cR_\lambda(\Sigma_\theta),Z_0\rangle\ge c(s+I).
\end{equation}
Moreover,
\begin{equation}\label{eq4.11}
\dnorm{\cR_\lambda(\Sigma_\theta)}+\dnorm{L_\theta Z_0}
\le C\begin{cases}
s+I,&n=4,5\text{ or }n\ge7,\\
\eps^2\sqrt{\log(2/\eps)},&n=6,
\end{cases}
\end{equation}
whereas in every dimension
\begin{equation}\label{eq4.12}
|\langle\cR_\lambda(\Sigma_\theta),H\rangle| +|\langle\cR_\lambda(\Sigma_\theta),B\rangle|\le C(s+I).
\end{equation}
\end{proposition}

\begin{proof}
With \(\mathcal M_{z,\xi,\eps}=\Sigma_\theta^{2^*-1}-H^{2^*-1}-B^{2^*-1}\), as in Lemma~\ref{Lem3.5}, the exact identities
\begin{equation*}
\cR_\lambda(\Sigma_\theta)=\cR_\lambda(B)-\mathcal M_{z,\xi,\eps}, \qquad L_\theta Z_0=\cD\cR_\lambda(B)-\cD\mathcal M_{z,\xi,\eps}
\end{equation*}
and Lemma~\ref{Lem3.5} prove \eqref{eq4.11}, using \eqref{eq4.9} and \(\eps^{(n+2)/4}=o(\eps^2)\) for \(n\ge7\). For \(n\ne6\), uniform boundedness of \(\norm H+\norm B\) also proves \eqref{eq4.12}.

For \(n=6\), the nonlinearity is quadratic on positive functions, so \(\mathcal M_{z,\xi,\eps}=2HB\). The identity \(\ip{H}{B}=Q\) gives the pairings with \(H\) and \(B\) directly:
\[
\langle\cR_\lambda(H+B),H\rangle=-Q-I, \qquad \langle\cR_\lambda(H+B),B\rangle =\langle\cR_\lambda(B),B\rangle-2I.
\]
Here \(Q+I=O(\eps^2)\), by Theorem~\ref{Thm3.3} and Proposition~\ref{Prop3.4}. Together with \eqref{eq4.9}, this proves \eqref{eq4.12} in dimension six.

To prove the signed estimate, first suppose \(n=4,5\). For \(h,b\ge0\), define
\[
G(h,b)=(h+b)^{2^*}-h^{2^*}-b^{2^*} -2^*h^{2^*-1}b-2^*hb^{2^*-1}.
\]
After a common isometry sends \(\xi\) to the origin, put \(\delta=\eps/2\), \(z_r=\tanh(r/2)e_1\), and \(h_r=\varrho^{(n-2)/2}B^{\HH}_{z_r,\lambda}\), as before \eqref{eq3.8}. Let
\[
\mathcal A(\eps,r) =\frac1{2^*}\int_{\BB^n}G(h_r,B^{\RR}_{\mathrm{pr},\delta})\,dx.
\]
Since \(\int(B^{\HH}_{z,\lambda})^{2^*-1}B^{\HH}_{\mathrm{pr},\xi,\eps}\dv=Q\), Taylor's formula gives the exact decomposition
\begin{equation}\label{eq4.13}
J_\lambda(H+B)-J_\lambda(H)-J_\lambda(B)=-I(\eps,r)-\mathcal A(\eps,r), \qquad\mathcal A(\eps,r)\ge0.
\end{equation}
For all sufficiently small \(\eps>0\), uniformly in \(r\ge0\), we claim that
\begin{equation}\label{eq4.14}
\mathcal A(\eps,r)+|\cD\mathcal A(\eps,r)| \le C I(\eps,r)
\begin{cases}
\eps(1+|\log\eps|),&n=4,\\
\eps^{1/2},&n=5.
\end{cases}
\end{equation}

Because \(2^*>3\), Taylor's formula on \(b\le h\) and on \(h<b\) gives
\begin{equation}\label{eq4.15}
0\le G(h,b)\le C
\begin{cases}
h^{2^*-2}b^2,&b\le h,\\
b^{2^*-2}h^2,&h<b,
\end{cases}
\end{equation}
and
\begin{equation}\label{eq4.16}
b|\partial_bG(h,b)|\le C \bigl(h^{2^*-2}b^2+b^{2^*-2}h^2\bigr).
\end{equation}
By Lemma~\ref{Lem2.4}, \(|\cD B^{\RR}_{\mathrm{pr},\delta}|\le C B^{\RR}_{\mathrm{pr},\delta}\). Thus \eqref{eq4.15}--\eqref{eq4.16} reduce both \(\mathcal A\) and \(|\cD\mathcal A|\) to
\begin{equation}\label{eq4.17}
\mathcal T =\int_{\BB^n} \bigl(h_r^{2^*-2}(B^{\RR}_{\mathrm{pr},\delta})^{2}+(B^{\RR}_{\mathrm{pr},\delta})^{2^*-2}h_r^2\bigr)\,dx.
\end{equation}

On \(\{|x|<1/2\}\),
\[
h_r\le CB^{\HH}_{0,\lambda}(r),\qquad B^{\RR}_{\mathrm{pr},\delta}(x)\le C\frac{\delta^{\frac{n-2}{2}}}{(\delta+|x|)^{(n-2)}}.
\]
Since \(2^*-2>1\) and \(B^{\HH}_{0,\lambda}\) is bounded,
\[
(B^{\HH}_{0,\lambda})^{2^*-2}+(B^{\HH}_{0,\lambda})^{2}\le CB^{\HH}_{0,\lambda}.
\]
Direct radial integration gives
\begin{align*}
\int_{|x|<1/2}\bigl((B^{\RR}_{\mathrm{pr},\delta})^{2}+(B^{\RR}_{\mathrm{pr},\delta})^{2^*-2}\bigr)\,dx &\le C\delta^2(1+\log(1/\delta)),&&n=4,\\
\int_{|x|<1/2}\bigl((B^{\RR}_{\mathrm{pr},\delta})^{2}+(B^{\RR}_{\mathrm{pr},\delta})^{2^*-2}\bigr)\,dx &\le C\delta^2,&&n=5.
\end{align*}
Thus the interior part of \(\mathcal T\) has the required bound.

On the annulus \(\mathcal C=\{x\in\BB^n:1/2<|x|<1\}\), use the normalization \(z_r=\tanh(r/2)e_1\). For \(r\ge1\), set \(\tau=e^{-r}\). By \eqref{eq3.25}, \(1-|x|\le|x-e_1|\), and polar coordinates centered at \(e_1\), the two terms in \eqref{eq4.17} have the following bounds. In \(n=4\), they coincide and
\begin{align*}
\int_{\mathcal C}h_r^2(B^{\RR}_{\mathrm{pr},\delta})^{2}\,dx &\le C\delta^2B^{\HH}_{0,\lambda}(r)^2 \int_0^2 \frac{t^{3+4\alpha_\lambda}} {(\tau+t)^{4\gamma_\lambda}}\,dt\\
&\le C\delta^2B^{\HH}_{0,\lambda}(r)^2(1+|\log\tau|) \le C\delta^2B^{\HH}_{0,\lambda}(r).
\end{align*}
Here \(3+4\alpha_\lambda=4\gamma_\lambda-1\), and \(B^{\HH}_{0,\lambda}(r)|\log\tau|\) is bounded. In \(n=5\), since
\[
5+\frac{10}{3}\alpha_\lambda=\frac{10}{3}\gamma_\lambda,
\]
we obtain
\begin{align*}
\int_{\mathcal C}h_r^{4/3}(B^{\RR}_{\mathrm{pr},\delta})^{2}\,dx &\le C\delta^3B^{\HH}_{0,\lambda}(r)^{4/3} \int_0^2 \frac{t^{4+(10/3)\alpha_\lambda}} {(\tau+t)^{(8/3)\gamma_\lambda}}\,dt\le C\delta^3B^{\HH}_{0,\lambda}(r),\\
\int_{\mathcal C}(B^{\RR}_{\mathrm{pr},\delta})^{4/3}h_r^2\,dx &\le C\delta^2B^{\HH}_{0,\lambda}(r)^2 \int_0^2 \frac{t^{4+(10/3)\alpha_\lambda}} {(\tau+t)^{4\gamma_\lambda}}\,dt\\
&\le C\delta^2B^{\HH}_{0,\lambda}(r)^2\tau^{-2\gamma_\lambda/3} \le C\delta^2B^{\HH}_{0,\lambda}(r).
\end{align*}
For \(r\le1\), the same estimates follow from \(B^{\HH}_{0,\lambda}(r)\sim1\) and the boundary estimate in Lemma~\ref{Lem3.2}. Combining the estimates on \(\{x\in\BB^n:|x|\le1/2\}\) and \(\{x\in\BB^n:1/2<|x|<1\}\), and using \(I\sim\delta^{\frac{n-2}{2}}B^{\HH}_{0,\lambda}(r)\), gives
\[
\mathcal T\le C
\begin{cases}
\delta(1+\log(1/\delta))\,I,&n=4,\\
\delta^{1/2}I,&n=5.
\end{cases}
\]
Replacing \(\delta\) by \(\eps/2\) proves \eqref{eq4.14}.

The energy identity \eqref{eq4.13} now yields
\[
-\langle\cR_\lambda(\Sigma_\theta),Z_0\rangle =s(\eps)+\cD I+\cD\mathcal A.
\]
By \eqref{eq4.14}, \(|\cD\mathcal A|=o(I)\) uniformly in \(r\), while Proposition~\ref{Prop3.4} gives \(\cD I\ge cI\). This proves \eqref{eq4.10} in these dimensions.

For \(n=6\), the exact quadratic identity gives
\begin{align*}
-\langle\cR_\lambda(\Sigma_\theta),Z_0\rangle &=s(\eps)+\int_{\HH^6}2HBZ_0\dv=s(\eps)+\cD I\\
&=\mathfrak m_6\bigl(\mu+2B^{\HH}_{0,\lambda}(r)\bigr)\eps^2 +O\!\left(\eps^4\log(2/\eps)\right),
\end{align*}
by \eqref{eq2.33} and \eqref{eq3.18}. The coefficient is bounded below by \(\mu\mathfrak m_6>0\), and \(s+I\sim\eps^2\), which proves the claim.

Finally, for \(n\ge7\), use Lemma~\ref{Lem3.5} and \(\norm{Z_0}\le C\) to obtain
\[
-\langle\cR_\lambda(\Sigma_\theta),Z_0\rangle =s(\eps)+\langle\mathcal M_{z,\xi,\eps},Z_0\rangle \ge c\eps^2-C\eps^{\frac{n+2}{4}}\ge c'\eps^2.
\]
In this range \(I=o(\eps^2)\) uniformly, so the last expression is bounded below by a positive constant times \(s+I\). This completes the proof.
\end{proof}

\begin{lemma}\label{Lem4.5}
For \(\Sigma_\theta=H+B\), \(0<\eps<\eps_0\), and any real-valued \(w\in H^1_\lambda(\HH^n)\), define
\[
\mathcal N_{\Sigma_\theta}(w) =|\Sigma_\theta+w|^{2^*-2}(\Sigma_\theta+w) -\Sigma_\theta^{2^*-1}-(2^*-1)\Sigma_\theta^{2^*-2}w.
\]
Uniformly in the centers and scale,
\begin{equation}\label{eq4.18}
\dnorm{\mathcal N_{\Sigma_\theta}(w)} \le C\bigl(\norm w^2+\norm w^{2^*-1}\bigr).
\end{equation}
In particular, if \(u\) has the decomposition \eqref{eq4.3} and
\[
a=|\nu_\Hyp-1|+|\nu_\Euc-1|, \qquad b=a+\norm\rho, \qquad w=u-\Sigma_\theta,
\]
then the right-hand side is at most \(C(b^2+b^{2^*-1})\).
\end{lemma}

\begin{proof}
The pointwise remainder estimates for \(f(t)=|t|^{2^*-2}t\) give
\[
|\mathcal N_{\Sigma_\theta}(w)|\le C
\begin{cases}
\Sigma_\theta^{2^*-3}|w|^2+|w|^{2^*-1},&n=4,5,\\
|w|^{2^*-1},&n\ge6;
\end{cases}
\]
see \cite{FigalliGlaudo2020}. These estimates hold for real-valued \(w\); H\"older's inequality, the Sobolev embedding and its dual, and the uniform bound for \(\|\Sigma_\theta\|_{2^*}\) give \eqref{eq4.18}. The last assertion follows from \(\norm{u-\Sigma_\theta}\le Cb\).
\end{proof}

In the following estimates, \(a\), \(b\), and \(w=u-\Sigma_\theta\) are as in Lemma~\ref{Lem4.5}.

\begin{lemma}\label{Lem4.6}
If \(b\) and \(\eps\) are sufficiently small, the following estimates hold uniformly in the centers. For \(n=4,5\) or \(n\ge7\),
\begin{align}
a &\le C\bigl(\Gamma_\lambda(u)+s+I+(s+I)\norm\rho +b^2+b^{2^*-1}\bigr),\label{eq4.19}\\
\norm\rho &\le C\bigl(\Gamma_\lambda(u)+s+I+(s+I)a +b^2+b^{2^*-1}\bigr),\label{eq4.20}\\
s+I &\le C\bigl(\Gamma_\lambda(u)+(s+I)b +b^2+b^{2^*-1}\bigr).\label{eq4.21}
\end{align}
For \(n=6\), put \(\omega_\eps=\eps^2\sqrt{\log(2/\eps)}\). Then
\begin{align}
a &\le C\bigl(\Gamma_\lambda(u)+s+I+\omega_\eps\norm\rho +b^2\bigr),\label{eq4.22}\\
\norm\rho &\le C\bigl(\Gamma_\lambda(u)+\omega_\eps +\omega_\eps a+b^2\bigr),\label{eq4.23}\\
s+I &\le C\bigl(\Gamma_\lambda(u)+\omega_\eps b +b^2\bigr).\label{eq4.24}
\end{align}
\end{lemma}

\begin{proof}
Use the common residual expansion
\begin{equation}\label{eq4.25}
\cR_\lambda(u)=\cR_\lambda(\Sigma_\theta)+L_\theta w-\mathcal N_{\Sigma_\theta}(w).
\end{equation}
Testing the linear terms associated with the two coefficients against \(H\) and \(B\) gives the matrix
\[
\mathbf A_\theta=
\begin{pmatrix}
\langle L_\theta H,H\rangle&\langle L_\theta B,H\rangle\\
\langle L_\theta H,B\rangle&\langle L_\theta B,B\rangle
\end{pmatrix}.
\]
The identities
\begin{align}
L_\theta H &=(2-2^*)P^{\HH}_\lambda H -(2^*-1)\bigl(\Sigma_\theta^{2^*-2}-H^{2^*-2}\bigr)H,\notag\\
L_\theta B &=(2-2^*)P^{\HH}_\lambda B+(2^*-1)\cR_\lambda(B) -(2^*-1)\bigl(\Sigma_\theta^{2^*-2}-B^{2^*-2}\bigr)B\label{eq4.26}
\end{align}
show, using Lemma~\ref{Lem3.5}, \eqref{eq4.9}, and \(Q=o(1)\), that
\begin{equation}\label{eq4.27}
\mathbf A_\theta =(2-2^*)
\begin{pmatrix}\norm H^2&0\\0&\mathcal S_{n,2}^{n/2}\end{pmatrix}+o(1).
\end{equation}
The convergence is uniform in \(r\), so \(\mathbf A_\theta^{-1}\) is uniformly bounded.

For every \(\phi\in\cX_\theta\), orthogonality removes the terms \(P^{\HH}_\lambda H\) and \(P^{\HH}_\lambda B\) in \eqref{eq4.26}. Lemma~\ref{Lem3.5} and \eqref{eq4.9} therefore give
\begin{equation}\label{eq4.28}
|\langle L_\theta H,\phi\rangle|+|\langle L_\theta B,\phi\rangle|
\le C\norm\phi\begin{cases}s+I,&n\ne6,\\\omega_\eps,&n=6.\end{cases}
\end{equation}
Testing \eqref{eq4.25} with \(H\) and \(B\), then using \eqref{eq4.12}, \eqref{eq4.27}, \eqref{eq4.28}, and \eqref{eq4.18}, gives \eqref{eq4.19} and \eqref{eq4.22}. In particular, the term involving \(\cR_\lambda(\Sigma_\theta)\) is \(O(s+I)\) even in dimension six.

Restrict \eqref{eq4.25} to \(\cX_\theta\). Apply the inverse estimate \eqref{eq4.5}, the residual bound \eqref{eq4.11}, and \eqref{eq4.28}. This proves \eqref{eq4.20} and \eqref{eq4.23}.

Finally, test with \(Z_0\). The signed pairing \eqref{eq4.10} is at least \(c(s+I)\). Self-adjointness and \eqref{eq4.11} bound \(|\langle L_\theta w,Z_0\rangle|\) by \(C(s+I)b\) if \(n\ne6\), and by \(C\omega_\eps b\) if \(n=6\). The nonlinear remainder gives the remaining terms, proving \eqref{eq4.21} and \eqref{eq4.24}.
\end{proof}

\begin{proof}[Proof of Theorem~\ref{Thm1.1}]
Proposition~\ref{Prop4.2} makes \(b\) arbitrarily small by decreasing the initial neighborhood. The operator \(L_\theta:H^1_\lambda\to H^{-1}_\lambda\) is uniformly bounded, so \eqref{eq4.11}, \eqref{eq4.18}, and \eqref{eq4.25} also give \(\Gamma_\lambda(u)<1\) after decreasing the neighborhood and \(\eps_0\).

Suppose first that \(n\ne6\). Since \(b<1\), adding \eqref{eq4.19} and \eqref{eq4.20} and substituting \eqref{eq4.21} gives
\[
b+s+I\le C\Gamma_\lambda(u) +C\bigl((s+I)b+b^2+b^{2^*-1}\bigr).
\]
The last terms are bounded by
\[
C\bigl(b+b^{4/(n-2)}\bigr) (b+s+I).
\]
Their coefficient tends to zero for each fixed \(n\), so absorption yields
\[
b+s+I\le C\Gamma_\lambda(u),\qquad n\ne6.
\]
Since \(s>0\), this also gives \(\Gamma_\lambda(u)>0\).

For \(n=6\), \(s+I\sim\eps^2\) and \(\omega_\eps=\eps^2\sqrt{\log(2/\eps)}\). Adding \eqref{eq4.22} and \eqref{eq4.23}, then absorbing \(C(\omega_\eps b+b^2)\), gives
\begin{equation}\label{eq4.29}
b\le C\bigl(\Gamma_\lambda(u)+\omega_\eps\bigr).
\end{equation}
Consequently,
\[
\omega_\eps b+b^2 \le C\bigl(\Gamma_\lambda(u)^2 +\omega_\eps\Gamma_\lambda(u)+\omega_\eps^2\bigr).
\]
The first two terms are at most \(C\Gamma_\lambda(u)\) for small neighborhood and scale. Since \(\omega_\eps^2\le C(s+I)^2\log(2/(s+I))\), substitution in \eqref{eq4.24} gives
\[
s+I\le C\Gamma_\lambda(u) +C(s+I)^2\log\frac2{s+I}.
\]
The uniform limit \((s+I)\log(2/(s+I))\to0\) permits absorption, yielding
\begin{equation}\label{eq4.30}
s+I\le C\Gamma_\lambda(u).
\end{equation}
In particular, \(s>0\) gives \(\Gamma_\lambda(u)>0\). We first control \(s+I\) in this way because \(\omega_\eps/(s+I)\sim\sqrt{\log(2/\eps)}\) grows as \(\eps\to0\). The function \(t\mapsto t\sqrt{\log(2/t)}\) is increasing for small positive \(t\). Thus \eqref{eq4.30}, \(s+I\sim\eps^2\), and \eqref{eq4.29} imply
\[
\omega_\eps+b \le C\Gamma_\lambda(u)\sqrt{\log\frac2{\Gamma_\lambda(u)}}.
\]
Returning to \eqref{eq4.22}, we have
\[
\omega_\eps\norm\rho+b^2 \le C\Gamma_\lambda(u)^2\log\frac2{\Gamma_\lambda(u)} \le C\Gamma_\lambda(u),
\]
so
\[
a+s+I\le C\Gamma_\lambda(u), \qquad \norm\rho\le C\Psi_6(\Gamma_\lambda(u)).
\]

In every dimension, choosing \(\Sigma_\theta=H+B\) in the comparison family gives
\[
\dist_{H^1_\lambda}(u,\cM_{\mix}(\eps_0)) \le\norm{u-\Sigma_\theta}\le Cb.
\]
This proves both estimates and \(0<\Gamma_\lambda(u)<1\).
\end{proof}

To estimate distance to the full family \(\cM_\mix(\eps_0)\), we must place every sufficiently close competitor in the chart of Proposition~\ref{Prop4.2}. The next lemma does this with the competing scale ranging throughout \((0,\eps_0)\).

For \(\theta=(z,\eps,\xi)\) and \(\theta'=(z',\eps',\xi')\), set
\[
d_{\mathrm{sc}}(\theta,\theta') =d_g(z,z')+\left|\log\frac{\eps'}{\eps}\right| +\frac{d_g(\xi,\xi')}{\eps}.
\]

\begin{lemma}\label{Lem4.7}
For every \(\eta_{\mathrm{par}}>0\) there exist \(\bar{\eps}(\eta_{\mathrm{par}})>0\) and \(\eta(\eta_{\mathrm{par}})>0\), depending only on \(n,\lambda,\eta_{\mathrm{par}}\), such that
\[
0<\eps<\bar{\eps}(\eta_{\mathrm{par}}),\qquad0<\eps'<\eps_0,\qquad\norm{\Sigma_\theta-\Sigma_{\theta'}}<\eta(\eta_{\mathrm{par}})
\]
imply
\[
d_{\mathrm{sc}}(\theta,\theta')<\eta_{\mathrm{par}}.
\]
Moreover, there are \(r_0,c_0,C_0>0\), depending only on \(n,\lambda\), such that, whenever
\[
0<\eps,\eps'<\eps_0,\qquad d_{\mathrm{sc}}(\theta,\theta')<r_0,
\]
\[
c_0d_{\mathrm{sc}}(\theta,\theta') \le\norm{\Sigma_\theta-\Sigma_{\theta'}} \le C_0d_{\mathrm{sc}}(\theta,\theta').
\]
\end{lemma}

\begin{proof}
We prove the first assertion by contradiction. Suppose that, for some \(\eta_{\mathrm{par}}>0\), there are
\[
\theta_\ell=(z_\ell,\eps_\ell,\xi_\ell),\qquad\theta_\ell'=(z_\ell',\eps_\ell',\xi_\ell'),
\]
such that
\begin{equation}\label{eq4.31}
\eps_\ell\to0,\qquad\norm{\Sigma_{\theta_\ell}-\Sigma_{\theta_\ell'}}\to0,\qquad d_{\mathrm{sc}}(\theta_\ell,\theta_\ell')\ge\eta_{\mathrm{par}}.
\end{equation}
Put \(\delta_\ell=\eps_\ell/2\) and \(\delta_\ell'=\eps_\ell'/2\). We use the rescaling \(\mathscr S_{\xi_\ell,\eps_\ell}\) defined before Proposition~\ref{Prop4.3}. Lemma~\ref{Lem2.5} gives
\begin{equation}\label{eq4.32}
\mathscr S_{\xi_\ell,\eps_\ell} B^{\HH}_{\mathrm{pr},\xi_\ell,\eps_\ell}\to B^{\RR}_{1} \quad\text{strongly in }\dot H^1(\RR^n).
\end{equation}
Every hyperbolic profile is uniformly smooth and bounded after recentering. Hence, on every fixed Euclidean ball, the gradients converge strongly in \(L^2\). We denote this local Dirichlet convergence by \(\dot H^1_{\mathrm{loc}}(\RR^n)\):
\begin{equation}\label{eq4.33}
\mathscr S_{\xi_\ell,\eps_\ell} B^{\HH}_{z_\ell,\lambda}\to0,\qquad\mathscr S_{\xi_\ell,\eps_\ell} B^{\HH}_{z_\ell',\lambda}\to0 \quad\text{strongly in }\dot H^1_{\mathrm{loc}}(\RR^n).
\end{equation}
The rescaled profiles also have uniformly bounded Dirichlet norms. Hence these local limits imply weak convergence to zero in \(\dot H^1(\RR^n)\).

We use the following parameter dichotomy for projected bubbles. After passing to a subsequence, exactly one of the following alternatives holds:
\par\smallskip
\noindent\textup{(i)}
\[
\eps_\ell'\to0,\qquad\frac{\delta_\ell'}{\delta_\ell}\to\sigma\in(0,\infty),\qquad\frac{d_g(\xi_\ell,\xi_\ell')}{\eps_\ell}=O(1).
\]
There are scaled centers \(y_\ell\to y_0\) such that
\begin{equation}\label{eq4.34}
\mathscr S_{\xi_\ell,\eps_\ell} B^{\HH}_{\mathrm{pr},\xi_\ell',\eps_\ell'} \to B^{\RR}_{y_0,\sigma} \quad\text{strongly in }\dot H^1(\RR^n), \qquad|y_\ell|\sim\frac{d_g(\xi_\ell,\xi_\ell')}{\eps_\ell}.
\end{equation}
\par\smallskip
\noindent\textup{(ii)}
\[
\mathscr S_{\xi_\ell,\eps_\ell} B^{\HH}_{\mathrm{pr},\xi_\ell',\eps_\ell'}\weakto0 \quad\text{in }\dot H^1(\RR^n).
\]
This occurs if the scale ratio tends to \(0\) or \(\infty\), if the scaled centers diverge while the ratio stays in a compact subset of \((0,\infty)\), or if \(\liminf_\ell\eps_\ell'>0\).

To prove this dichotomy, let \(h_\ell\) be the image of \(\xi_\ell'\) under the isometry sending \(\xi_\ell\) to the origin. First suppose that \(\delta_\ell'\to0\). As in Section~\ref{Sec2}, \(\widetilde{B^{\RR}_{\mathrm{pr},d}}\) denotes the zero extension, which belongs to \(\dot H^1(\RR^n)\). By \eqref{eq2.19} and scale invariance, as \(d\downarrow0\),
\[
\|\nabla(\widetilde{B^{\RR}_{\mathrm{pr},d}}-B^{\RR}_{d})\|_{L^2(\RR^n)}\to0.
\]
The conformal action preserves the Dirichlet norm on the whole Euclidean space, and the approximation in \eqref{eq2.21} is uniform in the center, as shown immediately after Lemma~\ref{Lem2.6}. Thus, uniformly in \(h_\ell\),
\[
\mathscr S_{\xi_\ell,\eps_\ell} B^{\HH}_{\mathrm{pr},\xi_\ell',\eps_\ell'} =\delta_\ell^{\frac{n-2}{2}}\bigl(\mathsf C_{\frac{n-2}{2}}(h_\ell)B^{\RR}_{\delta_\ell'}\bigr)(\delta_\ell\,\cdot) +o_{\dot H^1}(1).
\]

With the convention for \(\mathsf C_{\frac{n-2}{2}}\) used in Lemma~\ref{Lem2.6}, completing the square gives the exact formula
\[
\mathsf C_{\frac{n-2}{2}}(h)B^{\RR}_d=B^{\RR}_{y_h,d_h},\qquad d_h=\frac{d(1-|h|^2)}{1+d^2|h|^2},\qquad y_h=\frac{(1+d^2)h}{1+d^2|h|^2}.
\]
Consequently, on setting
\[
\sigma_\ell=\frac{\delta_\ell'(1-|h_\ell|^2)} {\delta_\ell(1+(\delta_\ell')^2|h_\ell|^2)},\qquad y_\ell=\frac{(1+(\delta_\ell')^2)h_\ell} {\delta_\ell(1+(\delta_\ell')^2|h_\ell|^2)},
\]
the preceding approximation is \(B^{\RR}_{y_\ell,\sigma_\ell}+o_{\dot H^1}(1)\).

The Talenti family converges weakly to zero whenever its parameters leave every compact subset of \((0,\infty)\times\RR^n\); this follows directly from its explicit formula by testing the equation against compactly supported functions. Thus a nonzero weak limit requires a bounded center and a scale in a compact subinterval of \((0,\infty)\).

Suppose now that the weak limit is nonzero. The exact center formula and \((1+(\delta_\ell')^2)/(1+(\delta_\ell')^2|h_\ell|^2)\ge1\) imply \(|h_\ell|\le\delta_\ell|y_\ell|=O(\delta_\ell)\). Hence
\[
\sigma_\ell=\frac{\delta_\ell'}{\delta_\ell}(1+o(1)),\qquad y_\ell=\frac{h_\ell}{\delta_\ell}(1+o(1)).
\]
In particular, \(\delta_\ell'/\delta_\ell\) lies in a compact subinterval of the positive real axis. Moreover,
\[
d_g(\xi_\ell,\xi_\ell')=2\operatorname{arctanh}|h_\ell|, \qquad\frac{d_g(\xi_\ell,\xi_\ell')}{\eps_\ell} =\frac{\operatorname{arctanh}|h_\ell|}{\delta_\ell} \sim|y_\ell|.
\]
After extraction, \(\sigma_\ell\to\sigma\in(0,\infty)\) and \(y_\ell\to y_0\). Continuity of translations and dilations in \(\dot H^1\) gives \eqref{eq4.34}. Conversely, the exact formulas show that each of the first two cases listed in (ii), with \(\delta_\ell'\to0\), precludes such parameter compactness and hence gives weak convergence to zero.

It remains to treat \(\liminf_\ell\delta_\ell'>0\). For any fixed \(d_*>0\), on the scale interval \([d_*,\eps_0/2]\), the profile estimates give a uniform bound on \(B^{\RR}_{\mathrm{pr},d}\) away from the boundary and \(B^{\RR}_{\mathrm{pr},d}(|x|)\le C(1-|x|)^{\alpha_\lambda}\) on the annulus \(\{x\in\BB^n:1/2<|x|<1\}\). The radial equation
\[
-|x|^{n-1}(B^{\RR}_{\mathrm{pr},d})'(|x|) =\int_0^{|x|} s^{n-1}\bigl(B^{\RR}_{d}(s)^{2^*-1}+\mu\varrho(s)^2B^{\RR}_{\mathrm{pr},d}(s)\bigr)\,ds
\]
then gives \(|(B^{\RR}_{\mathrm{pr},d})'(|x|)|\le C(1-|x|)^{\alpha_\lambda-1}\) on the annulus \(\{x\in\BB^n:1/2<|x|<1\}\), with uniform derivative bounds on \(\{x\in\BB^n:|x|\le1/2\}\). Thus
\[
\sup_{d\in[d_*,\eps_0/2]} \bigl(\|B^{\HH}_{\mathrm{pr},0,2d}\|_\infty+\|\nabla_gB^{\HH}_{\mathrm{pr},0,2d}\|_\infty\bigr)<\infty.
\]
These bounds hold uniformly under arbitrary hyperbolic isometries. Since the conformal factor and its derivatives are bounded on each fixed ball in the rescaled coordinates, the rescaled second profiles tend to zero strongly in \(H^1_{\mathrm{loc}}(\RR^n)\). Their global Dirichlet norms are uniformly bounded by \eqref{eq2.3} and \eqref{eq2.37}. Density gives weak convergence to zero in \(\dot H^1(\RR^n)\). This proves the dichotomy.

Applying \(\mathscr S_{\xi_\ell,\eps_\ell}\) to \eqref{eq4.31}, and using \eqref{eq4.32}--\eqref{eq4.33}, shows that the weak limit in \eqref{eq4.34} must equal \(B^{\RR}_{1}\). The Talenti parametrization is injective, so
\[
\sigma=1,\qquad y_0=0.
\]
Since this holds along every subsequence,
\begin{equation}\label{eq4.35}
\frac{\eps_\ell'}{\eps_\ell}\to1,\qquad\frac{d_g(\xi_\ell,\xi_\ell')}{\eps_\ell}\to0.
\end{equation}
The first derivative bounds in \eqref{eq2.23}, integrated along a parameter path satisfying \eqref{eq4.35}, now give
\begin{equation*}
\norm{B^{\HH}_{\mathrm{pr},\xi_\ell,\eps_\ell}-B^{\HH}_{\mathrm{pr},\xi_\ell',\eps_\ell'}}\to0.
\end{equation*}
Consequently,
\begin{equation}\label{eq4.36}
\norm{B^{\HH}_{z_\ell,\lambda}-B^{\HH}_{z_\ell',\lambda}}\to0.
\end{equation}

It remains to identify the hyperbolic centers. The profile distance depends continuously only on the distance between the centers, is positive for distinct centers, and tends to \(\sqrt2\norm{B^{\HH}_{0,\lambda}}\) as that distance tends to infinity, since escaping hyperbolic translates converge weakly to zero. Hence
\[
\inf_{d_g(z,w)\ge s}\norm{B^{\HH}_{z,\lambda}-B^{\HH}_{w,\lambda}}>0 \qquad\text{for every }s>0.
\]
Equation \eqref{eq4.36} yields \(d_g(z_\ell,z_\ell')\to0\). Together with \eqref{eq4.35}, this contradicts \eqref{eq4.31}, proving the first assertion.

For the local estimate, let \(v\) be the increment from \(\theta\) to \(\theta'\) in the scaled chart introduced before Proposition~\ref{Prop4.2}. Let \(T_\theta\) denote the differential at \(\theta\) of the map \(\vartheta\mapsto\Sigma_\vartheta\), expressed in this chart, and set
\[
\mathcal E_\theta(v) =\Sigma_{\theta'}-\Sigma_\theta-T_\theta v.
\]
The uniform second derivative bounds give
\[
\norm{\mathcal E_\theta(v)}\le C|v|^2.
\]
The tangent space to the comparison family is contained in \(\cK_\theta\). Hence Lemma~\ref{Lem4.1} gives
\[
\norm{T_\theta v}\ge c|v|.
\]
After decreasing \(r_0\), this proves the lower bound. The upper bound follows from the bounds for first derivatives in \eqref{eq2.23}, and \(|v|\) is uniformly equivalent to \(d_{\mathrm{sc}}(\theta,\theta')\) in this chart.
\end{proof}

\begin{proposition}\label{Prop4.8}
Let \(n=4,5\) or \(n\ge7\). With the constants of Theorem~\ref{Thm1.1}, there are real-valued functions \(u_\ell\) satisfying \eqref{eq1.2} for some \(z_{*,\ell},\xi_{*,\ell}\in\HH^n\) and \(0<\eps_{*,\ell}<\eps_0/2\), such that
\[
0<\Gamma_\lambda(u_\ell)\longrightarrow0, \qquad \dist_{H^1_\lambda}(u_\ell,\cM_{\mix}(\eps_0)) \ge c\Gamma_\lambda(u_\ell).
\]
Thus the linear rate cannot be replaced uniformly by \(o(t)\) as \(t\downarrow0\).
\end{proposition}

\begin{proof}
Choose \(\theta_\ell=(0,\eps_\ell,0)\) with \(\eps_\ell\to0\), and set \(\Sigma_\ell=\Sigma_{\theta_\ell}\). Then \(\dnorm{\cR_\lambda(\Sigma_\ell)}\to0\). By Proposition~\ref{Prop4.3}, any real-valued unit \(\phi_\ell\in\cX_{\theta_\ell}\) satisfies
\[
\|L_{\theta_\ell}\phi_\ell\|_{H^{-1}_\lambda} \ge\langle L_{\theta_\ell}\phi_\ell,\phi_\ell\rangle\ge c.
\]
The reverse bound is uniform as well, by Sobolev's inequality and the uniform bound for \(\Sigma_{\theta_\ell}^{2^*-2}\) in \(L^{n/2}\). Take \(t_\ell\downarrow0\) so slowly that \(\dnorm{\cR_\lambda(\Sigma_\ell)}=o(t_\ell)\), and set \(u_\ell=\Sigma_\ell+t_\ell\phi_\ell\). The estimate for the power nonlinearity \eqref{eq4.18} gives
\[
\cR_\lambda(u_\ell) =\cR_\lambda(\Sigma_\ell)+t_\ell L_{\theta_\ell}\phi_\ell+O_{H^{-1}_\lambda}(t_\ell^2+t_\ell^{2^*-1}).
\]
Since \(2^*-1>1\), both \(t_\ell^2\) and \(t_\ell^{2^*-1}\) are \(o(t_\ell)\). Consequently \(\Gamma_\lambda(u_\ell)\sim t_\ell\).

We prove the distance lower bound using Lemma~\ref{Lem4.7}. By construction,
\[
\dist(u_\ell,\cM_{\mix}(\eps_0))\le\norm{u_\ell-\Sigma_\ell}=t_\ell.
\]
Suppose that, for some \(\Sigma_{\theta_\ell'}\in\cM_{\mix}(\eps_0)\),
\[
\norm{u_\ell-\Sigma_{\theta_\ell'}}<\frac{t_\ell}{2}.
\]
Then
\[
\norm{\Sigma_{\theta_\ell'}-\Sigma_{\theta_\ell}}<\frac{3t_\ell}{2}\to0.
\]
Since \(\eps_\ell\to0\), Lemma~\ref{Lem4.7} shows that \(\theta_\ell'\) lies in the uniform scaled chart centered at \(\theta_\ell\). Thus
\[
\Sigma_{\theta_\ell'}-\Sigma_{\theta_\ell} =T_{\theta_\ell}v_\ell+\mathcal E_\ell,\qquad|v_\ell|\le Ct_\ell,\qquad\norm{\mathcal E_\ell}\le C|v_\ell|^2.
\]
Because the tangent space to the comparison family is contained in \(\cK_{\theta_\ell}\) and \(\phi_\ell\in\cX_{\theta_\ell}\),
\[
\ip{T_{\theta_\ell}v_\ell}{\phi_\ell}=0.
\]
Pairing \(u_\ell-\Sigma_{\theta_\ell'}\) with the unit vector \(\phi_\ell\) gives
\[
\norm{u_\ell-\Sigma_{\theta_\ell'}} \ge t_\ell-\norm{\mathcal E_\ell} \ge t_\ell-Ct_\ell^2>\frac{t_\ell}{2}
\]
for large \(\ell\), a contradiction. Hence
\[
\frac{t_\ell}{2}\le\dist(u_\ell,\cM_{\mix}(\eps_0))\le t_\ell.
\]
Together with \(\Gamma_\lambda(u_\ell)\sim t_\ell\), this proves the assertion. The neighborhood condition holds for large \(\ell\), since \(\norm{u_\ell-\Sigma_\ell}=t_\ell\to0\) and \(\eps_\ell<\eps_0/2\).
\end{proof}

The six-dimensional rate requires a different approach: a positive correction cancels the leading mixed residual while retaining a larger orthogonal component.

\begin{proposition}\label{Prop4.9}
Assume \(n=6\) and \eqref{eq1.1}. There are positive functions \(u_\eps\), with both centers at the origin, such that
\begin{equation}\label{eq4.37}
\Gamma_\lambda(u_\eps)\sim\eps^2,\qquad \dist_{H^1_\lambda}(u_\eps,\cM_{\mix}(\eps_0)) \sim\eps^2\sqrt{\log(2/\eps)}.
\end{equation}
They satisfy the neighborhood hypothesis of Theorem~\ref{Thm1.1} for every sufficiently small \(\eps\). Thus its distance bound cannot be replaced uniformly by any rate \(o(t\sqrt{\log(2/t)})\) as \(t\downarrow0\).
\end{proposition}

\begin{proof}
Fix \(z=\xi=0\), so \(H=B^{\HH}_{0,\lambda}\), \(B=B^{\HH}_{\mathrm{pr},0,\eps}\), and define
\begin{equation}\label{eq4.38}
\psi_\eps=(P^{\HH}_\lambda)^{-1}(2HB), \qquad u_\eps=H+B+\psi_\eps.
\end{equation}
The inverse exists by coercivity of the energy form. Testing the equation with the negative part shows that \(\psi_\eps\ge0\). Uniqueness and rotation invariance show that it is radial. Hence \(u_\eps>0\).

Let \(\psi^{\RR}_\eps=\varrho^2\psi_\eps\) be its Euclidean representative and recall \(h_0=\varrho^2H\), \(\delta=\eps/2\). Then
\[
\mathcal P^{\RR}_\mu\psi^{\RR}_\eps =2h_0B^{\RR}_{\mathrm{pr},\delta}.
\]
On \(\{|x|\le3/4\}\), the fixed function \(h_0\) is bounded above and below by positive constants. With \(c_\delta=B^{\RR}_\delta(1)\), Lemma~\ref{Lem2.3} and \(B^{\RR}_{\mathrm{pr},\delta}\ge B^{\RR}_\delta-c_\delta\) give
\[
B^{\RR}_{\mathrm{pr},\delta}(x) \sim\frac{\delta^2}{(\delta+|x|)^4} \qquad (|x|\le3/4).
\]
The Green kernel bounds in Lemma~\ref{Lem2.2} therefore imply
\begin{equation}\label{eq4.39}
\begin{aligned}
\psi^{\RR}_\eps(x)&\sim\frac{\delta^2}{(\delta+|x|)^2}, &&|x|\le1/2,\\
0\le\psi^{\RR}_\eps(x)&\le C\delta^2(1-|x|)^{\alpha_\lambda}, &&1/2<|x|<1.
\end{aligned}
\end{equation}
We first prove the interior estimate. For \(\delta\le|x|\le1/2\), split the radial kernel integral at \(s=|x|\). Its two interior parts are bounded by
\[
C|x|^{-4}\int_0^{|x|}\frac{\delta^2s^5}{(\delta+s)^4}\,ds \le C\delta^2|x|^{-2}, \qquad C\delta^2\int_{|x|}^{3/4}s^{-3}\,ds \le C\delta^2|x|^{-2}.
\]
For \(|x|<\delta\), both parts are \(O(1)\). The lower bound follows by integrating over \(|x|<s<3|x|/2\) when \(|x|\ge\delta\), or over \(\delta<s<2\delta\) when \(|x|<\delta\). The contribution from \(3/4<s<1\) is \(O(\delta^2)\), by \eqref{eq2.7} and the boundary bounds for \(h_0\) and the projected bubble. On the annulus \(\{x\in\BB^6:1/2<|x|<1\}\), these bounds give \(2h_0B^{\RR}_{\mathrm{pr},\delta}\le C\delta^2(1-|x|)^{2\alpha_\lambda}\). Comparison with \(C\delta^2(1-|x|^2)^{\alpha_\lambda}\), using the barrier computation in the proof of Lemma~\ref{Lem2.3} and the value at \(|x|=1/2\), proves the second line of \eqref{eq4.39}.

Testing the correction equation gives
\begin{equation}\label{eq4.40}
\norm{\psi_\eps}^2 =\int_{\BB^6}2h_0B^{\RR}_{\mathrm{pr},\delta}\psi^{\RR}_\eps\,dx \sim\delta^4\log(2/\delta).
\end{equation}
Indeed, on \(2\delta<s<1/4\) the radial integrand is comparable to \(\delta^4/s\). The regions \(s<2\delta\), \(1/4<s<1/2\), and the exterior annulus contribute at most \(C\delta^4\), by \eqref{eq4.39}. Thus \(\norm{\psi_\eps}\sim\eps^2\sqrt{\log(2/\eps)}\to0\).

Since all terms in \eqref{eq4.38} are nonnegative, the residual identity is exactly
\begin{equation*}
\cR_\lambda(u_\eps)=\cR_\lambda(B)-2(H+B)\psi_\eps-\psi_\eps^2.
\end{equation*}
On \(\{|x|\le1/2\}\), the Euclidean representatives of \(H\psi_\eps\) and \(B\psi_\eps\) are bounded by
\[
\frac{C\delta^2}{(\delta+|x|)^2}, \qquad\frac{C\delta^4}{(\delta+|x|)^6},
\]
respectively. Direct radial integration gives an \(L^{3/2}\) bound \(C\delta^2\) for each. On \(\{1/2<|x|<1\}\), their representatives are bounded by \(C\delta^2(1-|x|)^{2\alpha_\lambda}\) and \(C\delta^4(1-|x|)^{2\alpha_\lambda}\), respectively. Critical conformal invariance and Sobolev duality yield
\[
\dnorm{H\psi_\eps}+\dnorm{B\psi_\eps}\le C\eps^2, \qquad \dnorm{\psi_\eps^2}\le C\norm{\psi_\eps}^2 =O\!\left(\eps^4\log(2/\eps)\right).
\]
Together with \eqref{eq4.9}, this proves \(\Gamma_\lambda(u_\eps)\le C\eps^2\).

For the matching lower bound, use \eqref{eq4.25} with \(w=\psi_\eps\) and test against \(Z_0\). Equations \eqref{eq4.10}, \eqref{eq4.11}, and \eqref{eq4.40} give
\[
c\eps^2\le C\Gamma_\lambda(u_\eps) +C\eps^2\sqrt{\log(2/\eps)}\norm{\psi_\eps} +C\norm{\psi_\eps}^2 \le C\Gamma_\lambda(u_\eps)+C\eps^4\log(2/\eps).
\]
Absorbing the last term proves \(\Gamma_\lambda(u_\eps)\sim\eps^2\).

Let \(P_{\cK_\theta}\) and \(P_{\cX_\theta}\) be the energy-orthogonal projections at \(\theta=(0,\eps,0)\). The defining equation of \(\psi_\eps\) implies
\[
\ip{\psi_\eps}{H}=2Q,\qquad \ip{\psi_\eps}{B}=2I,\qquad \ip{\psi_\eps}{Z_0}=\cD I.
\]
Each pairing is \(O(\eps^2)\). All pairings with \(Y_j\) and \(Z_j\), \(1\le j\le n\), vanish by oddness, since the correction is radial and both centers coincide. Lemma~\ref{Lem4.1} therefore gives
\begin{equation}\label{eq4.41}
\norm{P_{\cK_\theta}\psi_\eps}=O(\eps^2), \qquad \norm{P_{\cX_\theta}\psi_\eps} \sim\eps^2\sqrt{\log(2/\eps)}.
\end{equation}

Energy orthogonality gives
\[
\norm{P_{\cX_\theta}\psi_\eps}^2 =\norm{\psi_\eps}^2 -\norm{P_{\cK_\theta}\psi_\eps}^2 =\norm{\psi_\eps}^2-O(\eps^4).
\]

It remains to compare with every pair in \(\cM_{\mix}(\eps_0)\). If \(\norm{u_\eps-\Sigma_{\theta'}}\le\norm{\psi_\eps}\), then \(\norm{\Sigma_{\theta'}-\Sigma_\theta}\le2\norm{\psi_\eps}\to0\). Lemma~\ref{Lem4.7} puts \(\theta'\) in the fixed scaled chart around \(\theta\). Its local estimate and the second derivative bounds give
\[
\Sigma_{\theta'}-\Sigma_\theta=T_\theta v+\mathcal E_\theta(v), \quad T_\theta v\in\cK_\theta, \quad |v|\le C\norm{\psi_\eps}, \quad\norm{\mathcal E_\theta(v)}\le C\norm{\psi_\eps}^2.
\]
By \eqref{eq4.40},
\[
\frac{\norm{\psi_\eps}^2} {\eps^2\sqrt{\log(2/\eps)}} =O\!\left(\eps^2\sqrt{\log(2/\eps)}\right)\longrightarrow0.
\]
Thus the second-order error is negligible at the scale of the orthogonal correction. Projecting onto \(\cX_\theta\) and applying \eqref{eq4.41} yields
\[
\norm{u_\eps-\Sigma_{\theta'}} \ge\norm{P_{\cX_\theta}\psi_\eps}-C\norm{\psi_\eps}^2 \ge c\eps^2\sqrt{\log(2/\eps)}.
\]
If \(\norm{u_\eps-\Sigma_{\theta'}}>\norm{\psi_\eps}\), the same lower bound follows directly from \eqref{eq4.40}. Choosing \(\Sigma_\theta\) gives the upper bound. This proves \eqref{eq4.37} and the claimed sharpness.
\end{proof}
\section{Positive profiles and global mixed stability}
\label{Sec5}

Throughout this section, \(n\ge4\) and \(\lamc<\lambda<\lambdab\) are fixed. We classify the profiles of nonnegative Palais--Smale sequences at arbitrary levels and replace localized Euclidean profiles by the projected bubbles from Section~\ref{Sec2}. Their energy and mass determine the profile numbers. This yields Theorem~\ref{Thm1.2}, its resonance obstruction, and the six-dimensional sharpness result. We also realize every pair of profile numbers. Theorem~\ref{Thm1.3}, which uses both energy and mass constraints, is proved in Section~\ref{Sec6}.

\subsection{Positive profiles, projection, and energy splitting}
\label{Sec5.1}

\begin{lemma}\label{Lem5.1}
Let \(n\ge4\), let \(\Omega\) be a connected open subset of \(\RR^n\) or \(\HH^n\), with its corresponding smooth metric \(g\), and let \(\lambda\in\RR\). If \(0\le u\in H^1_{\mathrm{loc}}(\Omega)\) is a weak solution of
\[
-\Delta_g u-\lambda u=u^{\frac{n+2}{n-2}}\quad\text{in }\Omega,
\]
then \(u\in C^2(\Omega)\). Either \(u\equiv0\) or \(u>0\) throughout \(\Omega\).
\end{lemma}

\begin{proof}
Locally, the Euclidean equation is \(-\Delta v=(V+v^{4/(n-2)})v\) with \(v=u\), \(V=\lambda\); by \eqref{eq2.1}, the hyperbolic equation has the same form with \(v=\varrho^{(n-2)/2}u\), \(V=(\lambda-\lamc)\varrho^2\). Since \(V\) is smooth and \(v\in H^1_{\mathrm{loc}}\), the Brezis--Kato argument \cite{BrezisKato1979}, in the local form of \cite{Struwe2008}, gives \(v\in C^2\), hence \(u\in C^2(\Omega)\). Finally, \(( -\Delta_g+|\lambda|)u\ge0\), so the strong maximum principle gives the alternative.
\end{proof}

The following identity links the energy to the residual:
\begin{equation}\label{eq5.1}
J_\lambda(u)=\frac1n\norm{u}^2 +\frac1{2^*}\langle\cR_\lambda(u),u\rangle.
\end{equation}
Thus, if \(|J_\lambda(u)|\le D\),
\[
\norm{u}^2\le nD+\frac n{2^*}\Gamma_\lambda(u)\norm{u}.
\]
Every Palais--Smale sequence at a level \(c\) is bounded in \(H^1_\lambda\), and \eqref{eq5.1} then gives \(\norm{u_\ell}^2\to nc\) and \(c\ge0\).

We use the profile decomposition of \cite{SandeepTintarev2019}. For Palais--Smale sequences, it yields the global compactness conclusion of \cite[Theorem~3.3]{BhaktaSandeep2012}, as follows. Testing the Palais--Smale condition in each extraction chart gives the limiting equations by local compactness; the Euclidean zeroth-order term has the factor \(\eps_\ell^2\). The Sobolev inequalities and the energy bound in that decomposition make the nonzero profile list finite. Its remainder \(r_\ell\) satisfies \(\|r_\ell\|_{2^*}\to0\) and is asymptotically orthogonal to each reconstructed profile. Testing the Palais--Smale condition against \(r_\ell\) therefore gives
\[
\norm{r_\ell}^2 =\int_{\HH^n}|u_\ell|^{2^*-2}u_\ell r_\ell\dv+o(1)=o(1).
\]
Thus these profiles give the required strong remainder and energy splitting.

The following uniform cutoff comparison permits the replacement of localized Euclidean profiles. For \(\chi\in C_c^\infty(\BB^n)\), put
\[
V_{\eps,\chi}=\varrho^{-\frac{n-2}{2}}\chi B^{\RR}_{\eps/2}.
\]

\begin{lemma}\label{Lem5.2}
Fix \(0<r_0<R_0<1\) and \(C_\chi>0\). For all cutoffs \(\chi\in C_c^\infty(\BB^n)\) satisfying
\[
\chi=1\text{ on }\{|x|<r_0\},\qquad \operatorname{supp}\chi\subset\{|x|<R_0\},\qquad \|\chi\|_{C^1}\le C_\chi,
\]
we have, uniformly over these cutoffs,
\[
\norm{B^{\HH}_{\mathrm{pr},0,\eps}-V_{\eps,\chi}}\to0 \qquad(\eps\downarrow0).
\]
If \(\eps_\ell\downarrow0\), the cutoffs \(\chi_\ell\) satisfy the same bounds, and \(T_\ell\) are hyperbolic isometries, then, with \(\xi_\ell=T_\ell(0)\),
\[
\norm{V_{\eps_\ell,\chi_\ell}\circ T_\ell^{-1} -B^{\HH}_{\mathrm{pr},\xi_\ell,\eps_\ell}}\to0.
\]
\end{lemma}

\begin{proof}
Set \(\delta=\eps/2\). By \eqref{eq2.18} in Lemma~\ref{Lem2.5},
\[
\|\nabla(B^{\RR}_{\mathrm{pr},\delta}-\chi B^{\RR}_\delta)\|_2\to0.
\]
The constants in that proof depend on the cutoff only through \(r_0,R_0,C_\chi\). In particular, the squared gradient term involving the cutoff is bounded by
\[
C\delta^2\int_{r_0/\delta<|y|<R_0/\delta}(B^{\RR}_1(y))^2\,dy =O(\delta^{n-2})\qquad(n\ge4).
\]
Both terms in the difference have zero trace. Conformal covariance and \(\mu>0\) therefore give
\[
\norm{B^{\HH}_{\mathrm{pr},0,\eps}-V_{\eps,\chi}} \le\|\nabla(B^{\RR}_{\mathrm{pr},\delta}-\chi B^{\RR}_\delta)\|_2\to0
\]
uniformly over these cutoffs. Isometry invariance proves the assertion at arbitrary centers.
\end{proof}

We retain the projected bubbles from Section~\ref{Sec2} and the energies \(E_{m,k}(\lambda)\) defined in \eqref{eq1.7}. For Euclidean concentration parameters, the separation condition is
\begin{equation}\label{eq5.2}
\frac{\eps_{i,\ell}}{\eps_{j,\ell}} +\frac{\eps_{j,\ell}}{\eps_{i,\ell}} +\frac{d_g(\xi_{i,\ell},\xi_{j,\ell})^2} {\eps_{i,\ell}\eps_{j,\ell}}\longrightarrow\infty \quad(i\ne j), \qquad \eps_{j,\ell}\longrightarrow0.
\end{equation}
Different hyperbolic profiles satisfy
\begin{equation}\label{eq5.3}
d_g(z_{i,\ell},z_{j,\ell})\longrightarrow\infty \quad(i\ne j).
\end{equation}
The different scales separate hyperbolic and Euclidean profiles for arbitrary relative positions of their centers, including coincident centers.

We now compare a localized positive Euclidean profile with a projected bubble and verify preservation of the separation parameters.

\begin{lemma}\label{Lem5.3}
Every localized positive Euclidean profile in the hyperbolic global compactness decomposition can be replaced by \(B^{\HH}_{\mathrm{pr},\xi_\ell,\eps_\ell}\), with \(\eps_\ell\to0\) and an error tending to zero in \(H^1_\lambda\). For a finite collection of profiles, the replacements preserve parameter separation in the form \eqref{eq5.2}.
\end{lemma}

\begin{proof}
Lemma~\ref{Lem5.1} and \cite[Corollary 8.2(b)]{CaffarelliGidasSpruck1989} identify the positive Euclidean limit as an Aubin--Talenti bubble. Normal and centered ball coordinates give the same reconstruction up to \(o_{H^1_\lambda}(1)\): their rescaled transition maps converge locally in \(C^1\) to a dilation, and compactly supported approximation controls the Dirichlet tail. Absorbing the fixed translation and dilation into the extraction parameters, we may therefore use the representative
\[
E_\ell= \bigl(\varrho^{-\frac{n-2}{2}}\eta B^{\RR}_{c_\ell,d_\ell}\bigr) \circ\Phi_\ell^{-1}, \qquad c_\ell\to c_0\in\BB^n, \qquad d_\ell\to0,
\]
where \(\Phi_\ell\) is a hyperbolic isometry and \(\eta\in C_c^\infty(\BB^n)\) equals one near \(c_0\). Such a fixed cutoff may be used because changing the extraction cutoff affects only a compact region at positive distance from \(c_\ell\), where
\[
|B^{\RR}_{c_\ell,d_\ell}|+|\nabla B^{\RR}_{c_\ell,d_\ell}| \le C d_\ell^{(n-2)/2}.
\]
The resulting change tends to zero in \(H^1_\lambda\). By Lemma~\ref{Lem2.6}, completing the square in the M\"obius formula gives the exact identity
\[
\mathsf C_{\frac{n-2}{2}}(-c_\ell)B^{\RR}_{c_\ell,d_\ell} =B^{\RR}_{\widetilde c_\ell,t_\ell},
\]
where
\[
A_\ell=(1-|c_\ell|^2)^2+d_\ell^2|c_\ell|^2, \qquad t_\ell=\frac{d_\ell(1-|c_\ell|^2)}{A_\ell}, \qquad \widetilde c_\ell=-\frac{d_\ell^2c_\ell}{A_\ell}.
\]
Set \(T_\ell=\Phi_\ell\circ\tau_{c_\ell}\) and \(\eta_\ell=\eta\circ\tau_{c_\ell}\). Then
\[
E_\ell= \bigl(\varrho^{-\frac{n-2}{2}} \eta_\ell B^{\RR}_{\widetilde c_\ell,t_\ell}\bigr) \circ T_\ell^{-1}, \qquad \frac{|\widetilde c_\ell|}{t_\ell} =\frac{d_\ell|c_\ell|}{1-|c_\ell|^2}\longrightarrow0.
\]
The cutoffs \(\eta_\ell\) satisfy the uniform assumptions in Lemma~\ref{Lem5.2}. Hence Sobolev's inequality and continuity of translations in \(\dot H^1(\RR^n)\) give
\[
\begin{aligned}
\bigl\|\nabla\bigl[\eta_\ell (B^{\RR}_{\widetilde c_\ell,t_\ell}-B^{\RR}_{t_\ell})\bigr]\bigr\|_2 &\le C\bigl\|\nabla\bigl[ B^{\RR}_1(\,\cdot-\widetilde c_\ell/t_\ell)-B^{\RR}_1\bigr]\bigr\|_2\longrightarrow0.
\end{aligned}
\]
Conformal covariance and Lemma~\ref{Lem5.2} therefore yield
\[
\norm{E_\ell-B^{\HH}_{\mathrm{pr},\xi_\ell,\eps_\ell}} \longrightarrow0, \qquad \xi_\ell=T_\ell(0)=\Phi_\ell(c_\ell), \qquad \eps_\ell=2t_\ell.
\]
In particular, the concentration center is unchanged, and
\[
\eps_\ell=\varrho(c_\ell)d_\ell\bigl(1+O(d_\ell^2)\bigr).
\]

To compare two concentration parameters, first suppose that \(d_g(\xi_{i,\ell},\xi_{j,\ell})\to0\). A common hyperbolic recentering puts both centers in a fixed compact ball. Let \(x_{j,\ell}\) be the center in this chart and let \(\bar d_{j,\ell}\) be the scale obtained from the differential of the chart change at \(c_{j,\ell}\). Conformality gives
\[
\bar d_{j,\ell} =\frac{\varrho(c_{j,\ell})}{\varrho(x_{j,\ell})}d_{j,\ell}, \qquad \eps_{j,\ell} =\varrho(x_{j,\ell})\bar d_{j,\ell}(1+o(1)).
\]
If \(\Theta_{j,\ell}\) denotes the chart change, then
\[
d_g(0,\Theta_{j,\ell}(0)) \le d_g(0,x_{j,\ell})+d_g(0,c_{j,\ell})\le C.
\]
Compactness of the stabilizer of the origin therefore gives uniform bounds for the chart changes and their inverses on the relevant compact sets. Thus
\[
\begin{aligned}
&\frac{\eps_{i,\ell}}{\eps_{j,\ell}} +\frac{\eps_{j,\ell}}{\eps_{i,\ell}} +\frac{d_g(\xi_{i,\ell},\xi_{j,\ell})^2} {\eps_{i,\ell}\eps_{j,\ell}}\sim \frac{\bar d_{i,\ell}}{\bar d_{j,\ell}} +\frac{\bar d_{j,\ell}}{\bar d_{i,\ell}} +\frac{|x_{i,\ell}-x_{j,\ell}|^2} {\bar d_{i,\ell}\bar d_{j,\ell}}.
\end{aligned}
\]
If instead \(d_g(\xi_{i,\ell},\xi_{j,\ell})\) is bounded below by a positive constant, \eqref{eq5.2} follows directly from \(\eps_{i,\ell},\eps_{j,\ell}\to0\). After passage to subsequences, these alternatives prove the assertion for every pair in a finite collection.
\end{proof}

Successive remainders may change sign. To establish positivity, we identify each extracted profile as a weak or almost-everywhere limit of the original sequence in its extraction coordinates.

\begin{lemma}\label{Lem5.4}
Let \(u_\ell\ge0\), \(J_\lambda(u_\ell)\to c\in\RR\), and \(\Gamma_\lambda(u_\ell)\to0\). Every nonzero hyperbolic or Euclidean profile in its global compactness decomposition is positive.
\end{lemma}

\begin{proof}
Equation~\eqref{eq5.1} gives boundedness in \(H^1_\lambda\), which equals \(H^1\) with equivalent norms. Since \(\HH^n\) is homogeneous with bounded geometry, \cite[Remark~4.2 and (4.8)]{SandeepTintarev2019} applies: each hyperbolic profile is a weak limit of the original sequence after an isometry, and each Euclidean profile is an almost-everywhere limit of positive rescalings of that sequence in its own normal coordinates. The fixed-scale terms vanish under concentration because, for each \(V\in H^1(\HH^n)\), absolute continuity and homogeneity give
\[
\sup_{z\in\HH^n}\int_{B_g(z,r)}|V|^{2^*}\dv\longrightarrow0 \qquad(r\downarrow0).
\]
Here \(B_g(z,r)\) is the hyperbolic ball of radius \(r\) about \(z\). This gives vanishing in \(L^{2^*}_{\mathrm{loc}}\) in every concentrating chart, uniformly in the centers, and almost-everywhere vanishing after passage to a subsequence. Thus all profiles, including the initial weak limit, are nonnegative. They solve the corresponding limiting equations by the preceding Palais--Smale specialization, so Lemma~\ref{Lem5.1} makes every nonzero profile strictly positive.
\end{proof}

By the spectral gap and Cauchy--Schwarz, the mass is Lipschitz on bounded subsets of \(H^1_\lambda\); hence an error tending to zero in the energy norm preserves the limiting mass.

\begin{proof}[Proof of Theorem~\ref{Thm1.4}: decomposition and mass]
Let \(u_\ell\ge0\) be Palais--Smale at a level \(c\). Equation~\eqref{eq5.1} bounds the sequence in \(H^1_\lambda\), so the global compactness theorem applies. By Lemma~\ref{Lem5.4}, every nonzero profile is positive. The classification of Mancini and Sandeep \cite{ManciniSandeep2008} identifies each hyperbolic profile with an isometric copy of \(B^{\HH}_{0,\lambda}\), and the classification of Caffarelli, Gidas, and Spruck \cite{CaffarelliGidasSpruck1989} identifies each Euclidean profile with an Aubin--Talenti bubble. Lemma~\ref{Lem5.3} gives integers \(m,k\in\mathbb N_0\) and parameters satisfying \eqref{eq5.2}--\eqref{eq5.3}, with
\begin{equation}\label{eq5.4}
u_\ell=\sum_{i=1}^m B^{\HH}_{z_{i,\ell},\lambda} +\sum_{j=1}^k B^{\HH}_{\mathrm{pr},\xi_{j,\ell},\eps_{j,\ell}} +\rho_\ell, \qquad \norm{\rho_\ell}\longrightarrow0.
\end{equation}
If the weak limit is nonzero, it is one further hyperbolic profile, whose center has divergent distance from every escaping hyperbolic center.

The classified profiles have energies \(S_{n,\lambda}^{n/2}/n\) and \(\mathcal S_{n,2}^{n/2}/n\). Energy splitting and \eqref{eq5.1} give the energy and energy-norm identities in
\begin{equation}\label{eq5.5}
\begin{gathered}
c=E_{m,k}(\lambda),\qquad \norm{u_\ell}^2\longrightarrow mS_{n,\lambda}^{n/2}+k\mathcal S_{n,2}^{n/2},\\
\|u_\ell\|_2^2\longrightarrow mM_\lambda.
\end{gathered}
\end{equation}

For the mass limit, Lemma~\ref{Lem2.5} and the spectral gap give
\[
\|B^{\HH}_{\mathrm{pr},\xi_{j,\ell},\eps_{j,\ell}}\|_2\to0, \qquad \|\rho_\ell\|_2 \le(\lambdab-\lambda)^{-1/2}\norm{\rho_\ell}\to0.
\]
Compactly supported approximation in \(L^2(\HH^n)\), together with \eqref{eq5.3}, yields
\[
\int_{\HH^n}B^{\HH}_{z_{i,\ell},\lambda} B^{\HH}_{z_{j,\ell},\lambda}\dv\longrightarrow0 \quad(i\ne j).
\]
All remaining cross terms tend to zero by Cauchy--Schwarz. Since \(\|B^{\HH}_{z_{i,\ell},\lambda}\|_2^2=M_\lambda\), expanding the mass in \eqref{eq5.4} proves the last limit in \eqref{eq5.5}.
\end{proof}

\subsection{Mixed energy selection and sharp stability}
\label{Sec5.2}

\begin{lemma}\label{Lem5.5}
For \(n\ge4\), the optimal Sobolev constant \(S_{n,\lambda}\) is strictly decreasing on \((\lamc,\lambdab)\) and extends continuously to \([\lamc,\lambdab]\), with
\[
S_{n,\lamc}=\mathcal S_{n,2},\qquad S_{n,\lambdab}=S_{n,\mathrm{bot}}>0, \qquad S_{n,\mathrm{bot}}<S_{n,\lambda}<\mathcal S_{n,2} \quad(\lamc<\lambda<\lambdab),
\]
where
\[
S_{n,\mathrm{bot}} =\inf_{0\ne v\in C_c^\infty(\HH^n)} \frac{\displaystyle\int_{\HH^n} (|\nabla_gv|_g^2-\lambdab v^2)\dv} {\displaystyle\|v\|_{2^*}^2}.
\]
\end{lemma}

\begin{proof}
For \(\lamc<\lambda_1<\lambda_2<\lambdab\), test the quotient at \(\lambda_2\) with the minimizer \(B=B^{\HH}_{0,\lambda_1}\) from \cite{ManciniSandeep2008}. Then
\[
S_{n,\lambda_2}\le S_{n,\lambda_1} -(\lambda_2-\lambda_1)\frac{\|B\|_2^2}{\|B\|_{2^*}^2} <S_{n,\lambda_1}.
\]
This is a short proof of the strict monotonicity established in \cite{BGKMPartI}. The constant is concave as an infimum of affine functions of \(\lambda\), hence continuous in the interior. At \(\lamc\), conformal covariance identifies the quotient with the Euclidean Dirichlet quotient on \(\BB^n\). Zero extension gives the lower bound \(\mathcal S_{n,2}\), while concentrating cutoff Aubin--Talenti bubbles give the reverse bound. Thus \(S_{n,\lamc}=\mathcal S_{n,2}\). By \eqref{eq2.3},
\[
\bigl(1-4(\lambda-\lamc)\bigr)\mathcal S_{n,2} \le S_{n,\lambda}\le\mathcal S_{n,2},
\]
which proves continuity at \(\lamc\). The endpoint inequality \cite{ManciniSandeep2008} gives \(S_{n,\mathrm{bot}}>0\). Given \(\eta>0\), choose a nonzero \(v_\eta\in C_c^\infty(\HH^n)\) with endpoint quotient below \(S_{n,\mathrm{bot}}+\eta\). Comparison of the quadratic forms gives
\[
S_{n,\mathrm{bot}}\le S_{n,\lambda} \le S_{n,\mathrm{bot}}+\eta +(\lambdab-\lambda) \frac{\|v_\eta\|_2^2}{\|v_\eta\|_{2^*}^2}.
\]
Letting \(\lambda\uparrow\lambdab\) and then \(\eta\downarrow0\) proves continuity there. Strict monotonicity gives the remaining strict inequalities.
\end{proof}

We retain the choice of \(\eps_0\) made before Theorem~\ref{Thm1.2}. Lemma~\ref{Lem5.5} gives \(0<E_{\mix}(\lambda)<2\mathcal S_{n,2}^{n/2}/n\). A resonant integer thus satisfies
\[
3\le m\le1+\left(\frac{\mathcal S_{n,2}}{S_{n,\mathrm{bot}}}\right)^{n/2}.
\]
Strict monotonicity permits at most one parameter for each such integer, so \(\Lambda_{\mix}\) is finite for every fixed \(n\).

\begin{proof}[Proof of Theorem~\ref{Thm1.2}]
Since the profile energies \(S_{n,\lambda}^{n/2}/n\) and \(\mathcal S_{n,2}^{n/2}/n\) are positive, the set
\[
\left\{\frac{mS_{n,\lambda}^{n/2}+k\mathcal S_{n,2}^{n/2}}n: m,k\in\mathbb N_0\right\}
\]
is locally finite. We may therefore choose
\[
0<\eta<\min\left\{\frac{E_{\mix}(\lambda)}2, \frac{\mathcal S_{n,2}^{n/2}-S_{n,\lambda}^{n/2}}{2n}\right\}
\]
so that \(E_{\mix}(\lambda)\) is the only such level in
\[
[E_{\mix}(\lambda)-2\eta,E_{\mix}(\lambda)+2\eta] \subset\left(0,2\mathcal S_{n,2}^{n/2}/n\right).
\]

Suppose there are nonnegative \(u_\ell\) outside the neighborhood in Theorem~\ref{Thm1.1} such that
\[
|J_\lambda(u_\ell)-E_{\mix}(\lambda)|<\eta, \qquad\Gamma_\lambda(u_\ell)\to0.
\]
Equation~\eqref{eq5.1} bounds the sequence in \(H^1_\lambda\). After passage to a subsequence, its energy converges to
\[
c\in[E_{\mix}(\lambda)-\eta,E_{\mix}(\lambda)+\eta] \subset\left(0,2\mathcal S_{n,2}^{n/2}/n\right).
\]
Theorem~\ref{Thm1.4} and the choice of \(\eta\) imply \(c=E_{\mix}(\lambda)\). Since \(c<2\mathcal S_{n,2}^{n/2}/n\), its Euclidean profile count satisfies \(k\le1\). The profile counts therefore satisfy
\[
mS_{n,\lambda}^{n/2}+k\mathcal S_{n,2}^{n/2} =S_{n,\lambda}^{n/2}+\mathcal S_{n,2}^{n/2}, \qquad m\in\mathbb N_0,\quad k\in\{0,1\}.
\]
If \(k=1\), then \(m=1\). If \(k=0\), then \(m\ge3\) and \((m-1)S_{n,\lambda}^{n/2}=\mathcal S_{n,2}^{n/2}\), contrary to \(\lambda\notin\Lambda_{\mix}\). Hence \((m,k)=(1,1)\). The decomposition \eqref{eq5.4} with \((m,k)=(1,1)\) now gives \(\eps_\ell\to0\) and a projected mixed configuration whose distance from \(u_\ell\) tends to zero. Eventually \(\eps_\ell<\eps_0/2\) and \eqref{eq1.2} holds, a contradiction.

Consequently, for some \(\Gamma_0\in(0,1)\), all functions satisfying the stated energy and residual bounds enter the local neighborhood. Theorem~\ref{Thm1.1} gives all the asserted conclusions.
\end{proof}

The energy constraint alone permits separated hyperbolic profiles at resonance. Their mass tends to \(mM_\lambda\), \(m\ge3\), which violates the mass constraint in Theorem~\ref{Thm1.3}.

\begin{lemma}\label{Lem5.6}
Assume \eqref{eq1.1} and let \(\lambda\in\Lambda_{\mix}\). There is a nonnegative sequence \((V_\ell)\subset H^1_\lambda(\HH^n)\) such that
\[
J_\lambda(V_\ell)\to E_{\mix}(\lambda),\qquad\dnorm{\cR_\lambda(V_\ell)}\to0,
\]
whereas
\[
\liminf_{\ell\to\infty} \dist_{H^1_\lambda} \left(V_\ell,\cM_{\mix}(\eps_0)\right)>0.
\]
Consequently, the conclusion of Theorem~\ref{Thm1.2} cannot hold at \(\lambda\) relative to the comparison family \(\cM_{\mix}(\eps_0)\).
\end{lemma}

\begin{proof}
Let \(m\geq3\) be the integer determined by
\[
\frac{S_{n,\lambda}^{n/2}}{\mathcal S_{n,2}^{n/2}}=\frac1{m-1}.
\]
Fix \(o\in\HH^n\) and a unit vector \(\zeta\in T_o\HH^n\), and set
\[
z_{j,\ell}=\exp_o\bigl((j-1)\ell\,\zeta\bigr),\qquad V_\ell=\sum_{j=1}^mB^{\HH}_{z_{j,\ell},\lambda}.
\]
Then \(d_g(z_{i,\ell},z_{j,\ell})\to\infty\) whenever \(i\ne j\).

Compactly supported approximation and H\"older's inequality show that products of translates with divergent centers tend to zero in every applicable H\"older norm. Since \(B^{\HH}_{0,\lambda}\in L^2\cap L^{2^*}\), this gives, with \((2^*)'=2n/(n+2)\) and \(i\ne j\),
\[
\int_{\HH^n}B^{\HH}_{z_{i,\ell},\lambda}B^{\HH}_{z_{j,\ell},\lambda}\dv=o(1),\qquad\int_{\HH^n}(B^{\HH}_{z_{i,\ell},\lambda})^{2^*-1}B^{\HH}_{z_{j,\ell},\lambda}\dv=o(1),
\]
and
\[
\|(B^{\HH}_{z_{i,\ell},\lambda})^{2^*-2}B^{\HH}_{z_{j,\ell},\lambda}\|_{L^{(2^*)'}(\HH^n)}=o(1).
\]
Testing \(P^{\HH}_\lambda B^{\HH}_{z_{i,\ell},\lambda}=(B^{\HH}_{z_{i,\ell},\lambda})^{2^*-1}\) against \(B^{\HH}_{z_{j,\ell},\lambda}\) yields
\[
\ip{B^{\HH}_{z_{i,\ell},\lambda}}{B^{\HH}_{z_{j,\ell},\lambda}} =\int_{\HH^n}(B^{\HH}_{z_{i,\ell},\lambda})^{2^*-1}B^{\HH}_{z_{j,\ell},\lambda}\dv=o(1).
\]

For nonnegative summands, the finite-sum estimate in \cite{FigalliGlaudo2020} holds for every \(\gamma>1\): separating a largest summand \(t_{\max}=t_{i_0}\), the mean value theorem bounds the difference \((\sum_i t_i)^\gamma-\sum_i t_i^\gamma\) by \(C_{\gamma,m}t_{\max}^{\gamma-1}\sum_{i\ne i_0}t_i\). Applying this estimate with \(\gamma=2^*\), and using \(\norm{B^{\HH}_{z_{j,\ell},\lambda}}^2=\int(B^{\HH}_{z_{j,\ell},\lambda})^{2^*}\dv=S_{n,\lambda}^{n/2}\), gives
\[
\norm{V_\ell}^2=mS_{n,\lambda}^{n/2}+o(1),\qquad\int_{\HH^n}V_\ell^{2^*}\dv=mS_{n,\lambda}^{n/2}+o(1).
\]
Consequently,
\[
J_\lambda(V_\ell) =\left(\frac12-\frac1{2^*}\right)mS_{n,\lambda}^{n/2}+o(1) =\frac{mS_{n,\lambda}^{n/2}}{n}+o(1) =E_{\mix}(\lambda)+o(1),
\]
where the last equality follows from \((m-1)S_{n,\lambda}^{n/2}=\mathcal S_{n,2}^{n/2}\).

Since \(P^{\HH}_\lambda V_\ell=\sum_j(B^{\HH}_{z_{j,\ell},\lambda})^{2^*-1}\),
\[
\cR_\lambda(V_\ell)=\sum_{j=1}^m(B^{\HH}_{z_{j,\ell},\lambda})^{2^*-1}-V_\ell^{2^*-1}.
\]
The same power estimate with \(\gamma=2^*-1\) and the product estimates imply
\[
\|\cR_\lambda(V_\ell)\|_{L^{(2^*)'}(\HH^n)} \leq C\sum_{i\ne j} \|(B^{\HH}_{z_{i,\ell},\lambda})^{2^*-2}B^{\HH}_{z_{j,\ell},\lambda}\|_{L^{(2^*)'}(\HH^n)} \to0.
\]
Sobolev duality proves \(\dnorm{\cR_\lambda(V_\ell)}\to0\).

The pairwise \(L^2\) products above also give
\[
\|V_\ell\|_2^2=mM_\lambda+o(1).
\]
By the choice of \(\eps_0\) and the triangle inequality,
\[
\overline M_\lambda:=\sup_{\Sigma\in\cM_{\mix}(\eps_0)}\|\Sigma\|_2^2 \le\Bigl(\sqrt{M_\lambda} +\sup_{0<\eps<\eps_0}\|B^{\HH}_{\mathrm{pr},0,\eps}\|_2\Bigr)^2 <3M_\lambda.
\]
Consequently, the spectral gap gives the quantitative separation
\[
\begin{aligned}
\liminf_{\ell\to\infty} \dist_{H^1_\lambda}(V_\ell,\cM_{\mix}(\eps_0)) &\ge\sqrt{\lambdab-\lambda} \bigl(\sqrt{mM_\lambda}-\sqrt{\overline M_\lambda}\bigr)>0,
\end{aligned}
\]
since \(m\ge3\).
\end{proof}

\begin{corollary}\label{Cor5.7}
Let \(n=6\) and assume \eqref{eq1.1}. The positive functions \(u_\eps\) of Proposition~\ref{Prop4.9} satisfy
\begin{equation}\label{eq5.6}
J_\lambda(u_\eps)\to E_{\mix}(\lambda),\qquad \|u_\eps\|_{L^2(\HH^6)}^2\to M_\lambda.
\end{equation}
Consequently the rate \(\Psi_6(t)=t\sqrt{\log(2/t)}\) cannot be replaced by \(o(\Psi_6(t))\) in Theorem~\ref{Thm1.3}, including at resonant parameters, or in Theorem~\ref{Thm1.2} at its nonresonant parameters.
\end{corollary}

\begin{proof}
Recall the functions \(H,B,\psi_\eps\) in \eqref{eq4.38}, with both centers at the origin. Lemma~\ref{Lem2.5} and the estimates in the proof of Proposition~\ref{Prop4.9} give
\[
\|B\|_{L^2(\HH^6)}\to0,\qquad \norm{\psi_\eps}=O(\eps^2\sqrt{\log(2/\eps)})\to0.
\]
The spectral gap therefore implies \(u_\eps\to H\) in \(L^2(\HH^6)\), which proves the second limit in \eqref{eq5.6}.

Since \(2^*=3\) in dimension six, expanding the energy and using \(\ip{H}{B}=\int H^2B\dv\) yields
\[
J_\lambda(H+B) =J_\lambda(H)+J_\lambda(B)-I(\eps,0).
\]
By Lemma~\ref{Lem2.9}, Proposition~\ref{Prop2.10}, and Proposition~\ref{Prop3.4},
\[
J_\lambda(B)\to\mathcal S_{6,2}^3/6, \qquad I(\eps,0)\to0.
\]
Thus \(J_\lambda(H+B)\to E_{\mix}(\lambda)\). Since \(J_\lambda\) is Lipschitz on bounded subsets of \(H^1_\lambda\), the estimate \(\norm{\psi_\eps}\to0\) proves the energy limit in \eqref{eq5.6}.

For every fixed \(\eta_E,\eta_M\) satisfying \eqref{eq1.5}, these positive functions eventually satisfy \eqref{eq1.6}. By \eqref{eq4.37},
\[
\Gamma_\lambda(u_\eps)\sim\eps^2,\qquad \dist_{H^1_\lambda}(u_\eps,\cM_{\mix}(\eps_0)) \sim\Psi_6\bigl(\Gamma_\lambda(u_\eps)\bigr).
\]
These comparisons exclude any smaller asymptotic rate in Theorem~\ref{Thm1.3}. If \(\lambda\notin\Lambda_{\mix}\), the same sequence eventually satisfies the energy and residual conditions of Theorem~\ref{Thm1.2}, proving its asserted sharpness.
\end{proof}

\subsection{Realization of prescribed profile numbers}
\label{Sec5.3}

Theorem~\ref{Thm1.4} restricts the possible energies and profile numbers of a nonnegative Palais--Smale sequence. We prove the converse by combining concentrating Euclidean bubbles at distinct fixed points with hyperbolic profiles whose centers diverge. The finite-sum power estimate \cite{FigalliGlaudo2020}, in the form used above, controls the nonlinear interactions for every \(n\ge4\).

\begin{proof}[Proof of Theorem~\ref{Thm1.4}: realization]
The zero sequence treats \((m,k)=(0,0)\). Fix a pair with \(m+k\ge1\). Choose distinct points \(\xi_1,\ldots,\xi_k\) and disjoint relatively compact hyperbolic balls around them. In the centered ball coordinates choose \(\chi_j\in C_c^\infty(\BB^n)\), with \(0\le\chi_j\le1\) and \(\chi_j=1\) near zero, so that
\[
b_{j,\ell} =\bigl(\varrho^{-\frac{n-2}{2}} \chi_j B^{\RR}_{\eps_\ell/2}\bigr) \circ\tau_{\xi_j}^{-1}, \qquad \eps_\ell\to 0,
\]
have supports in the chosen balls.

Sobolev and H\"older inequalities imply that \(J_\lambda:H^1_\lambda\to\RR\) and \(\cR_\lambda:H^1_\lambda\to H^{-1}_\lambda\) are Lipschitz on bounded sets. Lemma~\ref{Lem5.2} and Proposition~\ref{Prop2.11} thus give
\[
\begin{gathered}
\norm{b_{j,\ell}-B^{\HH}_{\mathrm{pr},\xi_j,\eps_\ell}}\to0, \qquad \Gamma_\lambda(b_{j,\ell})\to0,\\
\norm{b_{j,\ell}}^2\to\mathcal S_{n,2}^{n/2}, \qquad J_\lambda(b_{j,\ell})\to\frac1n\mathcal S_{n,2}^{n/2}.
\end{gathered}
\]

Fix a unit vector \(\zeta\in T_0\HH^n\), let \(L_\ell\to\infty\), and set
\[
z_{i,\ell}=\exp_0\bigl(iL_\ell\,\zeta\bigr), \qquad H_{i,\ell}=B^{\HH}_{z_{i,\ell},\lambda}, \qquad 1\le i\le m.
\]
These profiles have pairwise divergent centers, converge weakly to zero, and satisfy
\[
\|H_{i,\ell}\|_{L^{2^*}(K)}\to0 \quad\text{for every compact }K\subset\HH^n.
\]
The last assertion follows by approximating \(B^{\HH}_{0,\lambda}\) in \(L^{2^*}\) by compactly supported functions.

List the \(H_{i,\ell}\) and \(b_{j,\ell}\) as \(v_{1,\ell},\ldots,v_{m+k,\ell}\). We claim that
\begin{equation}\label{eq5.7}
\|v_{\alpha,\ell}^{2^*-2}v_{\beta,\ell}\|_{(2^*)'} \longrightarrow0 \qquad(\alpha\ne\beta).
\end{equation}
For two hyperbolic profiles, compactly supported approximation and divergence of the centers prove the claim by H\"older's inequality, since
\[
\frac1{(2^*)'}=\frac{2^*-2}{2^*}+\frac1{2^*}.
\]
Here the map \(v\mapsto v^{2^*-2}\) is continuous from the nonnegative cone of \(L^{2^*}\) to \(L^{n/2}\), including when \(2^*-2<1\). For two Euclidean terms the product is zero by disjoint supports. If \(K\) contains all the supports of \(b_{j,\ell}\), the two mixed orderings satisfy
\[
\begin{aligned}
\|H_{i,\ell}^{2^*-2}b_{j,\ell}\|_{(2^*)'} &\le\|H_{i,\ell}\|_{L^{2^*}(K)}^{2^*-2} \|b_{j,\ell}\|_{2^*}\longrightarrow0,\\
\|b_{j,\ell}^{2^*-2}H_{i,\ell}\|_{(2^*)'} &\le\|b_{j,\ell}\|_{2^*}^{2^*-2} \|H_{i,\ell}\|_{L^{2^*}(K)}\longrightarrow0.
\end{aligned}
\]
This proves \eqref{eq5.7}.

Set \(v_\ell=\sum_{\alpha=1}^{m+k}v_{\alpha,\ell}\). By the finite-sum power estimate with \(\gamma=2^*-1\), Sobolev duality gives
\[
\Gamma_\lambda(v_\ell) \le\sum_\alpha\Gamma_\lambda(v_{\alpha,\ell}) +C\sum_{\alpha\ne\beta} \|v_{\alpha,\ell}^{2^*-2}v_{\beta,\ell}\|_{(2^*)'} \longrightarrow0.
\]
Moreover,
\[
\int_{\HH^n}v_{\alpha,\ell}^{2^*-1}v_{\beta,\ell}\dv \le\|v_{\alpha,\ell}\|_{2^*} \|v_{\alpha,\ell}^{2^*-2}v_{\beta,\ell}\|_{(2^*)'} \longrightarrow0 \quad(\alpha\ne\beta).
\]
Whenever one term is hyperbolic, its exact equation gives
\[
\ip{H_{i,\ell}}{v_{\beta,\ell}} =\int_{\HH^n}H_{i,\ell}^{2^*-1}v_{\beta,\ell}\dv \longrightarrow0 \quad(v_{\beta,\ell}\ne H_{i,\ell}).
\]
The quadratic cross terms between two cutoff Euclidean profiles vanish. Applying the same estimate with \(\gamma=2^*\) therefore gives
\[
\begin{gathered}
\norm{v_\ell}^2 =\sum_\alpha\norm{v_{\alpha,\ell}}^2+o(1) =mS_{n,\lambda}^{n/2}+k\mathcal S_{n,2}^{n/2}+o(1),\\
\int_{\HH^n}v_\ell^{2^*}\dv =\sum_\alpha\int_{\HH^n}v_{\alpha,\ell}^{2^*}\dv+o(1), \qquad J_\lambda(v_\ell)\to E_{m,k}(\lambda).
\end{gathered}
\]

Finally, define
\[
u_\ell=\sum_{i=1}^m B^{\HH}_{z_{i,\ell},\lambda} +\sum_{j=1}^k B^{\HH}_{\mathrm{pr},\xi_j,\eps_\ell}.
\]
Then \(u_\ell\ge0\) and \(\norm{u_\ell-v_\ell}\to0\). The Lipschitz estimates above yield
\[
J_\lambda(u_\ell)\to E_{m,k}(\lambda), \qquad \Gamma_\lambda(u_\ell)\to0.
\]

The cutoff Euclidean terms have bounded energy and vanishing \(L^2\)-norm, hence weak limit zero in \(H^1_\lambda\). The hyperbolic terms also converge weakly to zero because their centers diverge. Thus \(u_\ell\weakto0\).

The chosen parameters satisfy \eqref{eq5.2}--\eqref{eq5.3}. In each extraction the selected term converges to its profile. All other terms vanish weakly: in a hyperbolic extraction this follows from center divergence and the vanishing \(L^2\)-norms of the Euclidean terms; in a Euclidean extraction it follows from \eqref{eq4.32}--\eqref{eq4.33} and the parameter dichotomy in the proof of Lemma~\ref{Lem4.7}, together with \eqref{eq5.2}. Thus this is a profile decomposition with \(\rho_\ell=0\).

The mass computation in the decomposition part also gives \(\|u_\ell\|_2^2\to mM_\lambda\). If another positive profile decomposition has numbers \((m',k')\), then \eqref{eq5.5} implies
\[
m'M_\lambda=mM_\lambda, \qquad m'S_{n,\lambda}^{n/2}+k'\mathcal S_{n,2}^{n/2} =mS_{n,\lambda}^{n/2}+k\mathcal S_{n,2}^{n/2}.
\]
Since \(M_\lambda>0\), these identities force \(m'=m\) and \(k'=k\), even when the same energy has other integer representations. This realizes every pair. Together with the decomposition assertion, it shows that the nonnegative Palais--Smale levels are exactly
\begin{equation}\label{eq5.8}
\{E_{m,k}(\lambda):(m,k)\in\mathbb N_0^2\}.
\end{equation}
\end{proof}

\section{Energy resonance and compactness}
\label{Sec6}

We classify the representations of the Palais--Smale energies \(E_{m,k}(\lambda)\) from Section~\ref{Sec5} and prove approximation by the corresponding configuration families. Prescribing the limiting mass determines the profile numbers even at resonant parameters. For one profile of each type, this yields Theorem~\ref{Thm1.3} from the local estimate in Theorem~\ref{Thm1.1}.

\subsection{Parameter dependence and energy resonance}
\label{Sec6.1}

Throughout this subsection, \(n\ge4\) is fixed and \(\lambda\in(\lamc,\lambdab)\). We use the energy ratio \(\mathfrak r_n\) from \eqref{eq1.8} and the endpoint constant \(S_{n,\mathrm{bot}}\) from Lemma~\ref{Lem5.5}. Let \(H^1_{\mathrm{rad}}(\HH^n;\RR)\) denote the real-valued \(H^1\) functions invariant under rotations about the origin.

\begin{theorem}\label{Thm6.1}
The map \(\lambda\mapsto B^{\HH}_{0,\lambda}\) belongs to \(C^1((\lamc,\lambdab);H^1_{\mathrm{rad}}(\HH^n;\RR))\), and
\begin{equation}\label{eq6.1}
\frac{\mathrm d}{\mathrm d\lambda} \frac{S_{n,\lambda}^{n/2}}n =-\frac12M_\lambda, \qquad \mathfrak r_n'(\lambda) =-\frac{nM_\lambda}{2\mathcal S_{n,2}^{n/2}}<0.
\end{equation}
The ratio \(\mathfrak r_n\) is a strictly decreasing \(C^1\) diffeomorphism from \((\lamc,\lambdab)\) onto
\begin{equation}\label{eq6.2}
\left( \left(\frac{S_{n,\mathrm{bot}}}{\mathcal S_{n,2}}\right)^{n/2}, 1\right).
\end{equation}
The parameters for which \(\mathfrak r_n(\lambda)\) is rational form a countable dense subset of \((\lamc,\lambdab)\).
\end{theorem}

\begin{proof}
We work in the real space \(H^1_{\mathrm{rad}}(\HH^n)\) with its usual fixed \(H^1\)-norm. Sobolev embedding and continuity of \(u\mapsto |u|^{2^*-2}\) from \(L^{2^*}\) to \(L^{n/2}\) show that \((\lambda,u)\mapsto\cR_\lambda(u)\), restricted to radial test functions, is \(C^1\) into the dual space, with nonlinear derivative \(h\mapsto(2^*-1)|u|^{2^*-2}h\).

At the centered solution, write
\[
\mathscr L_\lambda =P^{\HH}_\lambda -(2^*-1)(B^{\HH}_{0,\lambda})^{2^*-2}.
\]
The operator \(P^{\HH}_\lambda\) is an isomorphism from the radial space to its dual. Since \((B^{\HH}_{0,\lambda})^{2^*-2}\in L^{n/2}(\HH^n)\), multiplication by this function is compact between these spaces. This follows by \(L^{n/2}\) approximation and local Rellich compactness, as in the proof of Proposition~\ref{Prop4.3}. Thus \(\mathscr L_\lambda\) is Fredholm of index zero.

A radial function satisfying the linearized equation against radial test functions satisfies it against all test functions, by rotational averaging. The nondegeneracy theorem \cite[Theorem 1.1]{GangulySandeep2015} therefore gives
\[
\begin{aligned}
\ker\bigl(\mathscr L_\lambda|_{H^1_{\mathrm{rad}}}\bigr) &=H^1_{\mathrm{rad}}\cap \operatorname{span}\left\{ \left.\partial_{z,j}B^{\HH}_{z,\lambda}\right|_{z=0}: 1\le j\le n\right\}=\{0\}.
\end{aligned}
\]
Here each center derivative has zero rotational average. Hence the radial linearized operator is invertible.

The implicit function theorem gives a local \(C^1\) radial branch near \(B^{\HH}_{0,\lambda_0}\) for each \(\lambda_0\in(\lamc,\lambdab)\). Rotational averaging shows that its elements solve the full weak equation. Set \(u^-=\max\{-u,0\}\). Testing with the negative part and using the endpoint inequality gives
\[
S_{n,\mathrm{bot}}\|u^-\|_{2^*}^2 \le\norm{u^-}^2 =\|u^-\|_{2^*}^{2^*}.
\]
Consequently,
\[
u^-\ne0 \quad\Longrightarrow\quad \|u^-\|_{2^*}\ge S_{n,\mathrm{bot}}^{1/(2^*-2)}.
\]
This contradicts
\[
\|u^-\|_{2^*} \le\|u-B^{\HH}_{0,\lambda_0}\|_{2^*}\longrightarrow0
\]
after the neighborhood is reduced. The branch remains nonzero, so Lemma~\ref{Lem5.1} gives strict positivity. Uniqueness of the centered positive solution \cite{ManciniSandeep2008} identifies the local branches with \(B^{\HH}_{0,\lambda}\), proving the global \(C^1\) dependence on the open interval. Differentiating its equation gives
\[
\mathscr L_\lambda \bigl(\partial_\lambda B^{\HH}_{0,\lambda}\bigr) =B^{\HH}_{0,\lambda}.
\]

Since \(J_\lambda(B^{\HH}_{0,\lambda})=S_{n,\lambda}^{n/2}/n\), differentiation in the fixed \(H^1\) space yields
\[
\begin{aligned}
\frac{\mathrm d}{\mathrm d\lambda} J_\lambda(B^{\HH}_{0,\lambda}) &=-\frac12\|B^{\HH}_{0,\lambda}\|_2^2 +\bigl\langle\cR_\lambda(B^{\HH}_{0,\lambda}), \partial_\lambda B^{\HH}_{0,\lambda}\bigr\rangle=-\frac12M_\lambda.
\end{aligned}
\]
This proves \eqref{eq6.1}.

Lemma~\ref{Lem5.5} gives
\[
\lim_{\lambda\to\lamc}S_{n,\lambda}=\mathcal S_{n,2}, \qquad \lim_{\lambda\to\lambdab}S_{n,\lambda}=S_{n,\mathrm{bot}}, \qquad 0<S_{n,\mathrm{bot}}<\mathcal S_{n,2}.
\]
Together with \eqref{eq6.1}, these limits prove the asserted \(C^1\) diffeomorphism. Taking inverse images of the rational numbers in \eqref{eq6.2} proves countability and density.
\end{proof}

If \(\mathfrak r_n(\lambda_0)=p/q\) in lowest terms, \eqref{eq6.1} gives
\[
\frac{qS_{n,\lambda}^{n/2}-p\mathcal S_{n,2}^{n/2}}n =-\frac q2M_{\lambda_0}(\lambda-\lambda_0) +o(|\lambda-\lambda_0|).
\]
Thus these two energies cross with nonzero derivative. The following proposition classifies all equal-energy pairs.

\begin{proposition}\label{Prop6.2}
Fix \(\lambda\in(\lamc,\lambdab)\). For pairs in \(\mathbb N_0^2\),
\begin{equation}\label{eq6.3}
E_{m,k}(\lambda)=E_{m',k'}(\lambda) \quad\Longleftrightarrow\quad (m-m')\mathfrak r_n(\lambda)=k'-k.
\end{equation}
If \(\mathfrak r_n(\lambda)\) is irrational, every energy has at most one representation. If \(\mathfrak r_n(\lambda)=p/q\) in lowest terms, all representations of \(E_{m_0,k_0}(\lambda)\) are
\begin{equation}\label{eq6.4}
(m,k)=(m_0+qa,k_0-pa), \qquad a\in\mathbb Z, \qquad -\frac{m_0}q\le a\le\frac{k_0}p.
\end{equation}
The smallest energy admitting two distinct representations is
\begin{equation}\label{eq6.5}
\frac{p\mathcal S_{n,2}^{n/2}}n =\frac{qS_{n,\lambda}^{n/2}}n,
\end{equation}
and its only representations are \((q,0)\) and \((0,p)\). For the same rational ratio and \(N\in\mathbb N_0\), the representations of \(NS_{n,\lambda}^{n/2}/n\) are
\[
(m,k)=(N-qa,pa), \qquad a\in\mathbb N_0, \qquad 0\le a\le N/q.
\]
A second representation exists exactly when \(q\le N\), and a representation with both numbers positive exists exactly when \(q<N\). For fixed \(N\), only finitely many parameters admit a second representation of this pure hyperbolic level.
\end{proposition}

\begin{proof}
Dividing the energy identity by \(\mathcal S_{n,2}^{n/2}/n\) gives \eqref{eq6.3}. Irrationality forces \(m=m'\) and \(k=k'\). In the rational case,
\[
(m-m_0)p=(k_0-k)q \quad\Longrightarrow\quad m-m_0=qa,\qquad k_0-k=pa
\]
for some \(a\in\mathbb Z\), by coprimality. The bounds in \eqref{eq6.4} are exactly the conditions \(m,k\ge0\).

For distinct equal-energy pairs, interchange them if necessary so that \(m=m'+qa\) and \(k'=k+pa\), with \(a\ge1\). Their common energy satisfies
\[
\frac{nE_{m,k}(\lambda)}{\mathcal S_{n,2}^{n/2}} =m'\mathfrak r_n(\lambda)+k+pa\ge p.
\]
Equality forces \(m'=k=0\) and \(a=1\), and Theorem~\ref{Thm1.4} realizes both representations.

The pure hyperbolic case follows from \eqref{eq6.4} with \((m_0,k_0)=(N,0)\). A nonzero step exists exactly when \(q\le N\), and a step leaving both numbers positive exists exactly when \(q<N\). Since \(0<p<q\le N\), only finitely many reduced ratios can occur; Theorem~\ref{Thm6.1} gives at most one parameter for each.
\end{proof}

For fixed \(\lambda\), the set of energies in \eqref{eq5.8} is locally finite, because an upper bound on \(E_{m,k}(\lambda)\) bounds both nonnegative integers. At a rational ratio \(p/q\), distinct levels differ by at least \(\mathcal S_{n,2}^{n/2}/(nq)\). At an irrational ratio, positive gaps can be arbitrarily small as the levels increase. Indeed, the pigeonhole principle gives integers \(1\le a_L\le L\) and \(b_L\in\mathbb Z\) such that
\[
0<|a_L\mathfrak r_n(\lambda)-b_L|\le\frac1L.
\]
Irrationality implies \(a_L\to\infty\), and positivity of the ratio gives \(b_L>0\) for large \(L\). Hence
\[
0<|E_{a_L,0}(\lambda)-E_{0,b_L}(\lambda)| \le\frac{\mathcal S_{n,2}^{n/2}}{nL}\longrightarrow0.
\]

The resonant parameters are nevertheless finite when the common energy is bounded. More precisely, fix \(D>0\). If the first common level \eqref{eq6.5} is at most \(D\), then
\[
p\le\frac{nD}{\mathcal S_{n,2}^{n/2}}, \qquad q=\frac p{\mathfrak r_n(\lambda)} <p\left(\frac{\mathcal S_{n,2}}{S_{n,\mathrm{bot}}}\right)^{n/2}.
\]
Both integers range over finite sets, and each reduced ratio determines at most one parameter. Thus finiteness holds on the entire interval \((\lamc,\lambdab)\) under a fixed energy bound, whereas allowing all energy levels gives the countable dense set in Theorem~\ref{Thm6.1}.

\subsection{Compactness at fixed energy and mass}
\label{Sec6.2}

We now fix \(\lamc<\lambda<\lambdab\) and a sufficiently small \(\eps_0>0\). The preceding energy classification determines which configuration families may occur in Theorem~\ref{Thm1.4}. For \((m,k)\in\mathbb N_0^2\), set
\begin{equation}\label{eq6.6}
\begin{split}
\cM_{m,k}(\eps_0)=\biggl\{\sum_{i=1}^m B^{\HH}_{z_i,\lambda} +\sum_{j=1}^k B^{\HH}_{\mathrm{pr},\xi_j,\eps_j}:z_i,\xi_j\in\HH^n,\quad 0<\eps_j<\eps_0 \biggr\}.
\end{split}
\end{equation}
Empty sums are zero, so \(\cM_{0,0}(\eps_0)=\{0\}\), and \(\cM_{1,1}(\eps_0)=\cM_{\mix}(\eps_0)\). For \(c\ge0\), define the finite, possibly empty set
\[
\mathcal P_{n,\lambda}(c) =\{(m,k)\in\mathbb N_0^2:E_{m,k}(\lambda)=c\}.
\]
Finiteness follows from the integer bound used in Section~\ref{Sec6.1}. The families in \eqref{eq6.6} allow arbitrary centers and scales in \((0,\eps_0)\); the approximating parameters in a profile decomposition satisfy \eqref{eq5.2}--\eqref{eq5.3}.

\begin{theorem}\label{Thm6.3}
Fix \(c\ge0\). If \(\mathcal P_{n,\lambda}(c)=\varnothing\), there exists \(\Gamma_0>0\) such that no nonnegative \(u\in H^1_\lambda\) satisfies
\[
|J_\lambda(u)-c|+\Gamma_\lambda(u)<\Gamma_0.
\]
If \(\mathcal P_{n,\lambda}(c)\ne\varnothing\), every nonnegative Palais--Smale sequence at level \(c\) has, after passage to a subsequence, one pair \((m,k)\in\mathcal P_{n,\lambda}(c)\) such that
\[
\dist_{H^1_\lambda}(u_\ell,\cM_{m,k}(\eps_0))\longrightarrow0,
\]
with approximating parameters as in Theorem~\ref{Thm1.4}. Moreover, for every \(\kappa>0\) there is \(\Gamma_0>0\) such that every nonnegative \(u\in H^1_\lambda\) satisfying
\[
|J_\lambda(u)-c|+\Gamma_\lambda(u)<\Gamma_0
\]
obeys
\begin{equation}\label{eq6.7}
\dist_{H^1_\lambda}\left(u, \bigcup_{(m,k)\in\mathcal P_{n,\lambda}(c)} \cM_{m,k}(\eps_0)\right)<\kappa.
\end{equation}
\end{theorem}

\begin{proof}
The subsequential statement follows from Theorem~\ref{Thm1.4}, since all Euclidean scales are eventually less than \(\eps_0\). To prove the uniform assertions, suppose first that \eqref{eq6.7} fails for some \(\kappa>0\). There are nonnegative \(u_\ell\) satisfying
\[
|J_\lambda(u_\ell)-c|+\Gamma_\lambda(u_\ell)<\frac1\ell,
\]
whose distance from the stated union is at least \(\kappa\). The identity
\[
\frac1n\norm{u_\ell}^2 =J_\lambda(u_\ell) -\frac1{2^*}\langle\cR_\lambda(u_\ell),u_\ell\rangle \le c+1+\frac1{2^*\ell}\norm{u_\ell}
\]
gives boundedness. Theorem~\ref{Thm1.4} produces a pair \((m,k)\in\mathcal P_{n,\lambda}(c)\) and an approximation by \(\cM_{m,k}(\eps_0)\) with error tending to zero, a contradiction. The same argument rules out a sequence with the displayed energy and residual errors when \(\mathcal P_{n,\lambda}(c)\) is empty. Finally, \eqref{eq6.7} implies that every Palais--Smale sequence at level \(c\) satisfies
\[
\dist_{H^1_\lambda}\left(u_\ell, \bigcup_{(m,k)\in\mathcal P_{n,\lambda}(c)} \cM_{m,k}(\eps_0)\right)\longrightarrow0.
\]
\end{proof}

At the first level with two representations, the two decompositions are asymptotically separated.

\begin{corollary}\label{Cor6.4}
Suppose \(\mathfrak r_n(\lambda)=p/q\) in lowest terms. At the level \eqref{eq6.5}, every nonnegative Palais--Smale sequence approaches \(\cM_{q,0}(\eps_0)\cup\cM_{0,p}(\eps_0)\), and both alternatives occur. If \(H_\ell\) is a sum of \(q\) hyperbolic profiles satisfying \eqref{eq5.3}, and \(V_\ell\) is a sum of \(p\) projected Euclidean profiles satisfying \eqref{eq5.2}, then
\begin{equation}\label{eq6.8}
\norm{H_\ell-V_\ell}^2\longrightarrow qS_{n,\lambda}^{n/2}+p\mathcal S_{n,2}^{n/2} =2p\mathcal S_{n,2}^{n/2}.
\end{equation}
There are Palais--Smale sequences at this level for which no fixed pair of profile numbers gives a decomposition along the full sequence.
\end{corollary}

\begin{proof}
Proposition~\ref{Prop6.2} and Theorem~\ref{Thm6.3} give the two alternatives, and Theorem~\ref{Thm1.4} realizes both. Write
\[
H_\ell=\sum_{i=1}^q H_{i,\ell}, \qquad H_{i,\ell}=B^{\HH}_{z_{i,\ell},\lambda}, \qquad V_\ell=\sum_{j=1}^p B_{j,\ell}, \qquad B_{j,\ell} =B^{\HH}_{\mathrm{pr},\xi_{j,\ell},\eps_{j,\ell}}.
\]
The hyperbolic separation and compactly supported approximation give
\[
\ip{H_{i,\ell}}{H_{h,\ell}} =\int_{\HH^n}H_{i,\ell}^{2^*-1}H_{h,\ell}\dv \longrightarrow0 \qquad(i\ne h).
\]
For two Euclidean terms, use the rescaling \(\mathscr S\) defined before Proposition~\ref{Prop4.3}. Lemma~\ref{Lem2.5} and the parameter dichotomy in the proof of Lemma~\ref{Lem4.7}, together with \eqref{eq5.2}, yield
\[
\begin{aligned}
\mathscr S_{\xi_{j,\ell},\eps_{j,\ell}}B_{j,\ell} &\longrightarrow B^{\RR}_1 &&\text{strongly in }\dot H^1(\RR^n),\\
\mathscr S_{\xi_{j,\ell},\eps_{j,\ell}}B_{h,\ell} &\weakto0 &&\text{in }\dot H^1(\RR^n),\qquad j\ne h.
\end{aligned}
\]
Conformal covariance expresses their energy pairing as
\[
\begin{aligned}
\ip{B_{j,\ell}}{B_{h,\ell}} =\int_{\RR^n} \nabla\bigl(\mathscr S_{\xi_{j,\ell},\eps_{j,\ell}}B_{j,\ell}\bigr) \cdot\nabla\bigl(\mathscr S_{\xi_{j,\ell},\eps_{j,\ell}}B_{h,\ell}\bigr)\,dy-\mu\int_{\HH^n}B_{j,\ell}B_{h,\ell}\dv \longrightarrow0.
\end{aligned}
\]
The first term tends to zero by strong--weak convergence, and the second satisfies
\[
\left|\mu\int_{\HH^n}B_{j,\ell}B_{h,\ell}\dv\right| \le\mu\|B_{j,\ell}\|_2\|B_{h,\ell}\|_2\longrightarrow0.
\]
For the mixed pairing, Theorem~\ref{Thm3.3} gives, uniformly in the relative centers,
\[
\left|\ip{H_{i,\ell}}{B_{j,\ell}}\right| =Q\bigl(\eps_{j,\ell},d_g(z_{i,\ell},\xi_{j,\ell})\bigr) \le C\eps_{j,\ell}^{(n-2)/2}\longrightarrow0.
\]
The single-profile norm limits now imply
\[
\norm{H_\ell}^2\to qS_{n,\lambda}^{n/2}, \qquad \norm{V_\ell}^2\to p\mathcal S_{n,2}^{n/2}, \qquad \ip{H_\ell}{V_\ell}\to0.
\]
Expanding the squared norm proves \eqref{eq6.8}.

Finally, interleave the sequences realizing \((q,0)\) and \((0,p)\). The resulting sequence is Palais--Smale at their common energy, whereas
\[
\|u_{2\ell}\|_2^2\longrightarrow qM_\lambda, \qquad \|u_{2\ell+1}\|_2^2\longrightarrow0.
\]
A decomposition with a fixed pair \((m,k)\) along the full sequence would imply \(\|u_\ell\|_2^2\to mM_\lambda\) by \eqref{eq5.5}, which is impossible.
\end{proof}

By \eqref{eq5.5}, the limiting mass determines \(m\), and the energy then determines \(k\). The next result uses fixed energy and mass tolerances, independent of the residual.

\begin{corollary}\label{Cor6.5}
Fix \((m_0,k_0)\in\mathbb N_0^2\) and constants
\begin{equation*}
0<\eta_M<M_\lambda, \qquad 0<\eta_E<\frac{\mathcal S_{n,2}^{n/2}}n.
\end{equation*}
For every \(\kappa>0\), there is \(\Gamma_0>0\) such that every nonnegative \(u\in H^1_\lambda\) satisfying
\begin{equation*}
\begin{gathered}
|J_\lambda(u)-E_{m_0,k_0}(\lambda)|<\eta_E, \qquad \bigl|\|u\|_2^2-m_0M_\lambda\bigr|<\eta_M,\\
\Gamma_\lambda(u)<\Gamma_0
\end{gathered}
\end{equation*}
obeys
\[
\dist_{H^1_\lambda}(u,\cM_{m_0,k_0}(\eps_0))<\kappa.
\]
This holds for every \(\lambda\in(\lamc,\lambdab)\), including the resonant parameters.
\end{corollary}

\begin{proof}
Otherwise, for some \(\kappa>0\), there are nonnegative \(u_\ell\) satisfying the fixed energy and mass constraints, with
\[
\Gamma_\lambda(u_\ell)\to0, \qquad \dist_{H^1_\lambda}(u_\ell,\cM_{m_0,k_0}(\eps_0))\ge\kappa.
\]
After passage to a subsequence, the bounded energies converge to a finite level \(c\). Theorem~\ref{Thm1.4} gives profile numbers \((m,k)\), energy \(E_{m,k}(\lambda)\), and limiting mass \(mM_\lambda\). Passing to the limit in the mass constraint gives
\[
|m-m_0|M_\lambda\le\eta_M<M_\lambda,
\]
and hence \(m=m_0\). The energy constraint then gives
\[
|k-k_0|\frac{\mathcal S_{n,2}^{n/2}}n \le\eta_E<\frac{\mathcal S_{n,2}^{n/2}}n,
\]
so \(k=k_0\). The approximation in \eqref{eq5.4} contradicts the distance lower bound.
\end{proof}

\begin{proof}[Proof of Theorem~\ref{Thm1.3}]
Let \(\eps_0,\eta_0\) be the constants from Theorem~\ref{Thm1.1}, with \(\eps_0\) chosen as before Theorem~\ref{Thm1.2}. Apply Corollary~\ref{Cor6.5} with \((m_0,k_0)=(1,1)\), comparison family \(\cM_{1,1}(\eps_0/2)\), and approximation error \(\eta_0/2\). For some \(\Gamma_0\in(0,1)\), every function satisfying \eqref{eq1.6} then obeys
\[
\dist_{H^1_\lambda}\bigl(u,\cM_{1,1}(\eps_0/2)\bigr)<\eta_0/2.
\]
By the definition of the infimum, there are \(z_*,\xi_*\in\HH^n\) and \(0<\eps_*<\eps_0/2\) for which \eqref{eq1.2} holds. Theorem~\ref{Thm1.1} gives the asserted conclusions with a constant depending only on \(n,\lambda\).

In dimension six, the positive functions from Proposition~\ref{Prop4.9} satisfy
\[
J_\lambda(u_\eps)\longrightarrow E_{\mix}(\lambda), \qquad \|u_\eps\|_2^2\longrightarrow M_\lambda,
\]
as proved in Corollary~\ref{Cor5.7}. They eventually satisfy every fixed pair of constraints in \eqref{eq1.5}, and their distance and residual have the sharp logarithmic relation established in Proposition~\ref{Prop4.9}. This proves optimality under the stated energy and mass constraints.
\end{proof}

For a fixed number of hyperbolic profiles, the exceptional set in the critical case of \cite[Corollary 1.1]{BGKMPartI} can be replaced by a finite set.

\begin{corollary}\label{Cor6.6}
Let \(n=4,5\), fix \(N\in\mathbb N\), and put
\[
\Lambda_N=\left\{\lambda\in(\lamc,\lambdab): \mathfrak r_n(\lambda)=\frac pq,\quad 1\le p<q\le N,\quad \gcd(p,q)=1\right\}.
\]
For every \(\lambda\in(\lamc,\lambdab)\setminus\Lambda_N\), there are \(\delta,C>0\) such that every nonnegative \(u\in H^1_\lambda\) satisfying
\[
\bigl|\norm{u}^2-NS_{n,\lambda}^{n/2}\bigr| <\delta S_{n,\lambda}^{n/2}
\]
obeys
\[
\dist_{H^1_\lambda}\bigl(u,\cM_{N,0}(\eps_0)\bigr) \le C\Gamma_\lambda(u).
\]
The finite set \(\Lambda_N\) is necessary for this conclusion under the norm constraint alone. For every \(\lambda\in(\lamc,\lambdab)\) and every \(0<\eta_M<M_\lambda\), there exist \(\delta,C>0\) such that the same estimate holds for every nonnegative \(u\in H^1_\lambda\) satisfying the norm window and
\[
\bigl|\|u\|_2^2-NM_\lambda\bigr|<\eta_M.
\]
\end{corollary}

\begin{proof}
First, \eqref{eq5.1} converts a bounded norm constraint into the corresponding energy constraint as the residual tends to zero. Proposition~\ref{Prop6.2} and Theorem~\ref{Thm1.4} then give approximation by \(N\) hyperbolic profiles with pairwise divergent centers when \(\lambda\notin\Lambda_N\). A contradiction argument supplies a fixed norm window and a residual threshold for which the hypotheses of \cite[Theorem 1.1]{BGKMPartI} hold. Above that threshold, the norm bound gives the estimate after increasing \(C\). With the additional mass constraint, Corollary~\ref{Cor6.5} gives the same conclusion for every \(\lambda\); the profile decomposition again supplies the separation required by the local theorem.

To see necessity, suppose \(\mathfrak r_n(\lambda)=p/q\) in lowest terms and \(q\le N\). Theorem~\ref{Thm1.4} realizes \((N-q,p)\) at \(E_{N,0}(\lambda)\), with residual tending to zero and mass tending to \((N-q)M_\lambda\). Every sum of \(N\) positive hyperbolic profiles has mass at least \(NM_\lambda\), so its distance from this sequence stays positive by the spectral gap. Thus the fixed-level exclusion is exactly \(\Lambda_N\).
\end{proof}

On the other hand, \(\bigcup_{N\ge1}\Lambda_N\) is the full rational-ratio set of Theorem~\ref{Thm6.1}; a single exclusion valid at all pure hyperbolic levels must still contain this countable dense set.

For arbitrary profile numbers, Corollary~\ref{Cor6.5} gives qualitative approximation. Sharp residual estimates for general configurations containing several profiles of both types require further quantitative bounds for the Euclidean--Euclidean and hyperbolic--hyperbolic interactions and for the corresponding linearized operator.

\section*{Acknowledgments}
J. Wei and Z. Yan were supported by National R\&D Program of China (2022YFA1005602), and Hong Kong General Research Fund ``New frontiers in singular limits of nonlinear partial differential equation''. Z. Yang was supported by the National Natural Science Foundation of China (12301145, 12561020, 12261107) and the Yunnan Fundamental Research Projects (202401AU070123, 202601AT070048). 
The authors acknowledge the use of ChatGPT (OpenAI) for exploratory discussions and language refinement during the preparation of this work. All mathematical statements, proofs, and references were independently checked by the authors, who take full responsibility for the content of the paper.

\medskip
{\bf Author Contributions:} All authors contributed equally to the writing and preparation of the manuscript.

\medskip
{\bf Data availability:} Data sharing is not applicable to this article as no new data were created or analyzed in this study.

\medskip
{\bf Conflict of interest:} The authors declare that they have no conflict of interest.

\end{document}